\documentclass[leqno, 11pt]{amsart}
\usepackage{hyperref}
\usepackage{amsmath, amssymb, amsthm, cite, graphicx, float, mathrsfs, commath, dsfont, bbm, bm, mathtools, thmtools, enumitem}
\usepackage{cleveref}
\usepackage{algorithmicx, algorithm, algpseudocode}
\algrenewcommand\algorithmicrequire{\textbf{Input:}}
\algrenewcommand\algorithmicensure{\textbf{Output:}}
\usepackage{tikz, tikz-cd}
\usetikzlibrary{patterns,snakes}

\usepackage{xcolor}
\hypersetup{
    colorlinks,
    linkcolor={red!50!black},
    citecolor={blue!50!black},
    urlcolor={blue!80!black}
}
\usepackage{optidef}
\usepackage[mathscr]{eucal}
\usepackage[sc]{mathpazo}
\renewcommand{\bar}{\overline}
\renewcommand{\hat}{\widehat}

\newcommand{\bB}{\mathbb{B}}

\newcommand{\bR}{\mathbb{R}}
\newcommand{\bN}{\mathbb{N}}

\newcommand{\R}{\bR}

\newcommand{\eR}{{\overline\bR}}

\newcommand{\cA}{\mathcal{A}}
\newcommand{\cB}{\mathcal{B}}
\newcommand{\cP}{\mathcal{P}}
\newcommand{\cQ}{\mathcal{Q}}

\newcommand{\cD}{{\mathcal{D}}}

\newcommand{\cH}{{\mathcal{H}}}

\newcommand{\cK}{{\mathcal{K}}}

\newcommand{\cG}{\mathcal{G}}

\newcommand{\cN}{\mathcal{N}}

\newcommand{\cC}{\mathcal{C}}
\newcommand{\cU}{\mathcal{U}}

\newcommand{\cM}{\mathcal{M}}

\DeclareMathOperator{\epi}{epi}
\DeclareMathOperator*{\argmin}{arg\,min}

\DeclareMathOperator{\intr}{int}
\DeclareMathOperator{\bdry}{bdry}

\DeclareMathOperator{\cl}{cl}

\newcommand{\ip}[2]{\left\langle #1,\, #2\right\rangle}

\newcommand{\dom}[1]{\mathrm{dom}\left(#1\right)}

\newcommand{\dist}[2]{\mbox{dist}\left(#1\,\left|\, #2\right.\right)}

\newcommand{\bset}[2]{\left\{#1\,\left|\, #2\right.\right\}}

\newcommand{\ran}[1]{\mathrm{Ran}\left(#1\right)}
\newcommand{\Ran}[1]{\mathrm{Ran}\left(#1\right)}

\newcommand{\Null}[1]{\mathrm{Null}\left(#1\right)}
\newcommand{\diag}{\mathrm{diag}}
\newcommand{\ncone}[2]{N\left(#1\, |\, #2\right)}
\newcommand{\tcone}[2]{T\left(#1\, |\, #2\right)}

\newcommand{\ri}[1]{\mathrm{ri}\left(#1\right)}
\newcommand{\parr}[1]{\mathrm{par}\left(#1\right)}

\DeclareMathOperator{\aff}{aff}
\DeclareMathOperator{\gph}{gph}

\newcommand{\sd}{\partial}

\newcommand{\hsd}{{\hat\partial}}

\newcommand{\hmu}{{\hat \mu}}

\newcommand{\eps}{\epsilon}

\newcommand{\bc}{\bar c}
\newcommand{\bd}{{\bar d}}
\newcommand{\by}{\bar y}

\newcommand{\bk}{\bar k}
\newcommand{\bx}{\bar x}

\newcommand{\bv}{\bar v}

\newcommand{\bmu}{{\bar\mu}}

\newcommand{\hx}{{\hat x}}

\newcommand{\hc}{{\hat c}}

\newcommand{\hy}{{\hat y}}

\usepackage[margin=1in, top=.8in, left=.8in]{geometry}
\newtheorem{theorem}{Theorem}[section]
\newtheorem{corollary}{Corollary}[section]
\newtheorem{definition}{Definition}[section]
\newtheorem{lemma}{Lemma}[section]
\newtheorem{proposition}{Proposition}[section]
\newtheorem{example}{Example}
\theoremstyle{definition}
\newtheorem{remark}{Remark}
\newtheorem{assumptions}{Assumption}

\algnewcommand\algorithmicinput{\textbf{Initialize:}}
\algnewcommand\Initialize{\item[\algorithmicinput]}

\algnewcommand\algorithmicstepk{\textbf{Step }k\textbf{:}}
\algnewcommand\Step{\item[\algorithmicstepk]}

\newcommand{\calR}{\mathcal{R}}

\usepackage[parfill]{parskip}

\begin{document}
\title[PLQ Convex-Composite]{Strong Metric (Sub)regularity of KKT Mappings for Piecewise Linear-Quadratic Convex-Composite Optimization}
\author{J.~V.~Burke}
\thanks{Department of Mathematics, University of Washington, Seattle, WA. \texttt{\{jvburke,aengle2\}@uw.edu}. Supported in part by the U.S. National Science Foundation grant DMS-1514559.}
\author{A.~Engle}
\keywords{Convex-composite optimization, generalized equations, Newton's method, quasi-Newton methods, partial smoothness and active manifold identification, piecewise linear-quadratic, strong metric subregularity, strong metric regularity}
%\subjclass{asdf}
\begin{abstract}
\noindent 
This work concerns the local convergence theory of Newton and quasi-Newton methods for \emph{convex-composite} optimization: minimize 
$f(x):=h(c(x))$, where $h$ is an infinite-valued proper convex function and $c$ is $\cC^2$-smooth. We focus on the case where $h$ is infinite-valued piecewise linear-quadratic and convex. Such problems include nonlinear programming, mini-max optimization, 
estimation of nonlinear dynamics with non-Gaussian noise as well as many modern approaches to large-scale data analysis and machine learning. Our approach embeds the optimality conditions for convex-composite optimization problems into a generalized equation. We establish conditions for strong metric subregularity and strong metric regularity of the corresponding set-valued mappings. This allows us to extend classical convergence of Newton and quasi-Newton methods to the broader class of non-finite valued piecewise linear-\emph{quadratic} convex-composite optimization problems. In particular we establish local quadratic convergence of the Newton method under conditions that parallel those in nonlinear programming when $h$ is non-finite valued piecewise linear. 
%This work offers a generalization of the canonical convex-composite problem wherein a finite-valued piecewise linear convex function is precomposed with a smooth function. That problem class includes exact penalization for nonlinear programming. Our study focuses on extending classical convergence of Newton and quasi-Newton methods to the broader class of non-finite valued piecewise linear-\emph{quadratic} convex-composite optimization problems. In particular we employ modern active manifold and variational techniques that allow us to establish conditions for the strong metric (sub)regularity of the underlying KKT mappings.
%\smallskip
%\noindent\keywordsname.\,list of keywords
\end{abstract}
\maketitle % abstract goes before maketitle

\section{Introduction}
This work concerns local convergence theory of Newton and quasi-Newton methods for the solution of the \emph{convex-composite} problem:
\begin{mini}
	{x\in\R^n}{f(x):=h(c(x)),}{\tag{$\mathbf{P}$}}{\label{theprogram}}
\end{mini}
where $h:\R^m\to\R\cup\set{+\infty}$ is piecewise linear-quadratic (PLQ) and convex, and $c:\R^n\to\R^m$ is $\cC^2$-smooth. When $h=\frac{1}{2}\norm{\cdot}^2$, \ref{theprogram} is the classical
%the problem \ref{theprogram} dates back to the Gauss-Newton algorithm for the 
nonlinear least-squares problem.
Numerous other problems fall within this class including nonlinear programming (NLP), mini-max optimization, 
estimation of nonlinear dynamics with non-Gaussian noise
as well as many modern approaches to large-scale data analysis and machine learning \cite{aravkin2013sparse,aravkin2014optimization, davis2017nonsmooth}.
Convex-composite optimization has a long history with 
% and other hypotheses on the function $h$ that define \ref{theprogram} were 
investigations in the 1970s \cite{powell1978algorithms, powell1978fast}, 1980s \cite{burke1985descent,burke1987second,kawasaki1988second,rockafellar1985extensions,rockafellar1989second, wright1987local, yuan1985superlinear}, and 1990s \cite{burke1995gauss, burke1992optimality, deng1996uniqueness, rockafellar_wets_1998}, where much of the emphasis was on a calculus for compositions and its relationship to nonlinear programming (NLP) and exact
penalization \cite{fletcher2013practical}. Recently, there has been a resurgence of interest in local \cite{drusvyatskiy2018error, 2017arXiv170502356D} and global \cite{cui2018composite, davis2018stochastic, drusvyatskiy2018error,drusvyatskiyefficiency, duchi2017stochastic, lewis2016proximal} algorithms for this
class of problems especially with respect to establishing the iteration complexity of first-order methods for \ref{theprogram}. Much of this work has focused on the case where the
function $h$ is finite-valued. %\cite{drusvyatskiyefficiency,2017arXiv170502356D,davis2017nonsmooth}. %Much of this newer work concerns establishing iteration complexity bounds.

These, and almost all other methods for solving \ref{theprogram}, use a direction-finding subproblem similar to
\begin{mini}
	{x\in\R^n}{h(c(x^k)+\nabla c(x^k)[x-x^k]) + \frac{1}{2}[x-x^k]^\top H_k[x-x^k],}{\tag{$\mathbf{P}_k$}}{\label{introsubprob}}
\end{mini}
where $H_k$ is the Hessian of a Lagrangian for \ref{theprogram} \cite{burke1987second}. When the Hessian $H_k$ is used in the subproblems, the method corresponds to a Newton method \eqref{eq:geqnewton}, and when $H_k$ is approximated by a matrix $B_k$, it corresponds to a quasi-Newton method \eqref{eq:geqqn}. In either case, the subproblems \ref{introsubprob} may or may not be convex depending on whether $H_k,\,B_k\succeq0$. In the context of the broader class of prox-regular $h$, Lewis and Wright \cite{lewis2016proximal} take $B_k=\mu_kI$ at each iteration, thereby guaranteeing existence and uniqueness of the ``proximal step" and a global descent algorithm. Instead, our focus is on developing methods possessing fast local rates of convergence by taking advantage of second-order information together with the convex geometry of $\dom{h}$ developed by Rockafellar \cite{rockafellar1989second}.
%TODO Dima and Jim's results here.

When $h$ is assumed to be a finite-valued piecewise linear convex function, Womersley \cite{womersley1985local} established second-order rates of convergence for these algorithms under conditions comparable to those used in NLP, i.e., linear independence of the active constraint gradients, strict complementarity, and strong second-order sufficiency. Notwithstanding this correspondence to NLP, the method of proof differs significantly from the standard methodology to establishing such results in the NLP case developed by Robinson \cite{Robinson1972, Robinson1974}. Notably, in the case of NLP, the function $h$ is piecewise linear but not finite-valued. In subsequent work, Robinson  \cite{robinson1980strongly} introduced the revolutionary idea of \emph{generalized equations}, whose variational properties can be used to establish local rates of convergence for Newton's method for NLP. By employing the techniques of generalized equations, Cibulka et.~al.~\cite{cibulka2016strong} recently connected classical second-order necessary and sufficient conditions for a local minimizer of \ref{theprogram}
with strong metric subregularity (see \Cref{def:subreg}) of the underlying KKT mapping when $h$ is piecewise linear convex but not necessarily finite-valued. However, their analysis relies heavily on the fact that $h$ is piecewise linear. And so, the old question of what conditions imply local quadratic convergence when $h$ is not piecewise linear remains open. However, their technique created the possibility of an extension to the case where $h$ is a member of the PLQ class. This extension is our goal. It is hoped that the methods and techniques developed in this paper provide insight into how to extend these results beyond the PLQ class.

As noted above, we couch the analysis in the context \emph{Newton's method} for generalized equations. 
The first-order necessary conditions of a local minimum of \ref{theprogram} are encoded through a generalized equation of the form $g(x,y)+G(x,y)\ni0$, where $g:\R^{n+m}\to\R^{n+m}$ is a $\cC^1$-smooth function, $G:\R^{n+m}\rightrightarrows\R^{n+m}$ is a set-valued mapping, $(x,y)$ represents a primal-dual pair, and the function $\nabla g(x,y)$ is a KKT matrix for \ref{theprogram} (see \Cref{def:g+G}). Newton's method \eqref{eq:geqnewton} for solving this generalized equation corresponds to solving the optimality conditions for \ref{introsubprob}. The Newton iterate at $(x^k,y^k)$ is obtained by solving the following linearized generalized equation: 
\begin{equation}\label{eq:introNewton}
\text{Find }(x^{k+1},y^{k+1})\text{ such that }g(x^k,y^k)+\nabla g(x^k,y^k)\begin{pmatrix}
x^{k+1}-x^k\\
y^{k+1}-y^k
\end{pmatrix} + G(x^{k+1},y^{k+1})\ni 0.
\end{equation}
The details of this derivation appear in \Cref{sec:cvxcomp}.

The goal of this paper is to establish local convergence rates for algorithms based on iteratively solving \ref{introsubprob} in the case where $h$ is a PLQ convex function. We do this by augmenting the strategy of Cibulka et.~al.~\cite{cibulka2016strong} with additional innovations by Lewis \cite{lewis2002active} and Rockafellar \cite{rockafellar1989second}. In particular, we are able to establish conditions under which these algorithms are locally quadratically convergent. The first phase of our analysis involves extensive application of the first- and second-order PLQ calculus \cite{ rockafellar1989second,rockafellar_wets_1998} to establish conditions under which the underlying generalized equation is strongly metrically subregular. This allows us to establish sufficient conditions for the superlinear convergence of quasi-Newton methods for algorithms whose direction finding subproblems are based on \ref{introsubprob}. The second phase of our analysis  employs the technique of partly smooth functions in the sense of \cite{hare2004identifying,lewis2002active}
%, stemming from \cite{fletcher2013practical, womersley1985local, osborne2001simplicial, overton1994towards, burke1988identification}, 
to establish conditions under which a local approximation to the underlying generalized equation is strongly metrically regular (see \Cref{def:mr}). This allows us to give conditions for the local quadratic convergence of the Newton method based on \ref{introsubprob}.

We also note that recent work by Drusvyatskiy and Lewis \cite{drusvyatskiy2018error} considers similar types of results for convex-composite optimization problems of the form $\varphi(x)=h(c(x))+g(x)$, where $h$ is finite-valued and $L-$Lipschitz, $\nabla c$ is $\beta$-Lipschitz, and $g$ is closed, proper, convex, but infinite-valued. One of their goals is to understand the convergence of prox-linear type methods through either the subregularity \cite[Theorems 5.10 and 5.11]{drusvyatskiy2018error} or strong regularity \cite[Theorem 6.2]{drusvyatskiy2018error} of $\sd\varphi$ at \emph{stable} strong minima or sharp minima of $\varphi$ \cite[Theorems 7.1 and 7.2]{drusvyatskiy2018error}.

When $h$ is only assumed to be finite-valued convex and $g$ is zero, the first result on the local quadratic convergence for convex-composite problems was that of Burke and Ferris \cite{burke1995gauss}. In that work, the authors established a constraint qualification for the inclusion $c(\bx)\in\argmin h$ that ensures the local quadratic convergence of constrained Gauss-Newton methods. In \cite{burke1995gauss}, the authors assumed $\argmin h$ was a set of \emph{weak sharp minima} \cite{burke1993weak}. However, it was observed by Li and Wang \cite{li2002convergence} that the sharpness hypothesis was not required. Rather, a local quadratic growth condition \cite[Theorem 2]{li2002convergence} was sufficient for the proof  techniques in \cite{burke1995gauss} to succeed. The authors continued research \cite{li2007majorizing} in relaxations of the constraint qualification on $c(\bx)\in\argmin h $ and studied proximal methods \cite{hu2016convergence} for their convergence.

Our focus on the PLQ class is motivated by the great variety of modern problems in 
%robust statistics as well as 
data analysis, estimation of dynamical systems, inverse problems, and machine learning that are posed within this class. 
%In the data-driven age, the flexibility and scope of \ref{theprogram} led researchers to re-examine the study of local \cite{2017arXiv170502356D, davis2017nonsmooth} and global \cite{drusvyatskiyefficiency,lewis2016proximal} algorithms for convex-composite minimization through a modern lens. 
The key to the success of the convex-composite structure is that it separates the data associated to the problem, the function $c$, from the model within which we wish to explore the data, the function $h$. Consequently, the broader the class of functions $h$ available, the greater the variety of ways within which we can explore underlying extremal properties of the input function $c$, e.g., sparsity, robustness, network structure, dynamics, influence of hyperparameters, etc.  Importantly, we have learned that features of the data can be more readily extracted by imposing nonsmoothness in the function $h$.

The roadmap of the paper is as follows. \Cref{sec:notation} collects tools from convex and variational analysis used throughout the paper. \Cref{sec:cvxcomp} formally presents the convex-composite problem class. We take advantage of the structure of the problem class to rewrite the general first-order optimality conditions for proper functions in the presence of various constraint qualifications used in this work. We also present the generalized equation \eqref{kktS} associated with the first-order optimality conditions for \ref{theprogram}. \Cref{sec:geometry} discusses the convex geometry and differential theory of piecewise linear-quadratic functions collected in \cite{rockafellar_wets_1998}. The second-order theory of \cite{rockafellar_wets_1998} allows us to rewrite the general second-order necessary and sufficient conditions for a local minimum of \ref{theprogram}. We extract a crucial result from \cite{rockafellar_wets_1998} that highlights natural candidates for manifolds of partial smoothness \cite{lewis2002active} inherent to the function $h$. \Cref{sec:subreg} extends the result \cite[Theorem 7.1]{cibulka2016strong} relating the strong metric subregularity of \eqref{kktS} to the second-order sufficient conditions of local minima and ends with a convergence study of quasi-Newton methods for \ref{theprogram}. \Cref{sec:partialsmoothness} establishes conditions for the partly smooth structure of PLQ convex functions and sets the stage for  \Cref{sec:metricreg}, where we analyze the local quadratic convergence of Newton's method as in \cite{dontchev2014implicit}.
%
%\red{\begin{algorithm}[H]
%	\label{alg}
%	\caption{Newton Method}
%	\begin{algorithmic}[1]
%		\Initialize $(x^0,y^0)$ near $(\bx,\by)$ and $H^0=\nabla^2(y^0c)(x^0)$.
%		\Step $(k\geq0)$
%		\State Select $d^k\in\argmin_d \set{h(c(x^k) + \nabla c(x^k)d) + \frac{1}{2}d^\top H^kd}$;
%		\State Select $y^k\in \sd h(c(x^k)+\nabla c(x^k)d^k)$ satisfying $H^kd^k + \nabla c(x^k)^\top y^{k+1} = 0$;
%		\State $x^k \gets x^k + d^k$;
%		\State $H^k\gets \nabla^2(y^kc)(x^k)$;
%	\end{algorithmic}
%\end{algorithm}}

%
% %However, there are no known local second-order rates for the case in which $h$ is a general closed, proper, convex function. 

\section{Notation}
\label{sec:notation}
These sections summarize the relevant notation and tools of convex and variational analysis used in this work. Unless otherwise stated, we follow the notation in \cite{lewis2002active, rockafellar_wets_1998, dontchev2014implicit}.
\subsection{Preliminaries} 
We work in $(\R^n, \ip{\cdot}{\cdot})$ with the standard inner product $\ip{x}{y}=x^\top y = \sum_{i=1}^n x_iy_i$ and $\norm{x}^2=x^\top x$. Throughout, we switch between the notations $\ip{x}{y}$ and $x^\top y$ for clarity considerations. Let $\bB:=\bset{x\in\R^n}{\norm{x}\le1}$ be the closed unit ball.
For $A\in \R^{m\times n}$, its \emph{range}, \emph{null space}, and \emph{transpose} are $\Ran{A}, \Null{A}, A^\top$ respectively, and for a finite collection of mappings $\set{A_k}_{k\in J}$ with index set $J$, let $\diag{A_k}$ denote the block diagonal matrix with $k$th block $A_k$. Let $e_j\in\R^\ell$ denote the standard unit coordinate vector. %We study two modes of convergence of sequences generated by various Newton-like algorithms. Let $\set{x^k}_{k\in\bN}\subset\R^n$ with $x^k\to\bx$. When $x^k\ne\bx$, we say $x^k\to \bx$
%\emph{superlinearly} when \(\lim_{k\to\infty} \frac{\norm{x^{k+1}-\bx}}{\norm{x^k-\bx}} = 0,\) and \emph{quadratically} when \(\limsup_{k\to\infty} \frac{\norm{x^{k+1}-\bx}}{\norm{x^k-\bx}^2}<\infty.\)
\subsection{Convex Analysis}
A set $C\subset \R^m$ is \emph{locally closed} at a point $\bc$, not necessarily in $C$, if there exists a closed neighborhood $V$ of $\bc$ such that $C\cap V$ is closed. Any closed set is locally closed at all of its points, and the closure and interior of $C$ is denoted by $\cl{C}$ and $\intr{C}$, respectively.

For a closed convex set $C\subset\R^m$, let $\aff{C}$ denote the \emph{affine hull} of $C$ and $\parr{C}$ the \emph{subspace parallel} to $C$. Then, for any $c\in C$,
\(
\parr{C} := \aff{C} - c = \R(C-C),
\)
where we employ \emph{Minkowski set algebra} for addition of sets: for sets $C_1,C_2\subset\R^m$ and $t\in \R$, define $C+C':=\bset{c+c'}{c\in C,\ c'\in C'}$ and $\Lambda C := \bset{\lambda c}{\lambda\in\Lambda,\ c\in C}$. When $C=\set{c}$, we omit the set braces and write $c+C'$. The \emph{relative interior} of $C$ is given by
\(
\ri{C} = \bset{x \in \aff{C}}{\exists \,(\epsilon>0)\ (x+\epsilon\bB)\cap\aff{C}\subset C}.
\)
\subsection{Variational Analysis}
The functions in this paper take values in the extended reals $\eR:=\R\cup\set{\pm\infty}$. For $f:\R^n\to\eR$, the \emph{domain} of $f$ is $	\dom{f} := \bset{x\in\R^n}{f(x)<\infty}$,
and the \emph{epigraph} of $f$ is $\epi{f} := \bset{(x,\alpha)\in \R^n\times\R}{f(x)\le\alpha}$.
\\
We say $f$ is \emph{closed} if $\epi{f}$ is a closed subset of $\R^{n+1}$, $f$ is \emph{proper} if $\dom{f}\neq\emptyset$ and $f(x)>-\infty$ for all $x\in\R^n$, and $f$ is \emph{convex} if $\epi{f}$ is a convex subset of $\R^{n+1}$. 

Suppose $f:\R^n\to\eR$ is finite at $\bx$ and $w,v\in\R^n$. The \emph{subderivative} $\dif f(\bar{x}):\R^n\to\eR$ and \emph{one-sided directional derivative} $f'(\bx;\cdot)$ at $\bx$ for $w$ are
\begin{align*}
\dif f(\bar{x})(w) &:= \liminf_{\substack{t\searrow0\\ w'\to w}} \frac{f(\bar{x}+t w)-f(\bar{x})}{t}, &
f'(\bx;w) &:= \lim_{t\searrow0}\frac{f(\bx + t w) - f(\bx)}{t}.
\end{align*}
At points $w\in\R^n$ such that $f'(\bx;w)$ exists and is finite, the \emph{one-sided second directional derivative} is 
\begin{equation*}
f''(\bx;w) := \lim_{t\searrow0}\frac{f(\bx + t w) - f(\bx) - t f'(\bx; w)}{\frac{1}{2}t^2}.
\end{equation*}
For any $w,v\in\R^n$, the \emph{second subderivative at $\bx$ for $v$ and $w\in\R^n$} is
\begin{equation*}
\dif{^2} f(\bx|v)(w):=\liminf_{\substack{t\searrow0\\ w'\to w}}\Delta_t^2 f(\bx|v)(w),\text{ where }\Delta_t^2 f(\bx|v)(w):=\frac{f(\bx+t w') - f(\bx) - t\ip{v}{w'}}{\frac{1}{2}t^2}.
\end{equation*}
The structure of our problem class allows the classical one-sided first and second directional derivatives $f'(\bx;\cdot)$ and $f''(\bx;\cdot)$ to entirely capture the variational properties of their more general counterparts.
\\
Suppose $f:\R^n\to\eR$ is finite at $\bx$. Define the \emph{(Fr\'echet) regular subdifferential} \[\hat \sd f(\bx) := \bset{v\in\R^n}{f(x)\ge f(\bx)+\ip{v}{x-\bx} + o(\norm{x-\bx})},\]
and the \emph{(limiting or Mordukhovich) subdifferential} by
\begin{equation}\label{eq:limitingsd}
\sd f(\bx) := \bset{v\in\R^n}{\exists\, (x^n\xrightarrow[f]{}\bx)\ \exists\,(v^n\to v)\ \forall\,(n\in\bN)\ v^n\in \hsd f(x^n)},
\end{equation}
where $x^n\xrightarrow[f]{}\bx$ denotes \emph{$f$-attentive convergence}, i.e., that $x^n\to\bx,\text{ with }f(x^n)\to f(\bx)$. In the case of a closed, proper, convex function $f$, the set $\sd f(\bx)$ is the usual subdifferential of convex analysis.
The tools of first and second subderivative functions and subdifferential sets allow us to concisely write first-order necessary conditions and second-order necessary and sufficient conditions for local minima.
\begin{theorem}[First-order necessity, second-order necessity and sufficiency] \cite[Theorems 10.1, 13.24] {rockafellar_wets_1998} For a proper function $f:\R^n\to\eR$, consider the problem
	\label{thm:fonc}
	\label{thm:sonsc}
	$\min_x f(x)$.
	\begin{enumerate}[label=(\alph*)]
		\item If $f$ has a local minimum at $\bx$, then $0\in \sd f(\bx)$ and for all $w\in\R^n,\ 
		\dif f(\bx)(w)\ge0$ and $
		\dif{^2}f(\bx|0)(w)\ge0$.
		\item If $0\in\sd f(\bx)$ and $\dif{^2} f(\bx|0)(w)>0$ for $w\ne0$, then $\bx$ is a local minimizer of $f$.
		\item The statement $0\in\sd f(\bx)$ and $\dif{^2} f(\bx|0)(w)>0$ for $w\ne0$ is equivalent to $\bx$ being a \emph{strong local minimizer} of $f$, i.e., there exists a neighborhood $U$ of $\bx$ and a constant $\gamma>0$ such that
		\begin{equation}\label{eq:SLM}
		f(x) \ge f(\bar{x}) + \gamma \norm{x-\bar{x}}^2 \text{ for all }x\in U\cap \dom{f}.
		\end{equation}
	\end{enumerate}
\end{theorem}
A set-valued mapping $S:\R^n\rightrightarrows\R^m$ is a mapping from $\R^n$ into the power set of $\R^m$, so for each $x\in\R^n,\ S(x)\subset\R^m$. The \emph{graph} and \emph{domain} of $S$ are defined to be
\[
\gph{S} := \bset{(x,y)\in\R^n\times\R^m}{y\in S(x)} \text{ and } \dom{S} := \bset{x\in\R^n}{S(x)\neq\emptyset},
\]
and $S$ is \emph{graph-convex} whenever $\gph{S}$ is a convex subset of $\R^n\times\R^m$.
For a point $(\bx,\by)\in\gph{S}$, and neighborhoods $U$ of $\bx$ and $V$ of $\by$, a \emph{graphical localization} of $S$ at $\bx$ for $\by$ is a set-valued mapping $\widetilde S$ defined by $\gph{\widetilde S} = \gph{S}\cap(U\times V)$.	A \emph{single-valued localization} of $S$ \emph{at} $\bx$ for $\by$ is a graphical localization that is also function. If the domain of $\widetilde S$ is a neighborhood of $\bx,\ \widetilde S$ is called a single-valued localization of $S$ \emph{around} $\bx$ for $\by$. The mapping $S$ is \emph{outer semicontinuous} at $\bx$ relative to $X\subset\R^n$ if
\[
\limsup_{x\xrightarrow[X]{}\bx} S(x):=\bset{u}{\exists\, (x^n\xrightarrow[X]{}\bx)\ \exists\,(u^n\to u)\ \forall\,(n\in \bN)\ u^n\in S(x^n)} \subset S(\bx),
\]
and is \emph{inner semicontinous} relative to $X\subset\R^n$ if
\[
S(\bx)\subset \liminf_{x\xrightarrow[X]{}\bx}S(x):=\bset{u}{\forall\, (x^n\xrightarrow[X]{}\bx)\ \exists\,(N\in\bN,\ u^n\to u)\ \forall\,(n\ge N)\ u^n\in S(x^n)},
\]
%It is called continuous at $\bx$ if both conditions hold, i.e., if $S(x)\to S(\bx)$ as $x\to\bx$.
%These terms are invoked relative to $X$, a subset of $\R^n$ containing $\bx$, when the properties hold in restriction to convergence $x\to\bx$ with $x\in X$, in which case the sequences $x^n\to x$ in the limit formulas are required to lie in $X$. 
where $x^n\xrightarrow[X]{}\bx \Longleftrightarrow x^n\to\bx\text{ with }x^n\in X.$
Then, \eqref{eq:limitingsd} is $\sd f(\bx) := \limsup_{x\xrightarrow[f]{}\bx} \hsd f(x)$. The last notion employed from variational analysis is that of normal and tangent vectors. Let $C\subset\R^n$, and let $\bc\in C$. Define the \emph{normal cone} to $C$ at $\bc$ as
\begin{equation}
\label{eq:limitingnormals}
\ncone{\bc}{C} := \limsup_{c\xrightarrow[C]{}\bc}\hat N(c\, |\, C), \text{ where }\hat N(c\, |\, C):=\bset{v}{\forall\, (c'\in C)\ \ip{v}{c'-c}\le o(\norm{c'-c})},
\end{equation}
and the \emph{tangent cone} to $C$ at $\bc$ as $\tcone{\bc}{C}:=\limsup_{t\searrow0}t^{-1}(C-\bc)$. A set $C$ is \emph{Clarke regular} at $\bc\in C$ if $C$ is locally closed at $\bc$ and $\ncone{\bc}{C} = \hat N(\bc\, |\, C)$. A nonempty, closed, convex set $C$ is Clarke regular at all $\bc\in C$, with $\ncone{\bc}{C} = \bset{v}{\ip{v}{c-\bc}\le0\text{ for all }c\in C}$, and $\tcone{\bc}{C} = \bset{v}{\ip{v}{w}\leq0 \text{ for all }w\in\ncone{\bc}{C}} =\cl\set{\R_{++}(C-\bc)}$ \cite[Theorem 6.9]{rockafellar_wets_1998}. We refer the reader to \cite[Chapter 6]{rockafellar_wets_1998} for a thorough exposition.
\\
Suppose $g:\R^n\to \R^m$ is $\cC^1$-smooth, $G:\R^n\rightrightarrows \R^m$ is a set-valued mapping with closed graph and $\set{\mathbf{B}_k}_{k\in\bN}\subset\R^{m\times n}$. Consider the \emph{generalized equation} $0\in g(z)+G(z)$. The \emph{Newton method} for $g+G$ is the iteration
\begin{equation}
\label{eq:geqnewton}
\text{find }z^{k+1}\text{ such that }0\in g(z^k) + \nabla g(z^k)(z^{k+1}-z^k) + G(z^{k+1}), \mbox{ for }k\in\bN,
\end{equation}
and the \emph{quasi-Newton method} for $g+G$ is the iteration
\begin{equation}
\label{eq:geqqn}
\text{find }z^{k+1}\text{ such that }0\in g(z^k) + \mathbf{B}_k(z^{k+1}-z^k) + G(z^{k+1}), \mbox{ for }k\in\bN.
\end{equation}

\section{Convex-composite first- and second-order theory}
\label{sec:cvxcomp}
We begin by recalling the basic ingredients of convex-composite optimization and the associated variational structures.
\begin{definition}[Convex-composite functions]
	Let $h:\R^m\to\eR$ be a closed, proper, convex function and $c:\R^n\to\R^m$ a $\cC^2$-smooth function. Define $f:\R^n\to\eR$ by
	\(
	f(x) := h(c(x)).
	\)
	We say the function $f$ is convex-composite.
\end{definition}
\begin{definition}[Convex-composite Lagrangian]\cite{burke1987second}
	For any $y\in \R^m$, define the function $(yc): \R^n\to\R$ by $(yc)(x) := \ip{y}{c(x)}$. The Lagrangian for the convex-composite $f$ is defined by 
	\(
	L(x,y) := (yc)(x) - h^\star(y),
	\)
	where $h^\star:\R^m\to\eR$ denotes the \emph{Fenchel conjugate} of the convex function $h$ defined by
	\(
	h^\star(y) := \sup_{z\in \R^m} \ip{z}{y} - h(z).
	\)
	The Hessian of $L$ in its first variables is denoted
	\begin{equation}\label{eq:HessLagrangian}
	\nabla^2_{xx}L(x,y) = \nabla^2(y c)(x) = \sum_{i=1}^m y_i\nabla^2c_i(x).
	\end{equation}
\end{definition}
\begin{definition}[Convex-composite multiplier sets]
	Suppose $f$ is convex-composite. Define the set of multipliers at $\bx\in\dom{f}$ for $v\in\R^n$ as in \cite[Theorem 13.14]{rockafellar_wets_1998} by
	\begin{equation}
	\label{eq:cvxcompmultset}
	Y(\bx, v) := \bset{y}{\begin{pmatrix}
		v\\0
		\end{pmatrix} \in \begin{pmatrix}
		\sd_x L(\bx,y) \\
		\sd_y (-L)(\bx,y)
		\end{pmatrix}}=\bset{y\in \sd h(c(\bx))}{\nabla c(\bx)^\top y=v},
	\end{equation}
	and define the set of multipliers at $\bx$ for 0 by
	\begin{equation}
	\label{optMultipliers}
	M(\bx) := Y(\bx, 0) = \Null{\nabla c(\bx)^\top}\cap \sd h(c(\bx)).
	\end{equation}
\end{definition}
A calculus for convex-composite functions at a point $\bx\in\dom{f}$ requires various types of ``constraint qualifications."  Stronger versions of the \emph{basic constraint qualification} \eqref{eq:bcq} will be employed to ensure uniqueness of the multiplier and underlying strict complementarity properties in later sections.
\begin{definition}[Convex-composite constraint qualifications]
	\label{def:cqs}
	Suppose $f$ is convex-composite and $\bx\in\dom{f}$. We say $f$ satisfies the 
	\begin{itemize}
		\item \emph{basic constraint qualification} at $\bx$ if
		\begin{equation}
		\label{eq:bcq}
		\tag{BCQ}
		\Null{\nabla c(\bx)^\top}\cap\ncone{c(\bx)}{\dom{h}}=\set{0},
		\end{equation}
		\item \emph{transversality condition} at $\bx$ if
		\begin{equation}
		\tag{TC}
		\label{eq:generaltransversality}
		\Null{\nabla c(\bar{x})^\top} \cap \parr{\sd h(c(\bar{x}))} = \{0\},
		\end{equation}
		\item \emph{strict criticality condition} at $\bx\in\dom{f}$ for $\by$ if 
		\begin{equation}
		\tag{SC}
		\label{eq:strictcriticality}
		\Null{\nabla c(\bar{x})^\top} \cap \ri{\sd h(c(\bar{x}))} = \set{\bar{y}}.
		\end{equation}
	\end{itemize}
\end{definition}
\begin{remark}
	Following \cite[Definition 10.23]{rockafellar_wets_1998}, one says that a convex-composite function $f$ is strongly amenable at $\bx\in\dom{f}$ if $f$ satisfies \eqref{eq:bcq} at $\bx$. One says that $f$ is fully amenable at $\bx\in\dom{f}$ if $f$ satisfies \eqref{eq:bcq} at $\bx$ and the function $h$ is PLQ convex. Here, we make use of the underlying assumption that $c$ is $\cC^2$-smooth.
\end{remark}
Notice the basic constraint qualification is a \emph{local property} in the following sense. If $f$ satisfies \eqref{eq:bcq} at $\bx$, then there exists a neighborhood $U$ of $\bx$ such that $f$ satisfies \eqref{eq:bcq} at all $x\in [U\cap c^{-1}(\dom{h})]$. Moreover, the basic constraint qualification ensures that the chain rule applies in the subdifferential calculus for convex-composite functions and establishes a foundation for the application of tools from variational analysis.
\begin{theorem}[Convex-composite first order necessary conditions]
	\label{thm:cvxcompfonc}
	Suppose $f$ is convex-composite and $\bx\in\dom{f}$ is such that $f$ satisfies \eqref{eq:bcq} at $\bx$. Then,
	\(
	\sd f(\bx) = \nabla c(\bx)^\top \sd h(c(\bx)),
	\)
	and for any $d\in\R^n$,
	\(
	\dif f(\bx)(d) = h'(c(\bx);\nabla c(\bx)d).
	\)
	Suppose, in addition, that $\bx$ is a local solution to \ref{theprogram}. Then,
	\(
	M(\bx):=\Null{\nabla c(\bx)^\top}\cap\sd h(c(\bx))\neq\emptyset,\text{ or equivalently, } 0\in \sd f(\bx),
	\)
	and for any $d\in\R^n$,
	\(
	h'(c(\bx);\nabla c(\bx)d)\ge0.
	\)
\end{theorem}
\begin{proof}
This follows from \Cref{thm:fonc} and \cite[Proposition 8.21, Exercise 10.26(b)]{rockafellar_wets_1998}.
\end{proof}
We now establish a relationship between the various notions of a constraint qualification given in \Cref{def:cqs}.
\begin{lemma}
	\label{lem:CQimplicationchain}
	Suppose $f$ is convex-composite, $\bx\in\dom{f}$, and $\by\in\R^m$. Then, the following implications hold:
	\[
	\tiny
	\begin{tikzcd}[arrows=Rightarrow]
	& & \eqref{eq:bcq} \\
	\eqref{eq:strictcriticality} \arrow[r] & \eqref{eq:generaltransversality} \arrow[ur] \arrow[dr]\\
	& & (M(\bx)=\set{\by}) \arrow[uu]
	\end{tikzcd}
	\]
\end{lemma}
\begin{proof}

	[\eqref{eq:generaltransversality}$\Longrightarrow$\eqref{eq:bcq}]
	By \cite[Proposition 8.12]{rockafellar_wets_1998}, at any point $\bc\in\dom{\sd h},\ \ncone{\bc}{\dom{h}}\subset \parr{\sd h(\bc)}$. The implication follows.
	
	$[(M(\bx)=\set{\by})\Longrightarrow$\eqref{eq:bcq}]
	
	Let $M(\bx)=\set{\by}$ and suppose there exists 
	\[
	0\neq v\in \Null{\nabla c(\bx)^\top}\cap \ncone{c(\bx)}{\dom{h}}\subset \Null{\nabla c(\bx)^\top}\cap \parr{\sd h(c(\bx))}.
	\]
	Then, by the subgradient inequality, $v+\by\in \Null{\nabla c(\bx)^\top}\cap \sd h(c(\bx))=M(\bx)$, which is a contradiction.
	\\
	The rest of the proof appears in \Cref{lem:cqcvxsetlinmap} in the appendix as general facts about closed convex sets $C$ and linear maps $A$.
\end{proof}
Gauss-Newton methods for iteratively solving \ref{theprogram} are based on finding a search direction that approximates a solution to subproblems of the form
\begin{mini}|l|
	{d\in\R^n}{h(c(\hat x) + \nabla c(\hat x)d) + \frac{1}{2}d^\top \widehat Hd.}{\tag{$\mathbf{\widehat P}$}}{}
\end{mini}
Local rates of convergence for algorithms of this type, where the function $h$ is assumed to be finite-valued and piecewise linear convex were developed by Womersley \cite{womersley1985local} based on tools developed for classical nonlinear programming. More recently, Cibulka et.~al.~\cite{cibulka2016strong} successfully applied a modern approach through generalized equations to obtain similar and stronger results again in the piecewise linear convex case. Inspired by these results and the existence of a sophisticated first- and second-order subdifferential calculus for piecewise linear-quadratic convex functions \cite{rockafellar_wets_1998}, we develop a convergence theory in the piecewise linear-quadratic case from the generalized equations perspective. The basic notational objects for our development are given in the next definition.
\begin{definition}[Convex-composite generalized equations]
	\label{def:g+G}
	Let $f$ be convex-composite, and define the set-valued mapping $g+G:\R^{n+m}\rightrightarrows \R^{n+m}$ by 
	\begin{equation}
	\label{kktS}
	g(x,y) = \begin{pmatrix}
	\nabla c(x)^\top y \\
	-c(x)
	\end{pmatrix},\quad
	G(x,y) = \begin{pmatrix}
	\{0\}^n\\
	\sd h^\star(y)
	\end{pmatrix}.
	\end{equation}
	For a fixed $(\bx,\by)\in\R^n\times\R^m$, define the linearization mapping 
	\begin{equation}
	\label{eq:linearizedmapping}
	\cG:(x,y)\mapsto g(\bx, \by) + \nabla g(\bx, \by)\begin{pmatrix}
	x-\bx\\y-\by
	\end{pmatrix} + G(x, y),
	\end{equation}
	where
	\(
	\nabla g(\bx,\by) = \begin{pmatrix}
	\nabla^2(\by c)(\bx) & \nabla c(\bx)^\top \\
	-\nabla c(\bx) & 0
	\end{pmatrix}.
	\)
\end{definition}
Observe that for any $\bx\in\dom{f}$ where $f$ satisfies \eqref{eq:bcq}, $\bx$ satisfies the first-order necessary conditions of \Cref{thm:fonc} for the problem \ref{theprogram} if and only if there exists $\by$ such that $(\bx,\by)$ solves the generalized equation $g+G\ni0$. More precisely, we have
\begin{equation}
\label{eq:GEQstationarity}
0 \in g(\bx,\by)+G(\bx,\by) \Leftrightarrow \nabla c(\bx)^\top \by = 0\text{ and }\by\in\sd h(c(\bx))\Leftrightarrow M(\bx)\neq\emptyset.
\end{equation}
The relationship between the linearization of the generalized equation described in \eqref{eq:linearizedmapping} and the subproblems \ref{subproblems} is described in the following lemma. The proof follows from \Cref{thm:cvxcompfonc}.
\begin{lemma}
	Let $f$ be convex-composite and $(\hx,\hy)\in\R^n\times\R^m$ be such that $f$ satisfies \eqref{eq:bcq} at $\hx$, and define $\widehat H := \nabla^2(\hat yc)(\hat x)$. Then, $(\tilde d,\tilde y)\in\R^n\times\R^m$ satisfy the optimality conditions for
	\begin{mini}|l|
		{d\in\R^n}{h(c(\hat x) + \nabla c(\hat x)d) + \frac{1}{2}d^\top \widehat Hd}{\tag{$\mathbf{\widehat P}$}}{\label{subproblems}}
	\end{mini}
	if and only if $(\hat x\!+\!\tilde d,\tilde y)$ solves the Newton equations for $g\!+\!G$:
	\(
	0\!\in\!g(\hat x,\hy)\!+\!\nabla g(\hx,\hy)\begin{pmatrix}
	x\!-\!\hx \\ y\!-\!\hy
	\end{pmatrix}\!+\!G(x,y).
	\)
\end{lemma}
\section{Geometry of PLQ Functions and Their Domains} \label{sec:geometry}
In this section, unless otherwise stated, we let $f := h\circ c$ where $h$ is piecewise linear-quadratic convex and $c$ is $\cC^2$-smooth.
\begin{definition}[piecewise linear-quadratic]
	\label{def:plq}
	A proper function $h:\R^m\to\eR$ is called piecewise linear-quadratic (PLQ) if $\dom{h}\neq\emptyset$ and $\dom{h}$ can be represented as the union of $\cK\ge1$ polyhedral sets of the form
	\begin{equation}
	\label{eq:polyhedralsets}
	C_k = \bset{c}{
		\begin{aligned}
		& \ip{a_{kj}}{c} \le \alpha_{kj},\text{ for all } j\in\set{1,\dotsc,s_k}\\
		\end{aligned}
	}
	\end{equation}
	relative to each of which $h(c)$ is given by an expression of the form $\frac{1}{2}\ip{c}{Q_kc} + \ip{b_k}{c} + \beta_k$ for some scalar $\beta_k\in\R$, vector $b_k\in\R^n$, and symmetric matrix $Q_k$.
\end{definition}
\begin{remark}
	The sets $C_k$ do not necessarily form a partition of the set $C$.
\end{remark}
The following lemma is straightforward.
\begin{lemma}
	Suppose $h$ is piecewise linear-quadratic convex. Then, for any $k\in \cK$, the matrices $Q_k$ satisfy $\ip{c}{Q_kc}\ge0$ for all $c\in \parr{C_k}$.
\end{lemma}
%\begin{proof}
%	Suppose $\bar c\in \ri{C_k}\subset\dom{h}$. Then for $\epsilon>0$ sufficiently small and any $c'\in (\bar c + \eps\bB)\cap(\set{\bc} + \parr{C_k}),\ c'\in C_k$. Choose any $c\in \parr{C}$ with $\norm{c}\le\eps$. Convexity of $h$ implies for any $\lambda\in(0,1)$,
%	\[
%		q(\lambda(\bc + c) + (1-\lambda)\bc) \le \lambda q(\bc + c) + (1-\lambda)q(\bc) \Longleftrightarrow q(\lambda c + \bc) \le \lambda q(\bc + c) + (1-\lambda)q(\bc).
%	\]
%	Then
%	\begin{align*}
%		\ip{b_k}{\lambda c + \bc} + \frac{1}{2}\ip{\lambda c + \bc}{Q_k(\lambda c + \bc)} \\
%		\le \lambda[\ip{b_k}{c + \bc} + \frac{1}{2}\ip{c + \bc}{Q_k(c + \bc)}] + (1-\lambda)[\ip{b_k}{\bc} + \frac{1}{2}\ip{\bc}{Q_k\bc}]
%	\end{align*}
%	simplifies to $\ip{c}{Q_kc}\ge0$. Since $c\in\parr{C}$ was an arbitrary direction and the inequality is unaffected by scaling, the result follows.
%\end{proof}
For the sake of reference we recall the normal and tangent cone structure for polyhedral sets. 
%Our first step toward developing a deeper understanding of PLQ functions is to understand the piecewise polyhedral nature of $\dom{h}$.
\begin{definition}[Active indices]
	\label{def:activeidx}
	For a piecewise linear-quadratic function $h$ and a point $\bar c \in \dom{h}$, define the set
	\(
	\cK(\bar c) := \bset{k\in \cK}{\bar c \in C_k},
	\)
	and write $\bk := |\cK(\bc)|,$ so that $\cK(\bc) = \set{k_1,k_2,\dotsc,k_{\bk}}$.
\end{definition}
\begin{theorem}[Normal and Tangent Cones to Polyhedra]
	\label{thm:polyhedraltconencone}
	\cite[Theorem 6.46]{rockafellar_wets_1998} Suppose $c\in C_k$ with $C_k$ polyhedral as in \eqref{eq:polyhedralsets}. Let $I_k(c) = \bset{j\in\set{1,\dotsc,s_k}}{\ip{a_{kj}}{c}=\alpha_{kj}}$, and let $\ell_k = |I_k(c)|$. Then,
	\begin{equation}
	\label{eq:ckncone}
	\ncone{c}{C_k}=\bset
	{
		\sum_{j\in I_k(c)} \lambda_ja_{kj}
	}
	{
		\lambda_j\ge0,\ j\in I_k(c)
	}\text{ and } 		\tcone{c}{C_k}=\bset{v}
	{
		\ip{a_{kj}}{v} \le 0,\ j\in I_k(c)
	}.
	\end{equation}
\end{theorem}
Our first- and second-order analysis in the PLQ case heavily depends on the following results from  \cite{rockafellar_wets_1998}.
\begin{proposition}
\cite[Propositions 10.21, 13.9]{rockafellar_wets_1998}
	\label{plqformulae}
	If $h:\R^m\to\eR$ is piecewise linear-quadratic, then $\dom{h}$ is closed, $h$ is continuous relative to $\dom{h}$. Consequently, $h$ is closed. At any point $\bc\in \dom{h},\ h'(\bc;\cdot)=\dif h(\bar c)$, and $h'(\bc;\cdot)$ is piecewise linear with
	\(
	\dom{h'(\bc;\cdot)}=\bigcup_{k\in K(\bar c)} \tcone{\bar c}{C_k} = \tcone{\bar c}{\dom{h}}.
	\)
	In particular, for $k\in \cK(\bc)$ and $w \in \tcone{\bar c}{C_k}$,
	\begin{equation}
	\label{eq:plqsubderiv}
	h'(\bar c;w) = \ip{Q_k\bar c+b_k}{w}.
	\end{equation}
	If, in addition, $h$ is convex, then $\dom{h}$ is polyhedral, %The subgradient sets $\sd h(\bar c)$ and $\hznsd h(\bar c)$ at any point $\bar c \in \dom{h}$ are polyhedral as well, with $\hznsd h(\bc) = \ncone{\bc}{\dom{h}}$ and
	\begin{equation}
	\label{eq:plqsubdiff}
	\emptyset \ne \sd h(\bar c) = \bigcap_{k\in \cK(\bc)} \bset{y}{y-Q_k\bar c-b_k\in \ncone{\bar c}{C_k}},
	\end{equation}
	$h''(\bc;\cdot)$ is piecewise linear-quadratic, but not necessarily convex, and for any $w\in\R^m$,
	\begin{equation}
	\label{eq:plq2ndnonneg}
	0\le h''(\bc;w) = \begin{cases}
	\ip{w}{Q_kw} & \text{when }w\in \tcone{\bc}{C_k}, \\
	\infty & \text{when }w\not\in\tcone{\bc}{\dom{h}}.
	\end{cases}
	\end{equation}
	For every $y\in\sd h(\bc),\,\, \dif{^2} h(\bc|y)$ is piecewise linear-quadratic and convex. Let 
	\(
	K(\bc, y) := \bset{w}{h''(\bc;w)=\ip{y}{w}}.
	\)
	Then, $K(\bc, y)$ is a polyhedral cone, and
	\begin{equation}
	\label{eq:plqsecondsubderivxvw}
	\dif{^2} h(\bc|y)(w) = \lim_{\tau\searrow0}\Delta_\tau^2 h(\bc|y)(w) = \begin{cases} h''(\bc;w) & w\in K(\bc,y),\\
	+\infty & \text{otherwise.}	
	\end{cases}
	\end{equation}
	Moreover, there exists a neighborhood $V$ of $\bc$ such that
	\begin{equation}
	\label{eq:plq2ndexpansion}
	h(c) = h(\bc) + h'(\bc;c-\bc) + \frac{1}{2}h''(\bc;c-\bc) \text{ for }c\in V\cap\dom{h}.
	\end{equation}
\end{proposition}
\begin{theorem} 
	\label{thm:rockafellarplqcomposite}
	\cite[Theorem 13.14]{rockafellar_wets_1998} Let $f=h\circ c$ for a $\cC^2$ mapping $c:\R^n\to\R^m$ and a piecewise linear-quadratic convex $h:\R^m\to\eR$. Let $\bx\in \dom{f}$ and suppose $f$ satisfies \eqref{eq:bcq} at $\bx$.
	%which is known to imply that 
	%$f$ is regular at $\bx$ with 
	%$\sd f(\bx) = \bset{\nabla c(\bx)^\top y}{y\in\sd h(c(\bx))}$ and
	%\[
	%	\dif f(\bx)(w) = h'(c(\bx);\nabla c(\bx)w).
	%\]
	Then, for any $v\in\sd f(\bx)$, the set	$Y(\bx, v)$ given by \eqref{eq:cvxcompmultset}
	is compact as well as convex and nonempty, and for any $w\in\R^n$
	\begin{equation}
	\label{eq:secondderivHessianexpansion}
	\dif{^2} f(\bx|v)(w) = \dif{^2} \bar{f}(\bx|v)(w) + \max{\bset{\ip{w}{\nabla^2(yc)(\bx)w}}{y\in Y(\bx,v)}},
	\end{equation}
	with $\bar f(x):=h(c(\bx) + \nabla c(\bx)[x-\bx])$ piecewise linear-quadratic convex.
\end{theorem}
The standard development of first- and second-order optimality conditions requires the notion of directions of non-ascent.
\begin{definition}
	\label{def:nonascent}
	Let the \emph{directions of non-ascent} for any proper $f:\R^n\to\eR$ at $x\in\dom{f}$ be denoted by 
	\(
	D(x) := \bset{d\in \R^n}{\dif f(x)(d)\le0}.
	\)
	By \Cref{thm:cvxcompfonc}, if $f$ is convex-composite and $f$ satisfies \eqref{eq:bcq} at $x$, then
	\begin{equation}
	\label{eq:nonascentcone}
	D(x) = \bset{d\in \R^n}{h'(c(x);\nabla c(x)d)\le0}
	\end{equation}
\end{definition}
In the PLQ convex case, \eqref{eq:bcq} ensures that we have the following convenient representation of the set $D(\bx)$.
\begin{lemma}
	Let $f$ be as in \ref{theprogram}, and let $\bx\in\R^n$ be such that $f$ satisfies \eqref{eq:bcq} at $\bx$. Set $\bc := c(\bx)$. Then, $D(\bx)$ is convex and the union of finitely many polyhedral closed convex sets with following the representation:
	\begin{equation}
	\label{eq:plqnonascent}
	\begin{aligned}
	D(\bx) &= \bigcup_{k\in \cK(\bc)} \bset{d}{\nabla c(\bx)d\in \tcone{\bar c}{C_k}, \ip{Q_k\bar c+b_k}{\nabla c(\bx)d}\le0}\\
	&=\bigcup_{k\in \cK(\bc)} \bset{d}
	{
		\begin{aligned}
		& \ip{Q_k\bar c + b_k}{\nabla c(\bx)d} \le 0 \\
		& \ip{a_{kj}}{\nabla c(\bx)d} \le 0, j\in I_k(\bar c)
		\end{aligned}
	}
	\end{aligned}			
	\end{equation}
\end{lemma}
\begin{proof}
	\noindent$(\subset)$ Suppose $d\in D(\bx)$. By \eqref{eq:nonascentcone}, $\nabla c(\bx) d\in\dom{h'(\bc;\cdot)}$. In particular, by \Cref{plqformulae}, $\nabla c(\bx)d\in \tcone{\bar c}{C_k}$ for some $k\in \cK(\bc)$. By \eqref{eq:plqsubderiv}, we also have
	\(
	\ip{Q_k\bar c+b_k}{\nabla c(\bx) d}=h'(c(\bx);\nabla c(\bx) d)\le0.
	\)
	\\
	$(\supset)$ If $d\in \bigcup_{k\in \cK(\bc)} \bset{d}{\nabla c(\bx)d\in \tcone{c}{C_k}, \ip{Q_k\bar c+b_k}{\nabla c(\bx)d}\le0}$, then for some $k\in \cK(\bc), \nabla c(\bx)d\in \tcone{\bar c}{C_k}$. Then, again by \Cref{plqformulae},  $h'(c(\bx); \nabla c(\bx)d) = \ip{Q_k\bar c+b_k}{\nabla c(\bx)d}\le0$, so $d\in D(\bx)$.
\end{proof}
We now have the tools necessary to rewrite \Cref{thm:sonsc} in the context of piecewise linear-quadratic convex functions $h$.
\begin{theorem}[PLQ second-order necessary and sufficient conditions] \cite[Theorems 13.24(b), 13.14]{rockafellar_wets_1998}, \cite[Theorem 3.4]{rockafellar1989second}.
	\label{thm:plqsonsc}
	Let $h:\R^m\to\eR$ be piecewise linear-quadratic and convex with $\bx\in\dom{f}$ such that $f$ satisfies \eqref{eq:bcq} at $\bx$.
	\begin{enumerate}[label=(\alph*)]
		\item If $f$ has a local minimum at $\bx$, then $0\in \nabla c(\bx)^\top\sd h(c(\bx))$ and
		\[
		h''(c(\bx); \nabla c(\bx)d) + \max\bset{\ip{d}{\nabla^2(yc)(\bx)d}}{y\in M(\bx)}\ge0
		\]
		for all $d\in D(\bx)$.
		\item  If $0\in \nabla c(\bx)^\top \sd h(c(\bx))$ and
		\[
		h''(c(\bx); \nabla c(\bx)d) + \max\bset{\ip{d}{\nabla^2(yc)(\bx)d}}{y\in M(\bx)}>0
		\]
		for all $d\in D(\bx)\setminus\set{0}$, then $\bx$ is a strong local minimizer (see \eqref{eq:SLM}) of $f$.
	\end{enumerate}
\end{theorem}
\section{Strong Metric Subregularity of the KKT Mapping} \label{sec:subreg}
In this section we establish conditions under which the set-valued mapping \Cref{def:g+G} satisfies strong metric subregularity. 
\begin{definition}[Strong metric subregularity]
	\label{def:subreg}
	A set-valued mapping $S:\R^n\rightrightarrows \R^m$ is \emph{strongly metrically subregular} at $\bx$ for $\by$ if $(\bx,\by)\in\gph{S}$ and there exists $\kappa\ge0$ and a neighborhood $U$ of $\bx$ such that
	\(
	\norm{x-\bx}\le \kappa \dist{\by}{S(x)}\text{ for all }x\in U.
	\)
\end{definition}
%\begin{remark}
%	Immediately from the definition, $\bx$ is an isolated point of $S^{-1}(\by)$.
%\end{remark}
Our discussion of strong metric subregularity only requires $f$ to satisfy \eqref{eq:bcq} at $\bx\in\dom{f}$. 
%We begin by showing that a point $\bx\in\dom{f}$ where $f$ satisfies \eqref{eq:bcq} and conditions guaranteeing  uniqueness of $M(\bx)$. Then $\bx$ satisfies the second-order sufficient conditions of \ref{theprogram} are shown to be equivalent to a generalized form of injectivity of the set-valued mapping $g+G$ \eqref{kktS} together with local quadratic growth from $\bx$. A similar result is known for convex piecewise linear $h$ \cite[Theorem 7.1]{cibulka2016strong}.
\begin{lemma}
	\label{polyhedralgraph}
	Consider the KKT mapping $g+G$ and the mapping $\cG$ given in \Cref{def:g+G}. Then, strong metric subregularity of $g+G$ at $(\bx, \by)$ for $0$ is equivalent to the property that $(\bar{x},\bar{y})$ is an isolated point of $\cG^{-1}(0)$.
\end{lemma}
\begin{proof}By \cite[Corollary 3I.10]{dontchev2014implicit}, strong metric subregularity of $g+G$ at $(\bar{x},\bar{y})$ for 0 is equivalent to strong metric subregularity of the linearization $\cG$ \eqref{eq:linearizedmapping} at $(\bar{x},\bar{y})$.
	
	By \cite[Theorem 11.14, Proposition 12.30]{rockafellar_wets_1998} the mapping $G(x,y)$ is polyhedral; that is, $\gph{G}$ is the union of finitely many polyhedral sets.
	%		
	%		Set $H = \nabla^2_{xx}L(\bx, \by)$.
	%		
	%		Let $\pi:\R^{n+m+n+m+m}\to\R^{n+m+n+m}$ be the projection
	%		\[
	%			\pi(x,y,u,v,w) = (x,y,u,v).
	%		\]
	%		Then
	%		 the graph of $\cG$ is
	%	\begin{align*}
	%		\gph{\cG} &:= \bset{  
	%			(x,y,u,v)}{\begin{pmatrix}
	%			u\\v
	%			\end{pmatrix} \in \cG(x,y)
	%		}\\
	%		&= \bset{
	%		 (x,y,u,v)}{\begin{pmatrix}
	%		 u\\v
	%		 \end{pmatrix} \in \begin{pmatrix}
	%		 \nabla c(\bar{x})^\top \by\\
	%		 -c(\bx)
	%		 \end{pmatrix} + \begin{pmatrix}
	%		 H & \nabla c(\bx)^\top \\
	%		 -\nabla c(\bx) & 0
	%		 \end{pmatrix}\begin{pmatrix}
	%		 x-\bx\\y-\by
	%		 \end{pmatrix} +\begin{pmatrix}
	%		 \{0\}^n\\
	%		 \sd h^\star(y)
	%		 \end{pmatrix}
	%		} \\
	%		&= \pi\bset{
	%		(x,y,u,v,w)}{
	%		\begin{aligned}
	%		&\begin{pmatrix}
	%		y\\w 
	%		\end{pmatrix} \in \gph{\sd h^\star},\ -u + H x + \nabla c(\bx)^\top y = H\bx ,  
	%		\\ &
	%		\ w - v - \nabla c(\bx)x = \nabla c(\bx) \bx + c(\bx)
	%		\end{aligned}
	%		}\\
	%		&=: \pi(G)
	%	\end{align*}
	%	
	%	Therefore $G=\bigcup_{i=1}^m G_i$, where $G_i$ are polyhedral. Also $\pi(G) = \bigcup_{i=1}^m \pi(G_i)$, and each $\pi(G_i)$ is polyhedral, 
	Then \cite[Corollary 3I.11]{dontchev2014implicit} establishes the equivalence of strong metric subregularity of $\cG$ at $(\bx,\by)$ for $0$ and $(\bx,\by)$ being an isolated point of $\cG^{-1}(0)$.
\end{proof}
The main result of this section now follows.
\begin{theorem} Suppose $h:\R^m\to\eR$ is piecewise linear-quadratic and convex with $\bx\in\dom{f}$ such that $f$ satisfies \eqref{eq:bcq} at $\bx$. Then, the following are equivalent:
	\begin{enumerate}
		\item The set $M(\bx):= \Null{\nabla c(\bx)^\top}\cap \sd h(c(\bx))$ in \eqref{optMultipliers} is a singleton and the second-order sufficient conditions of \Cref{thm:plqsonsc} are satisfied at $\bx$;
		\item The mapping $g+G$ is strongly metrically subregular at $(\bar{x},\bar{y})$ for 0 and $\bar{x}$ is a strong local minimizer of $f$.
	\end{enumerate}
\end{theorem}
\begin{proof}
	For a point $x\in\dom{f}$, define $\Delta f(x;d) := h(c(x)+\nabla c(x)d)-h(c(x))$.
	
	$(\Rightarrow)$ By \Cref{polyhedralgraph} we argue strong metric subregularity of $g+G$ at $(\bx, \by)$ for 0 by showing that there is a neighborhood of $(\bx, \by)$ on which $(\bx, \by)$ is the unique solution to the generalized equation $\cG\ni 0$ \eqref{eq:linearizedmapping}.
%	\begin{align*}
%	&H(x-\bx) + \nabla c(\bx)^\top y = 0 \\
%	&c(\bx) + \nabla c(\bx)(x-\bx) \in \sd h^\star (y),
%	\end{align*}
%	
	After the change of variables $d:=x-\bx$, we show that there is a neighborhood 
	$U$ of $(0, \by)$ such that $(d,y)=(0,\bar{y})$ is the unique solution to the generalized equation
	\begin{align}
	&H d + \nabla c(\bx)^\top y = 0 \label{eq:ge1}\\
	&c(\bx) + \nabla c(\bx)d \in \sd h^\star (y)\quad (\Leftrightarrow y\in \sd h(c(\bx) + \nabla c(\bx)d))\label{eq:ge2},
	\end{align}
	where $H:=\nabla_{xx}^2L(\bx,\by)$.
	Suppose there is no such neighborhood. 
	Then, there exists a sequence of vectors $\{(d^i, y^i)\}_{i\in\bN}$ converging to $(0, \by)$ with $(d^i,y^i)\ne(0,\by)$ that solve the generalized equation \eqref{eq:ge1}, \eqref{eq:ge2}. First assume $d^i \ne 0$ for all $i\in\bN$. Define for each $i\in\bN,\, t_i := \norm{d^i}, v^i := d^i/\norm{d^i}$, and assume without loss of generality that $v^i\to \bar v$ and that 
	\begin{align}
	\label{activesubsequence}
	\set{c(\bx) + \nabla c(\bx)d^i}_{i\in\bN}\subset C_{k_0}\text{ for some }k_0\in K(c(\bx) + \nabla c(\bx)d^i)\subset K(\bc),
	\end{align} 
	since $d^i\to0$. Taking the inner product on both sides of \eqref{eq:ge1} with $d^i$, we obtain
	\begin{equation}
	\label{eq:quadformprimaldual}
	0 = \ip{d^i}{H d^i} + \ip{d^i}{\nabla c(\bx)^\top y^i} \text{ for all }i\in\bN.
	\end{equation}
	The subgradient inequality for $h$ at $c(\bx) + \nabla c(\bx)d^i$ with subgradient $y_i$ gives
	\begin{equation}
	\label{eq:delfbounds}
	\Delta f(\bx; d^i)\le \ip{d^i}{\nabla c(\bx)^\top y_i} = -\ip{d^i}{H d^i}.
	\end{equation}
	Dividing through by $t_i>0$ and letting $i\to\infty$,
	\(
	\dif f(\bx)(\bv) \le \liminf_i\frac{\Delta f(\bx; t_i v^i)}{t_i}.
	\)
	Hence by \eqref{eq:bcq}, \Cref{thm:rockafellarplqcomposite} and \eqref{eq:delfbounds},
	\(
	h'(c(\bx);\nabla c(\bx)\bv) = \dif f(\bx)(\bv) \le \lim_i -\ip{v^i}{H d^i} = 0,
	\)
	and so $\bv\in D(\bx)\setminus\set{0}$.	By second-order sufficiency, $h''(c(\bx);\nabla c(\bx)\bar v) + \bv^\top H \bv  > 0$. We now show $\nabla c(\bx)\bar v \in \tcone{\bar c}{C_{k_0}}$. By \eqref{activesubsequence} and the computation
	\(
	\frac{c(\bx)+\nabla c(\bx)d^i-c(\bx)}{t_i}=\nabla c(\bx)v^i\to \nabla c(\bx)\bar v \in \tcone{\bar c}{C_{k_0}}.
	\)
	Then by \eqref{eq:plq2ndnonneg}, $h''(c(\bx);\nabla c(\bx)\bar v) = \bar v^\top \nabla c(\bx)^\top Q_{k_0}\nabla c(\bx)\bar v$, so that
	\begin{equation}
	\bar v^\top H \bar v + \bar v^\top \nabla c(\bx)^\top Q_{k_0}\nabla c(\bx)\bar v > 0.
	\end{equation}
	On the other hand, by \eqref{eq:plqsubdiff},
	\[
	y^i\!\in\!\sd h(c(\bx)\!+\!\nabla c(\bx)d^i)\!=\!\bigcap_{k\!\in\!\cK(c(\bx)\!+\!\nabla c(\bx)d^i)} \bset{y}{y-Q_k(c(\bx) + \nabla c(\bx)d^i)-b_k\in \ncone{c(\bx) + \nabla c(\bx)d^i}{C_k}},
	\]
	and so $y^i-Q_{k_0}(c(\bx) + \nabla c(\bx)d^i)-b_{k_0}\in \ncone{c(\bx) + \nabla c(\bx)d^i}{C_{k_0}}$ for all $i\in \bN$. Since $c(\bx)\in C_{k_0}$, we have
	\begin{align*}
	0&\ge \ip{y^i-[Q_{k_0}(c(\bx) + \nabla c(\bx)d^i)+b_{k_0}]}{c(\bx) - [c(\bx) + \nabla c(\bx)d^i]}\\
	&= \ip{y^i-Q_{k_0}(c(\bx) + \nabla c(\bx)d^i)-b_{k_0}}{-\nabla c(\bx)d^i} \\
	&= -\ip{d^i}{\nabla c(\bx)^\top y^i} + \ip{Q_{k_0}(c(\bx) + \nabla c(\bx)d^i)+b_{k_0}}{\nabla c(\bx)d^i}.
	\end{align*}	
	Together with \eqref{eq:quadformprimaldual},
	\begin{align*}
	0 & \ge \ip{d^i}{H d^i} + \ip{Q_{k_0}(c(\bx) + \nabla c(\bx)d^i)+b_{k_0}}{\nabla c(\bx)d^i} \\
	&= \ip{d^i}{H d^i} + \ip{\nabla c(\bx)d^i}{Q_{k_0} \nabla c(\bx)d^i} + \ip{Q_{k_0}c(\bx)+b_{k_0}}{\nabla c(\bx)d^i}\\
	&= \ip{d^i}{H d^i} + \ip{\nabla c(\bx)d^i}{Q_{k_0} \nabla c(\bx)d^i} + h'(c(\bx);\nabla c(\bx)d^i) \text{ (by \eqref{eq:plqsubderiv})}\\
	&\ge \ip{d^i}{H d^i} + \ip{\nabla c(\bx)d^i}{Q_{k_0} \nabla c(\bx)d^i},
	\end{align*}
	where the final inequality follows from \Cref{thm:fonc}, \Cref{thm:plqsonsc}, and the observation that
	\(
	\nabla c(\bx)d^i \in C_{k_0}-c(\bx)\subset\tcone{c(\bx)}{C_{k_0}}.
	\)
	Next, divide the inequality
	\(
	0 \ge \ip{d^i}{Hd^i} + \ip{\nabla c(\bx)d^i}{Q_{k_0} \nabla c(\bx)d^i}
	\)
	by $t_i^2$ and let $i\to \infty$ to yield the contradiction
	\(
	0\ge \bar v^\top H \bar v + \bar v^\top \nabla c(\bx)^\top Q_k\nabla c(\bx)\bar v > 0.
	\)
	\\
	Consequently, $d^i = 0$ for all $i$ sufficiently large, so without loss of generality, we now suppose $d^i=0$ for all $i\in \bN$. Hence by hypothesis, and $y^i\ne\by$ for all $i\in \bN$. But then we contradict uniqueness of $M(\bx)$. %The result that $\bx$ is a strong local minimizer of $f$ whenever the second-order sufficient conditions are satisfied is given in \cite[Theorem 3.4]{rockafellar1989second}.
	\\
	$(\Leftarrow)$ By \Cref{polyhedralgraph}, $(\bx,\by)$ is an isolated point of $\cG^{-1}(0)$. That is, there is a neighborhood $U$ of $(\bar{x},\bar{y})$ on which  
	$(\bar{x},\bar{y})$ is the unique solution to the generalized equation
	\begin{align*}
	&H(x - \bx) + \nabla c(\bx)^\top y  = 0 \\
	&c(\bx) + \nabla c(\bx)(x - \bx) \in \sd h^\star (y).
	\end{align*}
	For $x=\bar{x}$, this implies there is a neighborhood $U_{\bar{y}}$ about $\bar{y}$ such that
	\begin{equation}\label{eq:y unique}
	U_{\bar{y}} \cap M(\bx) = \{\by\}.
	\end{equation}
	Suppose there is $y\in \left(M(\bx)\right)\setminus  U_{\bar{y}}$.
	Then $y_t = (1-t) \by + t y\in M(\bx)$ for $t\in [0,1]$.
	But for $t$ small, $y_t\in U_{\bar{y}} \cap M(\bx)$,
	which contradicts \eqref{eq:y unique}, so $M(\bx)$ is the singleton $\set{\by}$. Therefore, it only remains to show that the second-order sufficient conditions of \Cref{thm:plqsonsc} are satisfied at $\bx$.
	\\	
	Since $\bx$ is local minimizer of $f$ at which $f$ satisfies \eqref{eq:bcq}, \Cref{thm:cvxcompfonc} gives $0\in \nabla c(\bx)^\top \sd h(c(\bx))$ and $h'(c(\bx);\nabla c(\bx)d) \ge 0$ for all $d\in \R^n$. Let $\bd\in \R^n\setminus\{0\}$ with $h'(c(\bx);\nabla c(\bx)d) = 0$, or equivalently, $\bd \in D(\bx)$. Without loss of generality, suppose $\norm{\bd}=1$. In particular, by \eqref{eq:plqnonascent}, there exists $k_0\in K(\bc)$ such that 
	\begin{equation}
	\label{eq:dbxtcone}
	\nabla c(\bx)\bd \in \tcone{\bc}{C_{k_0}}\text{ and }\ip{Q_{k_0}\bc + b_{k_0}}{\nabla c(\bx)\bd} =  h'(c(\bx);\nabla c(\bx)\bd)=0		
	\end{equation}
	\\	 
	Since $h$ is PLQ convex, the second-order necessary conditions of \Cref{thm:plqsonsc} imply  $ h''(c(\bx);\nabla c(\bx)\bar d) + \bar d^\top H\bar d\ge 0$. 
	\\	
	We show this inequality is strict to complete the proof. Suppose to the contrary that 
	\begin{equation}
	\label{eq:subregconecontradiction}
	h''(c(\bx);\nabla c(\bx)\bar d) + \bar d^\top H\bar d = 0.
	\end{equation}
	Then, $\bar d\ne 0$ solves the program
	\begin{mini*}
		{d}{h'(c(\bx);\nabla c(\bx)d) + \frac{1}{2}h''(c(\bx);\nabla c(\bx)d) + \frac{1}{2}d^\top H d}{}{}
		\addConstraint{d}{\in D(\bx).}
	\end{mini*}
	By \eqref{eq:plq2ndexpansion} and continuity of $d\mapsto c(\bx) + \nabla c(\bx)d$, there exists $\epsilon>0$ so that
	\[
	\Delta f(\bx;d) =  h'(c(\bx);\nabla c(\bx)d) + \frac{1}{2}h''(c(\bx);\nabla c(\bx)d)\text{ for }d\in \eps\bB\cap \bset{d}{c(\bx) + \nabla c(\bx)d\in \dom{h}}.
	\]
	By \eqref{eq:dbxtcone} and polyhedrality, $c(\bx) + t\nabla c(\bx)\bd\in \dom{h}$ for sufficiently small $t>0$. It follows, after shrinking $\eps>0$ if necessary, that
	\begin{equation}
	\label{eq:bimpliesadelta=0}
	\Delta f(\bx; t\bd) + \frac{t^2}{2}\bd^\top H \bd = 0 \text{ for all }0\le t<\eps.
	\end{equation}
	Since $0\in \sd f(\bx)$ and $f$ satisfies \eqref{eq:bcq} at $\bx$, \eqref{eq:secondderivHessianexpansion} with $v=0,\ y=\by,$ and $w\in\R^n$ gives
	\(
	\dif{^2} f(\bx|0)(w) = \dif{^2} \bar{f}(\bx|0)(w) + w^\top Hw,
	\)
	where $\bar f$ is also piecewise linear-quadratic by the discussion following \eqref{eq:secondderivHessianexpansion}.
	Since $\bx$ is a strong local minimizer,
	\[
	\dif{^2} f(\bx|0)(w) = \liminf_{\substack{\tau\searrow0\\ w'\to w}}\frac{f(\bx + tw') - f(\bx)}{\frac{1}{2}\tau^2} \ge \liminf_{\substack{\tau\searrow0\\ w'\to w}}\gamma \norm{w'}^2 = \gamma \norm{w}^2 \text{ (see  \Cref{thm:sonsc})}.
	\]
	Then, we have
	\(
	\dif{^2} f(\bx|0)(w) = \dif{^2} \bar{f}(\bx|0)(w) + w^\top H w \ge \gamma \norm{w}^2.
	\)
	By \eqref{eq:plqsecondsubderivxvw} the $\liminf$ defining $\dif{^2} \bar{f}(\bx|0)(w)$ is also expressed as a limit only in $\tau$ (because $\bar f$ is piecewise linear-quadratic), so
	\[
	\dif{^2} \bar f(\bx|0)(w) = \lim_{\tau\searrow0}\frac{\bar f(\bx + \tau w) - \bar f(\bx)}{\frac{1}{2}\tau^2} = \lim_{\tau\searrow0} \frac{\Delta f(\bx; \tau w)}{\frac{1}{2}\tau^2}.
	\]
	Putting the last two observations together,
	\(
	\dif{^2} f(\bx|0)(w) = \lim_{\tau\searrow0} \frac{\Delta f(\bx; \tau w)}{\frac{1}{2}\tau^2} + w^\top H w \ge \gamma\norm{w}^2.
	\)
	But, for $0<\tau<\eps$ and $w=\bd$, \eqref{eq:bimpliesadelta=0} gives the contradiction
	\(
	0 = \lim_{\tau\searrow0} \set{\frac{\Delta f(\bx; \tau \bd) + \frac{\tau^2}{2}\bd^\top H \bd}{\frac{1}{2}\tau^2}}
	= \dif{^2} f(\bx|0)(\bd)
	\ge \gamma\norm{\bd}^2=\gamma>0.
	\)	
\end{proof}
\subsection{Application: superlinear convergence of quasi-Newton methods}
\label{subsec:qnmethods}
Let $f$ and $g+G$ be given by \Cref{def:g+G} and consider the corresponding quasi-Newton method \eqref{eq:geqqn} initialized at $(x^0,y^0)$. In this section, we assume the $\mathbf{B}_k$ defined in \eqref{eq:geqqn} take the form
\begin{equation}
\label{eq:qncB=b}	
\mathbf{B}_k = \begin{pmatrix}
B_k & \nabla c(x^k)^\top \\
-\nabla c(x^k) & 0
\end{pmatrix}.
\end{equation}
This choice allows us to relate the optimality conditions for the subproblems \ref{qnsubproblems} defined in \Cref{lem:qnoptimalityiffgeqqn} for solving \ref{theprogram} to the quasi-Newton method of \eqref{eq:geqqn}.
As in \Cref{sec:cvxcomp}, the following is immediate:
\begin{lemma}\label{lem:qnoptimalityiffgeqqn}
	Let $f$ be convex-composite, and let $(x^k,y^k)\in\R^n\times\R^m$ be such that $f$ satisfies \eqref{eq:bcq} at $x^k$, let  $B_k\in\R^{n\times n}$. Then, $(d^k,y^{k+1})\in\R^n\times\R^m$ satisfy the optimality conditions for
	\begin{mini}|l|
		{d\in\R^n}{h(c(x^k) + \nabla c(x^k)d) + \frac{1}{2}d^\top B_kd}{\tag{$\mathbf{Q}_k$}}{\label{qnsubproblems}}
	\end{mini}
	if and only if $(x^{k+1},y^{k+1})$ satisfy the quasi-Newton update for $g+G$ given by \Cref{def:g+G}, with the choice \eqref{eq:qncB=b}. Namely,
	\(
	0 \in g(x^k,y^k) + \mathbf{B}_k\begin{pmatrix}
	x^{k+1} - x^k \\ y^{k+1} - y^k
	\end{pmatrix} + G(x^{k+1},y^{k+1}),
	\)
	where $x^{k+1}:= x^k + d^k$.
\end{lemma}
As a consequence of strong metric subregularity of the linearization $\cG$ given by \eqref{eq:linearizedmapping}, we have the following convergence result:
\begin{theorem} \cite[Dennis-Mor\'e Theorem for Generalized Equations]{dontchev2014implicit}
	%Suppose that $k:X\to Y$ is continuously differentiable in an open and convex neighborhood $U$ of $(\bx,\by)$. Suppose that $K:X\rightrightarrows Y$ is a set-valued mapping with closed graph, and suppose $\set{\cB_k}_{k\in\bN}$ is a sequence of linear and bounded mappings acting from $X$ to $Y$. Consider the generalized equation
	%\begin{equation}
	%	\label{eq:defgeq}
	%	0\in k(x,y) + K(x,y),
	%\end{equation}
	%the class of Newton methods for solving \eqref{eq:defgeq}
	%\begin{equation}
	%	\label{eq:geqnewton}
	%	0\in k(x^k,y^k) + \nabla k(x^k,y^k)\begin{pmatrix}x^{k+1}-x^k \\ y^{k+1}-y^k \end{pmatrix}+ K(x^{k+1},y^{k+1}), \mbox{ for }k\in\bN,
	%\end{equation}
	%and the class of quasi-Newton methods for solving \eqref{eq:defgeq}
	%\begin{equation}
	%	\label{eq:geqqn}
	%	0\in k(x^k,y^k) + \cB_k\begin{pmatrix}x^{k+1}-x^k \\ y^{k+1}-y^k \end{pmatrix} + K(x^{k+1},y^{k+1}), \mbox{ for }k\in\bN.
	%\end{equation}
	\,Let $(\bx,\by)$ be a solution of $g+G\ni0$ given by \Cref{def:g+G} and let $U$ be a neighborhood of $(\bx,\by)$. For some starting point $(x^0,y^0)\in U$ consider a sequence $\set{(x^k,y^k)}_{k\in\bN}$ generated by \eqref{eq:geqqn} which remains in $U$ for all $k\in\bN$ and satisfies $(x^k,y^k)\ne(\bx,\by)$ for all $k\in \bN$. Define $\mathbf{E}_k := \mathbf{B}_k - \nabla g(\bx,\by)$ and $s^k := (x^{k+1}-x^k, y^{k+1}-y^k)$. 
	%If $(x^k,y^k)\to(\bx,\by)$ superlinearly, then
	%\[
	%	\lim_{k\to\infty} \frac{\dist{0}{g(\bx,\by) + E_ks^k + G(x^{k+1},y^{k+1})}}{\norm{s^k}} = 0.
	%\]
	%Conversely, 
	If the linearization mapping $\cG$ given by \eqref{eq:linearizedmapping} is strongly metrically subregular at $(\bx,\by)$ for 0 and the sequence $\set{(x^k,y^k)}_{k\in\bN}$ satisfies
	\(
	(x^k,y^k)\to(\bx,\by) \text{ and } 
	%\lim_{k\to\infty} \frac{\norm{\mathbf{E}_ks^k}}{\norm{s^k}} = 0, 
	\mathbf{E}_ks^k = o(||s^k||)
	\)
	then $(x^k,y^k)\to(\bx,\by)$ superlinearly.
\end{theorem}
\begin{remark}	Suppose the function $g$ is $\cC^1$-smooth and $(x^k,y^k)\to(\bx,\by)$. Then,
	\(
	\mathbf{E}_ks^k = o(||s^k||) \Longleftrightarrow [\mathbf{B}_k-\nabla g(x^k,y^k)]s^k = o(||s^k||).
	%\lim_{k\to\infty} \frac{\norm{\mathbf{E}_ks^k}}{\norm{s^k}} = 0 \Longleftrightarrow \lim_{k\to\infty} \frac{\norm{[\mathbf{B}_k-\nabla g(x^k,y^k)]s^k}}{\norm{s^k}}=0.
	\)
\end{remark}
The following corollary is of algorithmic significance.
\begin{corollary}\label{cor:qnMethod}
	Let $f$ be as in \ref{theprogram}. Suppose $M(\bx)=\set{\by}$ and the second-order sufficient conditions of \Cref{thm:plqsonsc} are satisfied at $\bx$. Then, $(\bx,\by)$ solves $0\in g(\bx,\by)+G(\bx,\by)$. Moreover, there exists a neighborhood $U$ of $(\bx,\by)$ such that if $(x^0,y^0)\in U$, the sequence $\set{(x^k,y^k)}_{k\in\bN}$ generated from the optimality conditions for \ref{qnsubproblems} remains in $U$ with $(x^k,y^k)\neq(\bx,\by)$ for all $k\in\bN$, and
	\[
	(x^k,y^k)\to(\bx,\by)\text{ and }(B_k-\nabla^2(y^kc)(x^k))[x^{k+1}-x^k]=o(||s^k||),
	%\lim_{k\to\infty}\frac{\norm{(B_k-\nabla^2(y^kc)(x^k))[x^{k+1}-x^k]}}{\norm{x^{k+1}-x^k}}=0,
	\]
	then $(x^k,y^k)\to(\bx,\by)$ superlinearly.
\end{corollary}
\begin{remark}
	Consequently, the sufficient conditions for superlinear convergence of quasi-Newton methods require us to choose $B_k$ as an approximation to the Hessian of the Lagrangian $\nabla^2_{xx}L(x^k,y^k)=\nabla^2(y^kc)(x^k)$ in the update direction $x^{k+1}-x^k$ at every iteration.
\end{remark}
\section{Partial Smoothness} \label{sec:partialsmoothness}
The notion of partial smoothness, introduced by Lewis \cite{lewis2002active}, generalizes classical notions of nondegeneracy, strict complementarity, and active constraint identification by illuminating the appropriate underlying manifold geometry of optimization problems. This allows for a more thorough understanding of the convergence behavior of algorithms applied to nonsmooth optimization problems, where solutions lie on well-defined submanifolds of the parameter space on which the function behaves smoothly and off of which it behaves nonsmoothly. Partial smoothness in the context of \ref{theprogram} allows us in \Cref{sec:metricreg} to establish metric regularity properties of the solution mapping.
\begin{definition}
	Define a set $\cM\subset \R^m$ to be a \emph{manifold} of codimension $\ell$ around $\bc\in \R^m$ if $\bc\in\cM$, and there exists an open set $V\subset \R^m$ containing $\bc$ and a $\cC^2$-smooth function $F:V\to\R^\ell$ with surjective derivative throughout $V$ such that
	\(
	\cM\cap V = \{c \in V : F(c) = 0\}.
	\)
	In which case (see \cite{lewis2002active}), the \emph{tangent space} to $\cM$ at $\bc$ is $\tcone{\bar c}{\cM} = \Null{\nabla F(\bar c)}$, the \emph{normal space} to $\cM$ at $\bc$ is $\ncone{\bc}{\cM} = \ran{\nabla F(\bar c)^\top}$, both independent of the choice of $F$. Moreover, the set $\cM$ is Clarke regular at $\bc$, and $\ncone{\bc}{\cM}$ equals the normal cone defined in \eqref{eq:limitingnormals}.
\end{definition}
\begin{definition}[Partial smoothness for closed, convex functions]
	Suppose $h:\R^m\to\eR$ is a closed, proper, convex function and that $\bc\in\cM\subset \R^m$. The function $h$ is \emph{partly smooth} at $\bc$ relative to $\cM$ if $\cM$ is a manifold around $\bar c$ and the following four properties hold:
	\begin{enumerate}[label=(\alph*)]
		\item (restricted smoothness) the restriction $h|_\cM$ is smooth around $\bc$, in that there exists a neighborhood $V$ of $\bc$ and a $\cC^2$-smooth function $g$ defined on $V$ such that $h=g$ on $V\cap\cM$;
		\item (existence of subgradients) at every point $c\in\cM$ close to $\bc,\ \sd h(c)\neq\emptyset$;
		\item (normals and subgradients parallel) $\parr{\sd h(\bc)} = \ncone{\bar c}{\cM};$
		\item (subgradient inner semicontinuity) the subdifferential map $\sd h$ is inner semicontinuous at $\bc$ relative to $\cM$.
	\end{enumerate}
	We say that $h$ is \emph{partly smooth relative to $\cM$} if $\cM$ is a manifold and $h$ is partly smooth at each point in $\cM$ relative to $\cM$.
\end{definition}
\begin{remark}
	By \cite[Proposition 2.4]{lewis2002active}, requiring (a) - (d) in the definition is equivalent to requiring (a), (b), (d), and \emph{normal sharpness}:
	\begin{equation}
	h'(\bc;-w) > -h'(\bc;w),\quad \forall\,w\in\ncone{\bc}{\cM}\setminus\set{0},
	\end{equation}
	and is also equivalent to requiring (a), (b), (d), and \emph{lineality and tangent equality}:
	\begin{equation}
	\label{eq:lineality=tangent}
	\bset{w\in\R^m}{-h'(\bc;w) = h'(\bc;-w)} =: \mathrm{lin}\ h'(\bc;\cdot) = \tcone{\bc}{\cM}.
	\end{equation}
\end{remark}
In the context of the PLQ functions given in 
\Cref{def:plq}, 
a natural choice for the active manifold at a point $\bc\in\dom{h}$ for \ref{theprogram} is the set given by
\begin{equation}\label{eq:manifold}
\cM_{\bc} := \ri{\bigcap_{k\in \cK(\bc)}C_k},
\end{equation}
where $\cK(\bc)$ are the active indices at $\bc$ (see \Cref{def:activeidx}).
The analysis of the manifold $\cM_{\bc}$ requires a more thorough understanding of the structure of $\dom{h}$, which we obtain from the following key result due to Rockafellar and Wets.
\begin{lemma} \cite[Lemma 2.50]{rockafellar_wets_1998}
	\label{plqstructthm}
	Suppose $C$ is a convex set which is the union of a finite collection of polyhedral sets $C_k$. If the polyhedral sets $\set{C_k}_{k=1}^{\cK}$ are represented in terms of a single family of non-constant affine functions $l_i(x) = \ip{a_i}{x}-\alpha_i$ indexed by $i=1,\dotsc,s$, then for each $k$ there is a subset $I_k$ of $\set{1,\dotsc,s}$ such that $C_k = \bset{x}{l_i(x)\le0\text{ for all }i\in I_k}$. Let $I$ denote the set of indices $i\in\set{1,\dotsc,s}$ such that $l_i\le0$ for all $x\in C$. Then,
	\(
	C = \bset{x}{l_i(x)\le0\text{ for all }i\in I}.
	\)
	If $\intr{C}\neq\emptyset$, then $C$ can be written as the union of a finite collection of polyhedral sets $\set{D_j}_{j\in J}$ such that
	\begin{enumerate}[label=(\alph*)]
		\item each set $D_j$ is included in one of the sets $C_k$,
		\item $\intr{D_j}\neq\emptyset$, so $D_j=\cl{\intr{D_j}}$,
		\item $\intr{D_{j_1}}\cap \intr{D_{j_2}}=\emptyset$ when $j_1\neq j_2$.
	\end{enumerate}
\end{lemma}
This result implies that the domain of $h$ has a  finite stratification \cite[Definition 3.1]{drusvyatskiy2014clarke} for which $h$ is a stratifiable function \cite[Definition 3.2]{drusvyatskiy2014clarke}. This stratification is central to our discussion of partial smoothness and is referred to as the Rockafellar-Wets PLQ Representation.
\begin{theorem}[Rockafellar-Wets PLQ Representation]
	\label{cor:plqstructthm}
	Suppose $h$ is piecewise linear-quadratic convex and $\intr{\dom{h}}\neq\emptyset$. Then, 
	without loss of generality, we may assume the polyhedral sets $\set{C_k}_{k=1}^\cK$ defining $h$ are given in terms of a common set of $s>0$ hyperplanes $\cH:=\set{(a_j,\alpha_j)}_{j=1}^s\subset (\R^{m}\setminus\set{0})\times\R$, so that for all $k\in\set{1,\dotsc,\cK}$,
	\[
	C_k = \bset{c}{\ip{\omega_{kj}a_j}{c}\le \omega_{kj}\alpha_{j},\text{ for all } j\in\set{1,\dotsc, s}},
	\]
	with $\omega_{kj}\in\set{\pm1}$,
	\begin{equation}
	\label{eq:RW-Ikc}	
	I_k(c) = \bset{j}{\ip{\omega_{kj}a_j}{c}= \omega_{kj}\alpha_{j}} = \bset{j}{\ip{a_j}{c}= \alpha_{j}}\subset \set{1,\dotsc,s},
	\end{equation}
	and
	\begin{enumerate}[label=(\alph*)]
		\item $\emptyset\neq\intr{C_k}=\bset{c}{
			\begin{aligned}
			& \ip{\omega_{kj}a_j}{c} < \omega_{kj}\alpha_{j}, \text{ for all }j\in\set{1,\dotsc,s_k}\\
			\end{aligned}
		},\text{ for all }k\in\set{1,\dotsc,\cK}$,
		\item $\intr{C_{k_1}}\cap\intr{C_{k_2}}=\emptyset$ when $k_1\neq k_2$. \label{int2}
	\end{enumerate}
	Condition (b) implies that if $c\in C_{k_1}\cap C_{k_2}$, then $c\in \bdry{C_{k_1}}\cap\bdry{C_{k_2}}$ when $k_1\neq k_2$.
\end{theorem}
\begin{proof}
	The proof of the previous lemma shows that for every polyhedron $D_j$ and every $i\in\set{1,\dotsc,s}$, either $l_i(x)\le0$ for all $x\in D_j$ or $l_i(x)\ge0$ for all $x\in D_j$. Therefore each affine function is used in the definition of $D_j$, and $D_j$ is contained entirely within one of the sets $C_k$, relative to which $h$ takes the form $\frac{1}{2}\ip{c}{Q_kc} + \ip{b_k}{c} + \beta_k$.
\end{proof}
%We call the representation of the $\dom{h}$, for any piecewise linear-quadratic $h$ given by \Cref{cor:plqstructthm}, the Rockafellar-Wets PLQ representation.
The basic assumptions employed for the remainder of this section are listed below.
\begin{assumptions}\hfill
	\label{assum:ps}
	\begin{enumerate}[label=(\alph*)]
		\item The function $h$ is PLQ convex with $\dom{h}$ given by the Rockafellar-Wets PLQ representation described in \Cref{cor:plqstructthm},
		\item $\bc\in\dom{h}$ satisfies $\bk:=|\cK(\bc)|\ge2$,
	\end{enumerate}
\end{assumptions} 
\begin{remark}
	Whenever $\cK(\bc)=\set{k_0},\ h$ is continuously differentiable on $\intr{C_{k_0}}$. Therefore, we assume that $\bk\ge2$ and delay the discussion of $\bk=1$ to \Cref{smoothproblems}
\end{remark}
The following lemma further supports the choice for the manifold $\cM_{\bc}$.
\begin{lemma}
	\label{activepolyhedraequal}
	Let $\cM_{\bc}$ be as in \eqref{eq:manifold} and let \Cref{assum:ps} hold. Then, for any $c\in \cM_{\bc},\ \cK(c) = \cK(\bc)$, and so $\cM_{c}=\cM_{\bc}$. Moreover, for any $k\in \cK(\bc)$, the active index sets $I_k(c)$ satisfy $I_k(c)=I_k(\bc)$
\end{lemma}
\begin{proof}
	Suppose $\cK(c) \neq \cK(\bc)$. Since the definition of $\cM_{\bc}$ implies $\cK(\bc)\subset \cK(c)$, there exists $j\in \cK(c')\setminus \cK(\bc)$. By \ref{int2} in \Cref{cor:plqstructthm}, we necessarily have $c\in\bdry{C_j}$.
	\\	
	We first argue the existence of $\epsilon>0$ such that that $(\bc + \epsilon\bB)\cap C_k=\emptyset$ for all $k\not\in \cK(\bc)$. If no such $\epsilon$ exists, since there are only finitely many $k\in K\setminus \cK(\bc)$, there would exist an index $k_0\not\in \cK(\bc)$ and an infinite sequence $c^n\to \bc$ with $\set{c^n}\subset C_{k_0}$. By closedness of the set $C_{k_0}$, $\bc\in C_{k_0}$, which is a contradiction.\\		
	Since $c, \bc \in \cM_{\bc}$, by \cite[Theorem 6.4]{rockafellar2015convex} there exists a $\mu>1$ such that $\tilde c := (1-\mu)\bc + \mu c \in \bigcap_{k\in \cK(\bc)}C_k$.
	\\
	Since $c\in\bdry{C_j}$, there exists a $z\in\intr{C_j}$ sufficiently close to $c$ so that the ray 
	$\calR := \bset{\tilde c + \lambda (z - \tilde{c})}{0\le \lambda}$ meets $\bc + \epsilon\bB$. We consider two cases. To set the stage, for any two points $x,y\in \R^m$, denote the line segment connecting them by $		[x,y]=\bset{(1-\lambda) x + \lambda y}{0\le\lambda\le1}.$
\\	
	Case 1. There is a point $x\in\calR\cap(\bc + \epsilon\bB)\cap C$. Then $z\in[\tilde{c},x]\subset C_k$ for some $k\in \cK(\bc)$. But then $z\in(\intr{C_j})\cap C_k$, a contradiction. 
\\
		Case 2. We have $\calR\cap(\bc + \epsilon\bB)\cap C=\emptyset$. Then there 
		is a point $x\in (\bc+ \eps\bB)\setminus C$ such that $z\in[\tilde c,x]$. Since $x\notin C$, there is a first point, 
		which we denote by $\hat z$, in
		$C_j$ on this line segment as one moves from $x$ to $\tilde c$.		
		Then the line segment $[\hat z, \bc]\subset C$. The point $\hat z$ is not on the line segment 
		$[\tilde c,\bc]$ since then both $c'$ and $z$ would be on the line segment $[\tilde c, \bc]$ and so
		$\intr{C_j}\cap\bdry{C_k}\ne\emptyset$ for some $k\in \cK(\bc)$, a contradiction.
		%ray connecting $\tilde c$ to $\bc$. If $\hat z$ is on the ray, then both $c'$ and $z$ would be on the line segment $[\tilde c, \bc]$. But that line segment is contained in $\cM_{\bc}$, so $\intr{C_j}$ would intersect $\bdry{C_k}$ for some $k\in \cK(\bc)$. Therefore $\hat z$ is not on the ray connecting $\tilde c$ to $\bc$.
		Consequently, the points $\tilde c, \bc$ and $\hat z$ are not all collinear and hence form a 
		triangle inside of $C$.
		%lie in a 2-dimensional affine set. 
		Let $\tilde z$ be on the boundary of $\bc + \epsilon\bB$ and on the line segment $[\hat z,\bc]$. Then the line segment $[\tilde z,\tilde c]$ passes through $\intr{C_j}$. This is again a contradiction.
\\
Therefore, no such $c$ exists, and $\cK(c)=\cK(\bc)$ for all $c\in \cM_{\bc}$.
	\\
	For the second claim, suppose there exists $k\in \cK(\bc),\ c\in\cM_{\bc}$ and $j\in\set{1,\dotsc,s}$ with
	\begin{equation}
	\label{eq:activecontradiction}
	\ip{c}{\omega_{kj}a_j} < \omega_{kj}\alpha_{j}\text{ and }\ip{\bc}{\omega_{kj}a_j} = \omega_{kj}\alpha_{j}.
	\end{equation}
	Again by \cite[Theorem 6.4]{rockafellar2015convex}, we may choose $\mu>1$ so that $\mu \bc + (1-\mu)c \in \cM_{\bc}$. In particular, $\mu \bc + (1-\mu)c\in C_k$. But writing $\mu = 1 + \epsilon$ with $\epsilon> 0$ gives the contradiction
	\begin{align*}
	\omega_{kj}\alpha_{j} &\ge \ip{\mu \bc + (1-\mu)c}{\omega_{kj}a_j} \\
	&= (1 + \epsilon)\ip{\bc}{\omega_{kj}a_j} - \epsilon \ip{c}{\omega_{kj}a_j}
	%&= \omega_{kj}\alpha_{j} + \epsilon\left(\ip{\bc}{\omega_{kj}a_j} - \ip{c}{\omega_{kj}a_j}\right) \\
	>\omega_{kj}\alpha_{j} \text{ by }\eqref{eq:activecontradiction}.
	\end{align*}
	Therefore $I_k(\bc) \subset I_k(c)$. Reversing the roles of $c$ and $\bc$ in \eqref{eq:activecontradiction} gives the other inclusion.
\end{proof}
The previous lemma tells us distinct points $c,c'\in\cM_{\bc}$ have the same active indices $\cK(c)$ and $\cK(c')$. Moreover, for any active polyhedron $C_k$, the active hyperplanes for that polyhedron, $I_k(c)$ and $I_k(c')$, at $c$ and $c'$ are the same. This observation offers a global description of  $\cM_{\bc}$ in terms of the active hyperplanes at $\bc$ alone.
\begin{lemma}
	Let $\cM_{\bc}$ be as in \eqref{eq:manifold}, and let \Cref{assum:ps} hold. Then,
	\label{Mfacialrep}
	\[
	\cM_{\bc} = \bset{c}{
		\begin{aligned}
		\ip{c}{a_j} &= \alpha_{j}\text{ for all }k\in \cK(\bc), j\in I_k(\bc) \\
		\ip{c}{\omega_{kj}a_j} &< \omega_{kj}\alpha_{j}\text{ for all }k\in \cK(\bc), j\not\in I_k(\bc)
		\end{aligned}
	}.
	\]
	In particular, $I_{k_1}(c) = I_{k_2}(c)$ for all $c\in \cM_{\bc}$ and $k_1,k_2\in \cK(\bc)$. Moreover, for any $k\in \cK(\bc)$ and $c\in\cM_{\bc}$,
	\(
	\tcone{c}{\cM_{\bc}} = \Null{A_k(\bc)^\top}, \text{ and } \ncone{c}{\cM_{\bc}} = \ran{A_k(\bc)},
	\)
	where $A_{k}(\bc)$ is the matrix whose columns are the gradients of the active constraints at $\bc\in C_{\bk}$ in some ordering.
\end{lemma}
\begin{remark}\label{rem:ccbarkj}
	By \Cref{activepolyhedraequal} and \Cref{Mfacialrep}, for all $c\in\cM_{\bc},\ k\in \cK(\bc)$, and $j\in \cK(c)$,
	\(
	\Ran{A_k(\bc)} = \Ran{A_j(c)}.
	\)
	This observation becomes important in a structural definition to follow.
\end{remark}
\begin{proof}
	Define
	\[
	\cC_1:= \bigcap_{k\in \cK(\bc)} C_k, \quad \cC_2:=\bset{c}{
		\begin{aligned}
		\ip{c}{\omega_{kj}a_j} &= \omega_{kj}\alpha_{j}\text{ for all }k\in \cK(\bc), j\in I_k(\bc) \\
		\ip{c}{\omega_{kj}a_j} &\le \omega_{kj}\alpha_{j}\text{ for all }k\in \cK(\bc), j\not\in I_k(\bc)
		\end{aligned}
	}.
	\] 
	We aim to show $\ri{\cC_1}\supset\ri{\cC_2}$. For $k\in \cK(\bc)$ and $j\in I_k(\bc)$ define $\cC_{k,j}:= \bset{c}{\ip{c}{\omega_{kj}a_j}=\omega_{kj}\alpha_{j}}$, and for $k\in \cK(\bc)$ and $j\not\in I_k(\bc)$, let $\cD_{k,j}:=\bset{c}{\ip{c}{\omega_{kj}a_j}\le\omega_{kj}\alpha_{j}}$. Then by defintion of $I_k(\bc)$,
	\[
	\bc \in \bigcap_{\substack{k\in \cK(\bc)\\ j\in I_k(\bc)}}\ri{\cC_{kj}}\cap \bigcap_{\substack{k\in \cK(\bc)\\ j\not\in I_k(\bc)}}\ri{\cD_{kj}},
	\]
	so \cite[Theorem 6.5]{rockafellar2015convex} gives
	\[
	\ri{\cC_2
	} = \bset{c}{
		\begin{aligned}
		\ip{c}{\omega_{kj}a_j} &= \omega_{kj}\alpha_{j}\text{ for all }k\in \cK(\bc), j\in I_k(\bc) \\
		\ip{c}{\omega_{kj}a_j} &< \omega_{kj}\alpha_{j}\text{ for all }k\in \cK(\bc), j\not\in I_k(\bc)
		\end{aligned}
	}.
	\]
	Moreover, $\cC_1\supset \cC_2$ with $\cC_2$ not entirely contained within the relative boundary of $\cC_1$ because $\bc\in \cC_2\cap\cM_{\bc}$. By \cite[Corollary 6.5.2]{rockafellar2015convex}, $\cM_{\bc}:=\ri{\cC_1} \supset \ri{\cC_2}$. \Cref{activepolyhedraequal} shows $\cM_{\bc}:=\ri{\cC_1}\subset \ri{\cC_2}$ because $I_k(c)=I_k(\bc)$ throughout $\cM_{\bc}$. 
	\\	
	For the second claim, the structure of $\cM_{\bc}$ implies that if $\ip{c}{\omega_{k_1j}a_j} = \omega_{k_1j}\alpha_{j}$ for some $k_1\in \cK(\bc)$, then $\ip{c}{\omega_{k_2j}a_j} = \omega_{k_2j}\alpha_{j}$ for any other $k_2\in \cK(\bc)$ as $\omega_{kj}\in \set{\pm1}$. Hence $I_{k_2}(c)\supset I_{k_1}(c)$, and this argument is symmetric in $k_1$ and $k_2$. 
	
	The tangent and normal cone formulas hold throughout $\cM_{\bc}$ by \Cref{thm:polyhedraltconencone}.
\end{proof}
Based on \Cref{Mfacialrep} and \Cref{rem:ccbarkj}, we now establish the notational tools required for our analysis.
\begin{definition}
	\label{def:AmatrixPs}
	Let $\cM_{\bc}$ be as in \eqref{eq:manifold}, and let \Cref{assum:ps} hold. Define $A_{\bk}(c)$ to be the matrix whose columns are the gradients of the active constraints at $c\in C_{\bk}$ in some ordering.	By \Cref{cor:plqstructthm} and \Cref{Mfacialrep}, without loss of generality, we can define
	\(
	A := A_{\bk}(c) \text{ independent of the choice of } c\in \cM_{\bc},
	\)
	and for any $j\in \set{1,\dotsc,\bk}$, there exists a diagonal matrix $P_j$ with entries $\pm1$ on the diagonal such that
	\begin{equation}
	\label{eq:APj=A_k(c)}
	AP_j = A_{k_j}(c) \text{ independent of }c\in\cM_{\bc}.
	\end{equation}
	We let $\ell$ be the common number of columns $\ell:=|I_k(\bc)|=|I_{k'}(\bc)|$ for all $k,k'\in \cK(\bc)$, so that
	\(
	A \in \R^{m\times \ell}, P_j\in \R^{\ell \times \ell}, P_{\bk} = I_\ell,
	\)
	and define the following block matrices $\hat{\cQ} := \diag(Q_k), \hat\cA := \diag{AP_j}$
	\begin{align}\small
	\cA := \begin{pmatrix}\label{eq:sdcoeffeq}
	(1-\bar k)AP_1 & AP_2 & \cdots & A\\
	AP_1 & (1-\bar k)AP_2 & \cdots & A \\
	\vdots & \ddots & \ddots& \vdots \\
	AP_1 & AP_2 & \cdots & (1-\bar k)A
	\end{pmatrix},\
	\cQ := \begin{bmatrix}Q_{k_1} \\ Q_{k_2}\\\vdots \\Q_{k_{\bar k}}\end{bmatrix},\ \mathcal{B} := \begin{bmatrix} b_{k_1}\\ b_{k_2}\\ \vdots \\ b_{k_{\bar k}}\end{bmatrix},\ J:=\begin{bmatrix}I_m \\ I_m \\ \vdots \\ I_m\end{bmatrix}%\in \R^{(\bar k m)\times m}.
	\end{align}
	and averaged quantities
	\begin{align*}
	\begin{aligned}
	\bar Q &= (1/\bk)J^\top\hat\cQ J, &\bar A &= (1/\bk) J^\top\hat\cA, &\bar b &= (1/\bk)J^\top\cB, &\lambda_0(\bc) &= \bar Q\bc + \bar b. 
	\end{aligned}
	\end{align*}
\end{definition}
In a fashion similar to the \emph{structure functional} approach of
\cite{womersley1985local, osborne2001simplicial, overton1994towards}, we give a formula for the subdifferential in terms of the active manifold structure previously laid out.
\begin{lemma}
	\label{sdaveraging}
	Let $\cM_{\bc}$ be as in \eqref{eq:manifold}, let \Cref{assum:ps} hold, and recall the notation of \Cref{def:AmatrixPs}. For any $c\in \cM_{\bc},\ \sd h(c)$ can be given by two equivalent formulations:
	\begin{align}
	\label{eq:plqsdsystem}
	%\label{eq:sdaveraging}
	\sd h(c) &= \bset{y}{
		\begin{aligned}
		&\exists\,\mu=(\mu_1^\top,\dotsc,\mu_{\bk}^\top)^\top\ge0\\
		&\text{such that }
		Jy = \cQ c + \cB + \hat\cA\mu\\
		\end{aligned}}
	= \lambda_0(c) + \bar A\cU(c)	
	, 
	\end{align} 
	where
	\begin{equation}
	\label{eq:U(c)}
	\cU(c):=\bset{\mu\geq0}{
		\cA \mu = \bar k\left[\cQ c + \cB - J(\bar Q c + \bar b)\right]
	}.
	\end{equation}
\end{lemma}
\begin{proof}
	By \eqref{eq:plqsubdiff} and \Cref{activepolyhedraequal}, $y\in\sd h(c)$ if and only if $y\in Q_{k_j}c + b_{k_j} + \ncone{c}{C_{k_j}}$ for all $j\in \set{1,\dotsc,\bk}$. In terms of the active indices at $c$ for the polyhedron $C_{k_j}$, \eqref{eq:ckncone} and \eqref{eq:APj=A_k(c)} imply
	\[
	y = Q_{k_j}c + b_{k_j} + AP_j\mu_j, \text{ where }j\in \set{1,\dotsc,\bar k}, \mu_j\ge0.
	\]
	Hence $y\in \sd h(c)$ if and only if there exists $\mu=(\mu_1^\top,\dotsc,\mu_{\bk}^\top)$ such that $(y, \mu)$ satisfies the system
	\begin{equation*}
	Jy = \cQ c + \cB + \hat\cA\mu,\quad
	\mu = (\mu_1^\top,\dotsc,\mu_{\bk}^\top)^\top\ge0.
	\end{equation*}
	Since $J^\top J = \bar k I_m$, multiplying both sides of the first equation in \eqref{eq:plqsdsystem} by $(1/\bar k)J^\top$ gives $y = \bar Q c + \bar b + \bar A\mu$, where $\mu$ satisfies 
	\[
	\bar Q c + \bar b + \bar A \mu= AP_j\mu_j + Q_{k_j}c + b_{k_j}, \text{for all }j\in\set{1,\dotsc,\bar k}, \mu\ge0.
	\]
	The set of $\mu$ that satisfy the display defines membership in $\cU(c)$, so $\sd h(c) = \lambda_0(c)+\bar A \cU(c)$.
\end{proof}
The notion of nondegeneracy that we use imposes linear independence of the columns of $A$. 
\begin{definition}[Nondegeneracy]
	\label{nondegeneracy}
	Let $\cM_{\bc}$ be as in \eqref{eq:manifold}, let \Cref{assum:ps} hold, and recall the notation of \Cref{def:AmatrixPs}. We say that $\cM_{\bc}$ satisfies the \emph{nondegeneracy condition} if $\Null{A}=\set{0}$.
\end{definition}
Nondegenercy yields a uniqueness property of the multipliers $\mu\in \cU(c)$.
\begin{lemma}
	\label{uniquemugivenv}
	Let $\cM_{\bc}$ be as in \eqref{eq:manifold}, let \Cref{assum:ps} hold, and recall the notation of \Cref{def:AmatrixPs}. Suppose $\cM_{\bc}$ satisfies the nondegeneracy condition of \Cref{nondegeneracy}, $c\in \cM_{\bc}$, and $y \in \sd h(c)$. Then, there is a unique $\mu\in \cU(c)$, given by $\mu(c, y)_j=P_j(A^\top A)^{-1}A^\top(y - (Q_{k_j}c + b_{k_j})), j\in\set{1,\dotsc,\bk}$ so that $y = \lambda_0(c) + \bar A \mu(c,y)$.
\end{lemma} 
\begin{proof}
	For any $j\in\set{1,\dotsc,\bar k}$, \Cref{sdaveraging} implies there exists $\mu_j\geq0$ such that $y = Q_{k_j}c + b_{k_j} + AP_j\mu_j$. Nondegeneracy implies $\mu_j$ is given uniquely by the equation
	\(
	\mu(c,y)_j = P_j(A^\top A)^{-1}A^\top(y - (Q_{k_j}c + b_{k_j})).
	\)
\end{proof}
A corresponding notion of strict complementarity is provided by the next lemma.
\begin{lemma}
	\label{risdiffstrict}
	Let $\cM_{\bc}$ be as in \eqref{eq:manifold}, let \Cref{assum:ps} hold, and recall the notation of \Cref{def:AmatrixPs}. Suppose $c\in\cM_{\bc}$ and $\ri{\sd h(c)}\ne\emptyset$. Then $y\in \ri{\sd h(c)}$ if and only if $\mu(c,y)_i>0$ for all $i\in\set{1,\dotsc,\bk}$.
\end{lemma}
\begin{proof}
	By \cite[Theorem 6.4]{rockafellar2015convex}, $y\in \ri{\sd h(c)}$ if and only if for all $y'\in \sd h(c)$, there exists $t>1$ so that $t y + (1-t)y' \in \sd h(c)$. Choose a $y'\in\sd h(c)$ with $y'\neq y$.
	\\
	$(\Rightarrow)$ If there exists $i_0\in\set{1,\dotsc,\bar k}$ and $j\in\set{1,\dotsc,\ell}$, with $(\mu(c,y)_{i_0})_j = 0$, then, by \eqref{eq:plqsdsystem},
	\[
	\sd h(c)\ni ty + (1-t)y' = Q_{i_0}c + b_{i_0} + AP_{i_0}[t\mu(c,y)_{i_0} + (1-t) \mu(c,y')_{i_0}]. 
	\]
	By \Cref{uniquemugivenv}, $\mu(c, ty + (1-t)y')_{i_0}=t\mu(c,y)_{i_0} + (1-t) \mu(c,y')_{i_0}$. By assumption, the right-hand side has its $j$th component is negative for all $t>1$, a contradiction.
	\\
	$(\Leftarrow)$ We must show there exists $\epsilon>0$ such that if $t := 1 + \epsilon$ then $t\mu(c,y)_{i_0} + (1-t) \mu(c,y')_{i_0} > 0$. After rearranging, this is equivalent to finding $\epsilon>0$ so that
	\(
	\mu(c,y)_{i_0} + \epsilon[\mu(c,y)_{i_0} - \mu(c,y')_{i_0}] > 0.
	\)
	If $\mu(c,y)_{i_0} - \mu(c,y')_{i_0}\ge0$, the claim is immediate. Otherwise, we choose $\epsilon$ via
	\[
	0<\epsilon < \min{\bset{\frac{(\mu(c,y)_{i_0})_j}{(\mu(c,y')_{i_0})_j-(\mu(c,y)_{i_0})_j}}{(\mu(c,y)_{i_0})_j - (\mu(c,y')_{i_0})_j < 0,j\in \set{1,\dotsc,\ell}}}.
	\] 
	Then $y\in \ri{\sd h(c)}$.
\end{proof}
However, a weaker notion of strict complementarity in conjunction with nondegeneracy suffices to show that $\ri{\sd h(c)}\neq\emptyset$ throughout $\cM_{\bc}$.
\begin{definition}[$k$-strict complementarity]
	\label{def:kstrictcc}
	Let $\cM_{\bc}$ be as in \eqref{eq:manifold}, let \Cref{assum:ps} hold, and recall the notation of \Cref{def:AmatrixPs}. We say \emph{$k$-strict complementarity} holds at $(c, y)$ for $\mu=(\mu_1^\top,\dotsc,\mu_{\bk}^\top)^\top$ if
	\begin{enumerate}[label=(\alph*)]
		\item $c\in\cM_{\bc},\ y\in \sd h(c)$,
		\item There exists $k\in \cK(\bc)$ with $\mu_k>0$,
		\item Whenever there exists $j\in \cK(c)\setminus\set{k}$ and $i\in \set{1,\dotsc,\ell}$ with $(\mu_j)_i=0$, then the scalars $(P_{j'})_{ii}=1$ for all $j'\in \cK(c)$, \label{kstrict-sameside}
		\item $(y,\mu)$ satisfies \eqref{eq:plqsdsystem}.
	\end{enumerate}
\end{definition}
\begin{remark}
	When $k$-strict complementarity holds at a pair $(c,y)$ and an index $j$ satisfies \ref{kstrict-sameside}, the active polyhedra $\set{C_k}_{k\in \cK(\bc)}$ are all within the same closed half-space of the corresponding hyperplane. Also observe that $y\in\ri{\sd h(c)}$ implies $k$-strict complementarity at $(c,y)$.
\end{remark}
A requirement of partial smoothness is that the normal space to $\cM_{\bc}$ and $\parr{\sd h(c)}$ are equal. The nondegeneracy condition allows us to describe $\parr{\sd h(c)}$ using the vectors in $\cU(c)$ rather than the subgradients in $\sd h(c)$.
\begin{lemma}
	Let $\cM_{\bc}$ be as in \eqref{eq:manifold}, let \Cref{assum:ps} hold, and recall the notation of \Cref{def:AmatrixPs}. Suppose $\cM_{\bc}$ satisfies the nondegeneracy condition. Then, for any $c\in\cM_{\bc}$,
	\begin{equation}
	\label{eq:subspaceiff}
	\parr{\sd h(c)} =\Ran{A} \Longleftrightarrow \parr{\cU(c)}=\Null{\cA}.
	\end{equation}
\end{lemma}
\begin{proof}
	By \Cref{Mfacialrep}, $\ncone{c}{\cM_{\bc}} = \Ran{A}$, and by \Cref{sdaveraging}, $\sd h(c) = \lambda_0(c) + \bar A \cU(c)$. The system of linear equations \eqref{eq:U(c)} in $\cU(c)$ has coefficient matrix $\cA$ defined in \eqref{eq:sdcoeffeq} which is block-circulant and can be block row-reduced to
	\begin{equation}
	\label{eq:rrefcirculant}	
	\begin{pmatrix}
	AP_1 & 0 & 0 & \cdots & -A\\
	0 & AP_2 & 0 &\cdots & -A \\
	\vdots & \ddots & \ddots & \ddots& \vdots \\
	0 & \cdots & 0 & AP_{\bar k - 1} & -A \\
	0 & \cdots & \cdots & 0 & 0
	\end{pmatrix}.
	\end{equation}
	We now compute $\Null{\cA}$. Suppose $\mu=(\mu_1^\top,\dotsc,\mu_{\bk}^\top)^\top\in \Null{\cA}$. Then \eqref{eq:rrefcirculant} and nondegeneracy imply that $\mu\in\Null{\cA}$ if and only if $\mu_j = P_j\mu_{\bar k}$ for all $j\in\set{1,\dotsc,\bar k - 1}$, i.e.,
	\begin{equation}
	\Null{\cA} = \bset{\begin{pmatrix}
		P_1\mu_{\bar k} \\ \vdots \\ P_{\bar k - 1}\mu_{\bar k} \\ \mu_{\bar k}
		\end{pmatrix}}{\mu_{\bar k}\in \R^\ell},
	\text{with basis }
	\label{eq:zeta}
	\bset{\begin{pmatrix}
		P_1e_p \\ \vdots \\ P_{\bar k - 1}e_p \\ e_p
		\end{pmatrix}}{p\in \set{1,\dotsc,\ell}} =: \set{\zeta_1,\dotsc,\zeta_\ell}.
	\end{equation}
	By \eqref{eq:U(c)},
	\begin{equation}
	\label{eq:U(c) subset Null(cA)}
	\parr{\cU(c)}:=\R(\cU(c) - \cU(c)) \subset \Null{\cA},
	\end{equation}	
	and since $\bar A = \frac{1}{\bar k}\begin{bmatrix}
	AP_1 & \cdots & AP_{\bar k - 1} & A
	\end{bmatrix}$, \eqref{eq:plqsdsystem} implies
	\begin{align*}
	\parr{\sd h(c)} = \parr{\bar A\cU(c)}
	= \bar A \parr{\cU(c)}
	\subset \bar A \Null{\cA}= \bset{A \mu_k}{\mu_k\in\R^\ell} = \Ran{A},
	\end{align*}
	so $(\Leftarrow)$ in \eqref{eq:subspaceiff} is clear as ``$\subset$" becomes an equation.
	For $(\Rightarrow)$, suppose strict containment: $\parr{\cU(c)}\subsetneq\Null{\cA}$. Then there exists $p\in\set{1,\dotsc,\ell}$ such that $\zeta_p\not\in \parr{\cU(c)}$. This implies that the $p$th column of $A$ is not in $\parr{\sd h(c)}$ which we have assumed equal to $\Ran{A}$. This contradiction establishes \eqref{eq:subspaceiff}.
\end{proof}
We now show that nondegeneracy and $k$-strict complementarity together imply that the normal space and subdifferential are parallel.
\begin{lemma}
	\label{pslemma}
	Let $\cM_{\bc}$ be as in \eqref{eq:manifold}, let \Cref{assum:ps} hold, and recall the notation of \Cref{def:AmatrixPs}. Suppose $\cM_{\bc}$ satisfies the nondegeneracy condition, and the $k$-strict complementarity of \Cref{def:kstrictcc} holds at $(c, y)$ for $\mu$.
	Then,
	\begin{equation}
	\label{eq:goal}
	\parr{\sd h(c)}=\ncone{c}{\cM_{\bc}},
	\end{equation}
	where it is shown in \Cref{Mfacialrep}  that $\ncone{c}{\cM_{\bc}}=\Ran{A}$.
	Moreover, \eqref{eq:goal} holds throughout $\cM_{\bc}$, and $\sd h$ is inner semicontinuous relative to $\cM_{\bc}$.
\end{lemma}
\begin{proof}
	We first show that a sufficient condition to guarantee the right-hand side of \eqref{eq:subspaceiff} is $(c,v)$ satisfying the $k$-strict complementarity condition of \Cref{def:kstrictcc} for $\mu\in \cU(c)$. To see this note that, by relabeling the active polyhedral sets if necessary, we can assume without loss of generality that the index $k$ in $k$-strict complementarity is $\bk$. Let $p\in\set{1,\dotsc,\ell},\ t\in \R$, and consider the step given by $\mu+t\zeta_p$, where $\zeta_p$ is the $p$th basis element of $\Null{\cA}$ given in \eqref{eq:zeta}, i.e.,
	\begin{equation}
	\label{eq:mu+tzeta}	
	\mu + t \zeta_p := \begin{pmatrix}
	\mu_1 \\ \vdots \\ \mu_{\bar k - 1} \\ \mu_{\bar k}
	\end{pmatrix} + t \begin{pmatrix}
	P_1e_p \\ \vdots \\ P_{\bar k - 1}e_p \\ e_p
	\end{pmatrix},
	\end{equation}
	We consider two cases. If, for all $j\in\set{1,\dotsc,\bar k}$, $(\mu_j)_p > 0$, then for sufficiently small $t$,
	\(
	\mu + t\zeta_p\ge0,\text{ and } \cA(\mu+t\zeta_p) = \cA\mu.
	\)
	That is, both $\mu\in \cU(c)$ and $\mu + t\zeta_p \in \cU(c)$, which implies $\zeta_p\in \parr{\cU(c)}$. Otherwise, there exists $j\in\set{1,\dotsc,\bk}$ with $(\mu_j)_p=0$. By part (c) of $k$-strict complementarity, the scalars $P_{j'}e_p=1$ for all $j'\in \set{1,\dotsc,\bar k}$, so repeating the previous argument with $t>0$ gives $\zeta_p\in \parr{\cU(c)}$. Since $p\in\set{1,\dotsc,\ell}$ was arbitrary, $k$-strict complementarity is a sufficient condition guaranteeing $\parr{\cU(c)}=\Null{\cA}$. 
	\\
	This argument shows, under nondegeneracy, that
	\begin{equation}\label{eq:kstrictimpliesstrict}
	k\text{-strict complementarity at }(c,y)\text{ for }\mu \Longrightarrow \ri{\sd h(c)}\neq\emptyset,
	\end{equation}
	because, given any $\mu\in \cU(c)$, the fact that $\parr{\cU(c)}=\Null{\cA}$ together with \eqref{eq:plqsdsystem} implies there exists a strictly positive $\tilde \mu\in \cU(c)$ and a $\tilde y\in \sd h(c)$ given by $\tilde{y} = \lambda_0(c) + \bar A \tilde\mu$, with $\mu(c,\tilde{y}) = \tilde \mu$. By \Cref{risdiffstrict}, $\tilde y\in\ri{\sd h(c)}$.
	\\
	We now argue that if, for some $c\in\cM_{\bc},\ y\in\sd h(c),\ k$-strict complementarity holds at $(c,y)$ for $\mu$, then $\ri{\sd h(c)}\neq\emptyset$ throughout $\cM_{\bc}$ . This will imply \eqref{eq:goal} holds throughout $\cM_{\bc}$ as well. By \eqref{eq:kstrictimpliesstrict}, suppose $y\in\ri{\sd h(c)}$ so that $\mu(c,y)>0$ by \Cref{risdiffstrict}.
	\\	
	Choose any other $c'\in \cM_{\bc}$. Since $\cM_{\bc}$ is relatively open, there exists $c''\in\cM_{\bc}$ and $\lambda\in(0,1)$ so that $c'=\lambda c + (1-\lambda)c''$. Let $y''\in \sd h(c'')$.  By \Cref{uniquemugivenv}, there exists a unique vector $\mu(c'',y'')$ associated with $(c'',y'')$. Since $c,\ c''\in \cM_{\bc}$ and $\mu(c,y)>0$,  $\lambda \mu(c', y') + (1-\lambda)\mu(c,y)>0$. It follows from \eqref{eq:plqsdsystem} that for all $j\in \set{1,\dotsc,\bar k}$ and $\lambda\in(0,1)$,
	\begin{equation}
	\label{eq:strictcomplementarityopen}
	\lambda y + (1-\lambda) y'' = Q_{k_j}c' + b_{k_j} + AP_j(\lambda \mu(c, y) + (1-\lambda)\mu(c'',y'')).
	\end{equation}
	Define $y' := \lambda y + (1-\lambda) y''$. Then \eqref{eq:strictcomplementarityopen} implies that the equations \eqref{eq:plqsdsystem} defining membership $y' \in \sd h(c')$ are satisfied, with $\mu(c',y')=\lambda \mu(c, y) + (1-\lambda)\mu(c'', y'')>0$, so $y'\in\ri{\sd h(c')}$ by \Cref{risdiffstrict}. Since $c'\in \cM_{\bc}$ was arbitrary, $\ri{\sd h(c)}\neq\emptyset$ for all $\cM_{\bc}$.
	\\
	We lastly establish $\sd h(c)$ is inner semicontinuous relative to $\cM_{\bc}$. The previous paragraph and \eqref{eq:strictcomplementarityopen} showed $\sd h|_{\cM_{\bc}}$ is graph-convex. By defining $S(c) = \sd h(c)$ for $c\in \cM_{\bc}$ and $S(c)=\emptyset$ otherwise and noting the convex sets $\set{c}$ and $\cM_{\bc}$ cannot be separated, \cite[Theorem 5.9(b)]{rockafellar_wets_1998} gives inner semicontinuity of $\sd h$ at all $c\in \cM_{\bc}$ relative to $\cM_{\bc}$.
\end{proof}
The main result of this section shows that partial smoothness follows from nondegeneracy and $k$-strict complementarity.
\begin{theorem}
	\label{thm:ps}
	Let $\cM_{\bc}$ be as in \eqref{eq:manifold}, let \Cref{assum:ps} hold, and recall the notation of \Cref{def:AmatrixPs}. Suppose $\cM_{\bc}$ satisfies the nondegeneracy condition, and $c\in\cM_{\bc}$ and $y\in\sd h(c)$ are such that $(c,y)$ satisfies the $k$-strict complementarity condition of \Cref{def:kstrictcc}. Then $h$ is partly smooth relative to $\cM_{\bc}$.
\end{theorem}
\begin{proof}
	By definition of $\cM_{\bc}$, for any $k\in \cK(\bc)$ and any $c\in \cM_{\bc},\ h(c) = \frac{1}{2}\ip{c}{Q_kc} + \ip{b_k}{c} + \beta_k$, so $h|_{\cM_{\bc}}$ is smooth. By \Cref{plqformulae}, $\dom{\sd h} = \dom{h}\supset \cM_{\bc}$, so existence of subgradients holds throughout $\cM_{\bc}$ as well. The normal cone and subdifferential being parallel along with subdifferential inner semicontinuity relative to $\cM_{\bc}$ are the content of \Cref{pslemma}.
\end{proof}
\begin{remark}\label{remark:nondegenandtransversality}
	Observe that if the hypotheses of \Cref{thm:ps} are satisfied, the assumption that $f$ satisfies \eqref{eq:generaltransversality} at $\bx$ is equivalent to requiring
	\begin{equation}
	\label{eq:pstransversality}
	\Null{\nabla c(\bar{x})^\top} \cap \ran{A} = \{0\}.
	\end{equation}
	This condition and the nondegeneracy condition imply the $n\times \ell$ matrix $\nabla c(\bx)^\top A$ has full rank equal to $\ell\leq n$, i.e., $\Null{\nabla c(\bx)^\top A} = \set{0}$.
\end{remark}
We now show the assumptions of \Cref{thm:ps} allow us to write the cone of non-ascent directions as a subspace at strictly critical points.
\begin{lemma}[Non-ascent directions]
	\label{lem:D(x)andA^Tc'(x)}
	Let $\cM_{\bc}$ be as in \eqref{eq:manifold}, let \Cref{assum:ps} hold, and recall the notation of \Cref{def:AmatrixPs}. Suppose $f$ satisfies \eqref{eq:bcq} at $\bx,\ \by\in M(\bx)$, and $\bc:=c(\bx)$. Then, $D(\bx) \supset \Null{A^\top \nabla c(\bx)}$. If, in addition, $f$ satisfies \eqref{eq:strictcriticality} at $\bx$ for $\by$ and $\cM_{\bc}$ satisfies the nondegeneracy condition, then $D(\bx) \subset \Null{A^\top \nabla c(\bx)}$.
	\\
	%Suppose $f$ satisfies \eqref{eq:strictcriticality} at $\bx$ for $\by,\ c(\bx)\in\cM_{\bc},$ and $\cM_{\bc}$ satisfies the nondegeneracy condition. Let $D(\bx)$ be given by \eqref{eq:nonascentcone}. Then $D(\bx) \subset \Null{A^\top \nabla c(\bx)}$. If, in addition, $\bx$ is a local minimizer, then $D(\bx) \supset \Null{A^\top \nabla c(\bx)}$.
\end{lemma}
\begin{proof}
	Since $f$ satisfies \eqref{eq:bcq} at $\bx$, \Cref{thm:cvxcompfonc} gives
	\(
	D(\bx) = \bset{d\in \R^n}{h'(c(\bx);\nabla c(\bx)d)\le0}.
	\)
	$(\supset)$ Since $\by\in M(\bx)$, by \eqref{eq:plqsdsystem}, there exists $\bmu\in \cU(\bc)$ so that $J\by = \cQ\bc + \cB + \hat\cA \bmu$. Then, for any $j\in \set{1,\dotsc,\bk}$,
	\begin{align*}
	D(\bx) %&= \bset{d\in \R^n}{h'(c(\bx);\nabla c(\bx)d)\le0} \\
	&= \bigcup_{j=1}^{\bk} \bset{d}
	{
		\begin{aligned}
		& \ip{Q_{k_j}\bar c + b_{k_j}}{\nabla c(\bx)d} \le 0 \\
		& P_jA^\top \nabla c(\bx)d\leq0
		\end{aligned}
	}\quad\text{ by \eqref{eq:plqnonascent}, \Cref{def:AmatrixPs}}\\
	&= \bigcup_{j=1}^{\bk} \bset{d}
	{
		\begin{aligned}
		& \ip{\by-AP_j\bmu_j}{\nabla c(\bx)d} \le 0 \\
		& P_jA^\top \nabla c(\bx)d\leq0
		\end{aligned}
	}\quad\text{ since }\by\in M(\bx)\\
	&= \bigcup_{j=1}^{\bk} \bset{d}
	{
		\begin{aligned}
		& \ip{\bmu_j}{P_jA^\top\nabla c(\bx)d} \ge 0 \\
		& P_jA^\top \nabla c(\bx)d\leq0
		\end{aligned}
	}.
	\end{align*}
	The inclusion follows.
	%	and $f$ satisfies \eqref{eq:bcq} is satisfied, \Cref{thm:fonc} implies
	%	\[
	%	D(\bx) = \bset{d}{h'(c(\bx);\nabla c(\bx)d)=0}.
	%	\] 
	%	Let $0\neq d\in \Null{A^\top \nabla c(\bx)}$. Then $\nabla c(\bx)d\in \Null{A^\top} = \tcone{\bc}{\cM_{\bc}}$ and for $t\in \R$ sufficiently small, polyhedrality implies $\bc \pm t\nabla c(\bx)d\in \cM_{\bc}$. 
	%	By \Cref{thm:ps} and \eqref{eq:lineality=tangent}, $h'(c(\bx);\cdot)$ is linear on $\tcone{\bc}{\cM_{\bc}}$, and by \Cref{thm:fonc} it is nonnegative there as well. Then $h'(c(\bx);\nabla c(\bx)d) = 0$.
	\\
	$(\subset)$ 
	Let $0\neq d\in D(\bx)$, and suppose to the contrary that $d = d_1 + d_2$, where $d_1\in\Null{A^\top\nabla c(\bx)}$ and $d_2=\nabla c(\bx)^\top Aw$, $w\neq0$. By \Cref{pslemma}, $\Ran{A}\subset \parr{\sd h(\bc)}$. Since $\by\in \ri{\sd h(\bc)}$, there exists $\epsilon>0$ so that $\by + \epsilon Aw\in \sd h(\bc)$. Then,
	\begin{align*}
	0 &\ge h'(c(\bx);\nabla c(\bx)d) \\
	&= \sup_{y\in \sd h(\bc)}\ip{\nabla c(\bx)^\top y}{d}\\
	&\ge \ip{\by + \epsilon Aw}{\nabla c(\bx)(d_1 + \nabla c(\bx)^\top Aw)} \\
	&\ge \ip{\nabla c(\bx)^\top\by}{d} + \epsilon\norm{\nabla c(\bx)^\top Aw}^2 \\
	&= \epsilon\norm{\nabla c(\bx)^\top Aw}^2,
	\end{align*}
	so $w=0$ (see \Cref{remark:nondegenandtransversality}).
\end{proof}
%\begin{remark}
%	Suppose $\cM_{\bc} = \set{\bc}$, and $\cM_{\bc}$ satisfies the nondegeneracy condition. Then $m=\ell$ so that $\ncone{\bc}{\cM_{\bc}}=\R^m=\parr{\sd h(\bc)}$. If $f$ satisfies \eqref{eq:strictcriticality} at $\bx$ for $\by$ with $\ell=n$, then $\by\in\intr{\sd h(\bc)}$ and $\nabla c(\bx)^\top A$ is invertible. Consequently, $D(\bx)\subset \Null{A^\top \nabla c(\bx)}=\set{0}$. In addition, there exists a constant $\gamma>0$ and a neighborhood $U$ of $\bx$ such that
%	\begin{equation}		\label{eq:sharpLM}
%		f(y)\ge f(\bx) + \gamma \norm{y-\bx}, \text{ for all }y\in U.
%	\end{equation}
%	That is, $\bx$ is a local sharp solution to \ref{theprogram} (see \cite{burke2002weak, 2017arXiv170502356D}).
%\end{remark}
By a continuity argument in $(x,y)$, we have the following result which is important for our discussion of the metric regularity of Newton's iteration in the next section. It states that, in the presence of partial smoothness, \eqref{eq:generaltransversality} and the curvature condition are local properties.
\begin{lemma}
	\label{posdeflocally}
	Suppose \eqref{eq:pstransversality} holds and that for all $j\in \cK(\bc)$ and
	\begin{equation*}
	d^\top\nabla c(\bx)^\top Q_j \nabla c(\bx)d + d^\top \nabla^2(\by c)(\bx) d > 0,\quad\forall\ d\in \Null{A^\top \nabla c(\bx)}\setminus\set{0}.
	\end{equation*}
	Then, there exists a neighborhood $\cN$ of $(\bx,\by)$ such that if $(x,y)\in \cN$ then for all $j\in \cK(\bc)$,
	\begin{equation}
	\label{eq:posdefhess2}
	d^\top\nabla c(x)^\top Q_j \nabla c(x)d + d^\top\nabla^2(y c)(x)d > 0,\quad\forall\  d\in \Null{A^\top \nabla c(x)}\setminus\set{0}.
	\end{equation}
	and
	\(
	\Null{\nabla c(x)^\top}\cap\Ran{A} = \set{0}.
	\)
\end{lemma}
The following examples are inspired by the discussion in \cite{lewis2002active}.
\begin{example}
	In $\R^2$, let $h_a(c) = \norm{c}_1^2$, so $h$ is piecewise linear-quadratic convex.
	If $\cM := \set{0}$, then $h_a$ is not partly smooth relative to $\cM$ because $\sd h_a(0) = \set{0}$ while $\ncone{0}{\cM}=\R^n$. On the other hand, if $h_b(c) = \norm{c}_1$ with the same domain representation, then $\sd h(0) = \bB_\infty$, in which case $h_b$ is partly smooth relative to $\cM$.
	\\
	Suppose we represent the domain of $h_a$ and $h_b$ as the four quadrants in the plane, relative to each of which $h_a, h_b$ are linear-quadratic. This representation meets the criteria of the Rockafellar-Wets PLQ representation of \Cref{cor:plqstructthm}. For both $h_a$ and $h_b$, the nondegeneracy condition for $\cM$ holds since $A$ can be taken to be $I_2$.
\end{example}
\begin{example}
	In $\R^2$, the domain of $h_a$ and $h_b$ in the previous example can be presented in the following way. Take each of the four quadrants in the plane and split them along their respective diagonal. Define $h_a$ as usual on each of the pieces. Then this presentation describes $\dom{h_a}$ using 4 hyperplanes and also meets the Rockafellar-Wets PLQ representation theorem. However, the nondegeneracy condition fails for $\cM$ in this representation.
	\\
	On the manifold $\cM$ given by an ``artificial" diagonal, the matrix $A$ is comprised of a single column, with $\ncone{c}{\cM}=\Ran{A}$ for any $c\in\cM$. However, $h_a$ is smooth on $\cM$ with $\parr{\sd h(c)} = \set{0}$.
\end{example}
%\begin{proof}
%	If the result were false, there exists a sequence $\set{(x_i,y_i)}\to (\bx,\by)$ an index $j_0\in \cK(\bc)$ and a sequence of vectors $d_i\neq 0,\ \norm{d_i}=1$ converging to $\bd$ so that
%	\[
%		d_i^\top[\nabla^2(y_i c)(x_i) + \nabla c(x_i)^\top Q_{j_0} \nabla c(x_i)]d_i \le 0
%	\]
%	with $A^\top \nabla c(x_i)d_i =0$. By continuity, $i\to\infty$ gives a contradiction. The existence of a neighborhood of $\bx$ on which \eqref{eq:generaltransversality} holds is a similar argument. The result follows.
%\end{proof}
We end this section with a relationship between partial smoothness and the convergence analysis of quasi-Newton methods studied in \ref{subsec:qnmethods}. The following result is a finite identification property for any algorithm solving \ref{theprogram} in the presence of an active manifold at a solution.
\begin{theorem}\cite[Theorem 4.10]{lewis2016proximal}
	\label{thm:lewisactiveidentify}
	Suppose the closed, proper, convex function $h:\R^m\to\eR$ is partly smooth at the point $\bc\in\R^m$ relative to a manifold $\cM\subset\R^m$. Consider a subgradient $\by\in\ri{\sd h(\bc)}$. Suppose the sequence $\set{\hat c_k}\subset \R^m$ satisfies $\hat c_k\to\bc$ and $h(\hat c_k)\to h(\bc)$. Then, $\hat c_k\in \cM_{\bc}$ for all large $k$ if and only if $\dist{\by}{\sd h(\hat c_k)}\to0$.
\end{theorem}
Combining \Cref{cor:qnMethod} and \Cref{thm:lewisactiveidentify}, we have the following relationship between the sufficient conditions for superlinear convergence of the quasi-Newton method \ref{qnsubproblems} and the finite identification of an active manifold at a solution.
\begin{corollary}
	Let $\cM_{\bc}$ be as in \eqref{eq:manifold}, let \Cref{assum:ps} hold, and recall the notation of \Cref{def:AmatrixPs} Let $\bx\in\dom{f}$ and $\bc:=c(\bx)$. 
	\\
	Suppose 
	\begin{enumerate}[label=(\alph*)]
		\item $\cM_{\bc}$ satisfies the nondegeneracy condition,
		\item the $k$-strict complementarity condition of \Cref{def:kstrictcc} holds at $(c,y)\in\R^m\times\R^m$,
		\item $M(\bx)=\set{\by}$, and
		\item the second-order sufficient conditions of \Cref{thm:plqsonsc} are satisfied at $\bx$.
	\end{enumerate}
	Consider the neighborhood $U$ of $(\bx,\by)$ of \Cref{cor:qnMethod}, and a starting point $(x^0,y^0)\in U$. Suppose the sequence $\set{(x^k,y^k)}_{k\in\bN}$ is generated from the optimality conditions for \ref{qnsubproblems}, remains in $U$ for all $k\in\bN$, and satisfies $(x^k,y^k)\ne(\bx,\by)$ for all $k\in \bN$. Then, the sufficient conditions for superlinear convergence of \Cref{cor:qnMethod} imply $c(x^k)+\nabla c(x^k)[x^{k+1}-x^k]\in \cM_{\bc}$ for all large $k$.
\end{corollary}
\begin{proof}
	Since $x^k\to\bx,\ d^k\to 0$. By continuity, $\hat c_k:=c(x^k)+\nabla c(x^k)[x^{k+1}-x^k] \to \bc$. The quasi-Newton method \eqref{eq:geqqn} with $\mathbf{B}_k$ given by \eqref{eq:qncB=b} implies $y^{k+1}\in \sd h(\hat c_k)$, so $\set{\hat c_k}\subset\dom{h}$. By \Cref{plqformulae}, $h(\hat c_k)\to h(\bc)$. Since $y^k\to\by$,
	\(
	\dist{\by}{\sd h(\hat c_k)}\leq \norm{\by - y^{k+1}}\to0.
	\)
	Then, by partial smoothness and \Cref{thm:lewisactiveidentify}, $\hat c_k\in \cM_{\bc}$ for all large $k$.
\end{proof}
\section{Strong Metric Regularity and Local Quadratic Convergence of Newton's Method}
\label{sec:metricreg}
The point of this section is to marry the partial smoothness hypothesis to the hypotheses used to establish strong metric subregularity in \Cref{sec:partialsmoothness} to establish strong metric regularity of a solution mapping that is an appropriately defined local version of $g+G$ in \eqref{kktS}. In addition, we establish the local quadratic convergence of the Newton method for $g+G$.
\begin{definition}[Metric regularity]
	\label{def:mr}
	A set-valued mapping $S:\R^n \rightrightarrows\R^m$ is \emph{metrically regular} at $\bx$ for $\by$ when $\by\in S(\bx)$, the graph of $S$ is locally closed at $(\bx,\by)$, and there exists $\kappa\ge0$ and neighborhoods $U$ of $\bx$ and $V$ of $\by$ such that
	\(
	\dist{x}{S^{-1}(y)}\le\kappa\dist{y}{S(x)}\text{ for all }(x,y)\in U \times V.
	\)
	The infimum of $\kappa$ over all $(\kappa,\ U,\ V)$ satisfying the display is called the metric regularity modulus of $S$ at $\bx$ for $\by$, and is denoted $\mathrm{reg}(S;\bx|\by)$.
\end{definition}
\begin{definition}[Strong metric regularity]
	\label{def:strongmr}
	A set-valued mapping $S:\R^n \rightrightarrows\R^m$ is \emph{strongly metrically regular} at $\bx$ for $\by$ when it is metrically regular at $\bx$ for $\by$ and $S^{-1}$ has a single-valued localization at $\by$ for $\bx$. Equivalently, when $S^{-1}$ has a Lipschitz continuous single-valued localization around $\by$ for $\bx$.
\end{definition}
\subsection{Partly Smooth Problems}\hfill

In this section, we make the following assumptions:
\begin{assumptions}
	\label{assum:mr}
	Let $f$ be as in \ref{theprogram}, $(\bx,\by)\in\dom{f}\times\R^m,\ \bc:=c(\bx),\ \bk=|\cK(\bc)|$, where $\cK(\bc)$ are the active indices given in \Cref{def:activeidx}. Let $\cM_{\bc}$ be the active manifold defined in \eqref{eq:manifold} and let $\bmu_j\in\R^\ell$ for $j\in\set{1,\dotsc,\bk}$, where $\ell=|I_k(\bc)|$ for any $k\in \cK(\bc)$ with $I_k(\bc)$ defined in \eqref{eq:RW-Ikc}. Recall that $\ell$ is well-defined by \Cref{activepolyhedraequal}. With these specifications, we assume that 
	\begin{enumerate}[label=(\alph*)]
		\item $\dom{h}$ is given by the Rockafellar-Wets PLQ representation of \Cref{cor:plqstructthm},
		\item $c$ is $\cC^3$-smooth,
		\item $\cM_{\bc}$ satisfies the nondegeneracy condition (in particular, $\bk\ge2$), \label{assumitem:mrNonDegen}
		\item $f$ satisfies \eqref{eq:strictcriticality} at $\bx$ for $\by$; i.e.,
		$
		\Null{\nabla c(\bx)^\top}\cap\ri{\sd h(\bc)}=\set{\by}
		$
		\label{assumitem:mrSC}, so that in particular, as in \eqref{eq:plqsdsystem}, $J\by = \cQ\bc + \cB + \hat\cA \bmu$, where $\bmu = (\bmu_1^\top,\dotsc,\bmu_{\bk}^\top)^\top>0$ by \Cref{risdiffstrict},
		\item \label{assumitem:mrSOSC}
		$\bx$ satisfies the second-order sufficient conditions of \Cref{thm:plqsonsc}, i.e.,
		\[
		h''(c(\bx); \nabla c(\bx)d) + \ip{d}{\nabla^2(\by c)(\bx)d} >0\quad \forall\, d\in\Null{A^\top\nabla c(\bx)}\setminus\set{0} ,
		\]
		where, by \Cref{lem:CQimplicationchain}, $M(\bx)=\set{\by}$, and by \Cref{lem:D(x)andA^Tc'(x)}, $D(\bx)=\Null{A^\top\nabla c(\bx)}$.
	\end{enumerate}
\end{assumptions}
The conditions (c) - (e) %\cref{assumitem:mrNonDegen, assumitem:mrSC,assumitem:mrSOSC} 
in \Cref{assum:mr} can be interpreted in terms of similar assumptions employed in classical NLP. Condition (c) corresponds to the linear independence of the active constraint gradients, (d) corresponds to strict complementary slackness, and (e) corresponds to the strong second-order sufficiency condition. The convergence results developed in this section subsume those known for NLP, since they follow from the case in which $h$ is non finite-valued piecewise linear convex.
\\
We begin with a key technical lemma important for establishing metric regularity.
\begin{lemma}
	\label{qqvinranA} In the notation of \Cref{def:AmatrixPs},
	for any $i,j\in \set{1,\dotsc,\bk},\ (Q_{k_i}-Q_{k_j})\Null{A^\top}\subset \ran{A}$.
\end{lemma}
\begin{proof}
	Let $w\in\Null{A^\top}$. By polyhedrality, there exists $|t|>0$ such that $c_t:=\bc + tw \in \cM_{\bc}$. By \Cref{plqformulae}, $\dom{\sd h} = \dom{h}$, so there exists $v\in \sd h(c_t)$ and $\bv \in \sd h(\bc)$. By \eqref{eq:plqsdsystem}, $(v,\mu(c_t,v))$ and $(\bv,\bmu)$ satisfy
	\(
	Jv = \cQ c_t + \cB + \hat \cA \mu(c_t,v)\) and 
$J\bv = \cQ\bc + \cB+ \hat \cA \bar\mu.$
	Then for any $i,j\in \cK(\bc)$, 
	\begin{align*}
	0 &= (Q_{k_i}-Q_{k_j})c_t + A (P_i\mu(c_t,v)_i-P_j\mu(c_t,v)_j) + b_{k_i} - b_{k_j}, \\
	0 &= (Q_{k_i}-Q_{k_j})\bc + A(P_i\bar\mu_i-P_j\bar \mu_j) + b_{k_i} - b_{k_j}.
	\end{align*}
	Subtracting the second equation from the first and rearranging gives
	\begin{equation}
	\label{eq:twist}
	(Q_{k_i}-Q_{k_j})w = t^{-1}A\set{
		P_j(\mu(c_t,v)_j-\bar \mu_j) - P_i(\mu(c_t,v)_i - \bar\mu_i)
	}.
	\end{equation}
\end{proof}
We now define a family of local approximations to $g+G$ for which strong metric regularity is established.
\begin{definition}\label{def:g_jandG}
	For a point $\bc\in \cM_{\bc}$ and each $j\in\set{1,\dotsc,\bk}$, define $g_j:\R^{n+m+\ell}\to\R^{n+m+\ell+\ell}$.
	\[
	g_j(x,y,\mu_j) := \begin{pmatrix}
	\nabla c(x)^\top y\\
	y - Q_{k_j}c(x)-b_{k_j} - AP_j\mu_j\\
	A^\top [c(x)-\bc] \\
	-\mu_j
	\end{pmatrix},\quad G_0 := \begin{pmatrix}
	\set{0}^n \\
	\set{0}^m\\
	\set{0}^\ell\\
	\R_+^\ell.
	\end{pmatrix}
	\]
	and set $\bm\bx_j := (\bx, \by, \bmu_j)\in\R^{n+m+\ell}$, where $\bx,\by,\bmu_j$ are as in \Cref{assum:mr}. Then
	\[
	\nabla g_j(x,y,\mu_j) = \begin{pmatrix}
	\nabla^2 (y c)(x) & \nabla c(x)^\top & 0 \\
	-Q_{k_j}\nabla c(x) & I & -AP_j \\
	A^\top \nabla c(x) & 0 & 0 \\
	0 & 0 & -I_\ell
	\end{pmatrix}, \quad g_j(\bm\bx_j)=\begin{pmatrix}0\\0\\0\\-\bmu_j\end{pmatrix}\in -G_0 \text{ (see \Cref{assum:mr} \ref{assumitem:mrSC})}.
	\]
\end{definition}
In parallel to the study in \Cref{sec:subreg}, we introduce the linearization of these mappings.
\begin{definition}[$\cM_{\bc}$-restricted KKT Mappings]
	Let $\bc$ and $\bk$ be given by \Cref{assum:mr}, and $g_j$ and $G_0$ be as in \Cref{def:g_jandG}. For all $j\in\set{1,\dotsc,\bk}$, define the linearization of $g_j + G_0$ at $\bm u = (\hat x,\hat y,\hat \mu_j)$
	\begin{align}
	\label{eq:defGju}
	\cG^j_{\bm u}(\bm x) &:= g_j(\bm u) + \nabla g_j(\bm u)(\bm x - \bm u) + G_0,\text{ or equivalently,}\\
	\cG^j_{(\hat x,\hat y,\hat \mu_j)}(x,y,\mu_j) &:= g_j(\hat x,\hat y,\hat \mu_j) + \nabla g_j(\hat x,\hat y,\hat \mu_j)\begin{pmatrix}
	x - \hat x \\ y-\hat y \\ \mu_j - \hat\mu_j\
	\end{pmatrix} + G_0.\nonumber
	\end{align} 
	For any $\bm u = (\hat x,\hat y,\hat \mu_j)$, define the function
	\begin{equation}
	\label{eq:F3linear}
	F_{\bm u}(\bm x,\bm z) :=  g_j(\bm u) + \nabla g_j(\bm u)(\bm x - \bm u) - \bm z = \begin{pmatrix}
	\nabla c(\hat x)^\top y + \nabla^2(\hat yc)(\hat x)[x-\hat x] - z_1 \\
	y - Q_{k_j}[c(\hat x)+\nabla c(\hat x)[x-\hat x]] - b_{k_j} - AP_j\mu_j - z_2 \\
	A^\top[c(\hat x) + \nabla c(\hat x)[x-\hat x] - \bc] - z_3 \\
	-\mu_j - z_4
	\end{pmatrix}.
	\end{equation}
	Then, 
	\begin{equation}
	\label{eq:gphG^j_u}
	\gph{\cG^j_{\bm u}} = \bset{(\bm x, \bm z)}{F_{\bm u}(\bm x, \bm z)\in -G_0},
	\end{equation}
	with $	\dom{\cG^j_{\bm u}}=\R^{n+m+\ell}$. 
	Explicitly,
	\begin{equation}
	\label{eq:gphGj}
	\gph{\cG^j_{(\hat x,\hat y,\hat\mu_j)}} = \bset{(x,y,\mu_j,z_1,z_2,z_3,z_4)}{\begin{aligned}
		z_1 &= \nabla c(\hat x)^\top y + \nabla^2(\hat yc)(\hat x)[x-\hat x] \\
		z_2 &= y - Q_{k_j}[c(\hat x)+\nabla c(\hat x)[x-\hat x]] - b_{k_j} - AP_j\mu_j  \\
		z_3 &= A^\top[c(\hat x) + \nabla c(\hat x)[x-\hat x] - \bc] \\
		z_4 &\in -\mu_j + \R_+^\ell
		\end{aligned}}.
	\end{equation}
\end{definition}
The next lemma shows that the error in the Newton iterates can be measured in terms of $(x,y)$ alone, independent of the vectors $\mu_j$.
\begin{lemma}
	Let $\bx,\by,\bmu,\bc,\bk$, and $\cQ$ be as in \Cref{assum:mr}, and $g_j$ and $G_0$ be as in \Cref{def:g_jandG}. For any $j\in\set{1,\dots,\bk}$, define $\eta_j:\R^n\times\R^m\to\R^{n+m+\ell}$ by
	\begin{equation}\label{eq:etadef}
	\eta_j(x,y) := \begin{pmatrix}
	\nabla c(x)^\top y\\
	Q_{k_j}(\bc - c(x))\\
	A^\top(c(x)-\bc)
	\end{pmatrix}.
	\end{equation}
	Observe that for any $(x,y,\mu_j)\in\R^n\times\R^m\times\R^\ell$,
	{\small
		\[
		g_j(x,y,\mu_j) = \begin{pmatrix}
		\eta_j(x,y) \\ 0
		\end{pmatrix} + \begin{pmatrix}
		0\\y-\by + AP_j(\bmu_j-\mu_j)\\0\\-\mu_j
		\end{pmatrix}\mbox{ and }\nabla g_j(x,y,\mu_j) = \begin{pmatrix}
		\nabla\eta_j(x,y) & 0 \\0&0
		\end{pmatrix} + \begin{pmatrix}
		0 & 0 & 0\\
		0 & I & -AP_j \\
		0 & 0 & 0 \\
		0 & 0 & -I
		\end{pmatrix}
		\]
	}
	Set $\bm\bx_j:=(\bx,\by,\bmu_j)$. Then, for any $\bm u := (\hx,\hy,\hmu_j)\in\R^n\times\R^m\times\R^\ell$,
	\begin{equation}
	\label{eq:linearizationnomudep}
	\norm{F_{\bm u}(\bm\bx_j, g_j(\bm\bx_j))} = \norm{\eta_j(\hat x,\hat y) + \nabla \eta_j(\hat x,\hat y)\begin{pmatrix}
		\bx - \hat x \\ \by - \hat y
		\end{pmatrix} - \eta_j(\bx, \by)},
	\end{equation}
	since $\eta_j(\bx,\by)=0$.
\end{lemma}
The following lemma uses the strict criticality assumption to show the normal cone to the graph of these linearization are captured by the range of $\nabla F_{\bm\bx_j}$.
\begin{lemma}
	\label{nconecqlemma}
	Let $\bx,\by,\bmu,\bc,\bk$, and $\cQ$ be as in \Cref{assum:mr} and set $\bm\bx_j:=(\bx,\by,\bmu_j)$.
	Then, for all $j\in \set{1,\dotsc,\bk}$, the mapping $\cG^j_{\bm\bx_j}$ in  \eqref{eq:gphG^j_u} has $\ncone{(\bm\bx_j,\bm0)}{\gph{\cG^j_{\bm\bx_j}}}=\ran{W}$, where
	\begin{equation}
	\label{eq:MmatrixMR}
	W :=
	\begin{pmatrix}
	\nabla^2(\by c)(\bx) & -\nabla c(\bx)^\top Q_{k_j} & \nabla c(\bx)^\top A\\
	\nabla c(\bx) & I_m & 0\\
	0 & -P_jA^\top & 0\\
	-I_n & 0 & 0 \\
	0 & -I_m & 0 \\
	0 & 0 & -I_\ell\\
	0 & 0 & 0
	\end{pmatrix}.
	\end{equation}
\end{lemma}
\begin{proof}
	The set $\gph{\cG^j_{\bm\bx_j}} = \bset{(\bm x, \bm z)}{F_{\bm\bx_j}(\bm x, \bm z)\in -G}$ defined in \eqref{eq:gphG^j_u} is closed with $(\bm\bx_j,\bm 0)\in \gph{\cG^j_{\bm\bx_j}}$. In addition, $\bmu_j>0$, $\ncone{F_{\bm\bx_j}(\bm\bx_j,\bm 0)}{-G_0} = \R^{n+m+\ell}\times\set{0}^\ell$, and 
	\[
	\nabla F_{\bm\bx_j}(\bm\bx_j,\bm0)^\top = 
	\begin{pmatrix} 
	\nabla^2(\by c)(\bx) & -\nabla c(\bx)^\top Q_{k_j} & \nabla c(\bx)^\top A & 0\\
	\nabla c(\bx) & I_m & 0 & 0\\
	0 & -P_jA^\top & 0 & I_\ell\\
	-I_n & 0 & 0 & 0\\
	0 & -I_m & 0 & 0 \\
	0 & 0 & -I_\ell& 0\\
	0 & 0 & 0 & I_\ell\\ 
	\end{pmatrix} = \begin{pmatrix}
	W\,|\,R
	\end{pmatrix},
	\] 
	where the matrix $R$ is being defined by this expression. Combining the facts in the previous two sentences, the constraint qualification \eqref{eq:rtrCQ} in \Cref{thm:cq} (see appendix), for $\ncone{(\bm\bx_j,\bm0)}{\gph{\cG^j_{\bm\bx_j}}}$ is the requirement that $\Null{W}=\set{0}$. If we verify $\Null{W}=\set{0}$, then $\ncone{(\bm\bx_j,\bm0)}{\gph{\cG^j_{\bm\bx_j}}}=\ran{W}$ by \Cref{thm:cq}. But the presence of the identity matrices in $W$ immediately give $\Null{W}=\set{0}$.
\end{proof}
The metric regularity of the mappings $g_j+ G_0$ follow from the second-order sufficient conditions of \Cref{thm:plqsonsc}.
\begin{lemma}
	\label{newtonmr}
	Let $\bx,\by,\bmu,\bc,\bk$, and $\cQ$ be as in \Cref{assum:mr}, $W$ as in \eqref{eq:MmatrixMR} and set $\bm\bx_j:=(\bx,\by,\bmu_j)$.
	For all $j\in \set{1,\dotsc,\bk}$, 
	\[
	(\bm0, -\bm z) \in \ncone{(\bm\bx_j,\bm0)}{\gph{\cG^j_{\bm\bx_j}}}\ \Longleftrightarrow\ \bm z = 0,
	\]	
	where $\cG^j_{\bm\bx_j}$ is given by \eqref{eq:gphG^j_u}. Then, $\cG^j_{\bm\bx_j}$ is metrically regular at $\bm\bx_j$ for $\bm0$ and
	\[
	\begin{pmatrix}
	\nabla^2(\by c)(\bx) & \nabla c(\bx)^\top & 0 \\
	-Q_{k_j}\nabla c(\bx) & I_m & -AP_j \\
	A^\top\nabla c(\bx) & 0 & 0
	\end{pmatrix}
	\]
	is nonsingular.
\end{lemma}
\begin{proof}
	By \Cref{nconecqlemma}, $\ncone{(\bm\bx_j,\bm0)}{\gph{\cG^j_{\bm\bx_j}}}=\ran{W}$, and so the statement
	\[
	(\bm0, -\bm z) \in \ncone{(\bm\bx_j,\bm0)}{\gph{\cG^j_{\bm\bx_j}}}\ \Longleftrightarrow\  \bm z = 0
	\]
	is equivalent to
	{\small
		\begin{equation}\label{eq:coderivmonster}
		\left[\begin{pmatrix}
		0\\0\\0\\-z_1\\-z_2\\-z_3\\-z_4
		\end{pmatrix} = \begin{pmatrix}
		\nabla^2(\by c)(\bx) & -\nabla c(\bx)^\top Q_{k_j} & \nabla c(\bx)^\top A\\
		\nabla c(\bx) & I & 0\\
		0 & -P_jA^\top & 0\\
		-I_n & 0 & 0 \\
		0 & -I_m & 0 \\
		0 & 0 & -I_\ell\\
		0 & 0 & 0
		\end{pmatrix}\begin{pmatrix} d\\v\\w\end{pmatrix}\text{ for some }
		\begin{pmatrix} d\\v\\w\end{pmatrix}\right]
		\Longleftrightarrow \begin{pmatrix}z_1\\z_2\\z_3\\z_4\\\end{pmatrix}=0.
		\end{equation}}
	Since $(\Leftarrow)$ is trivial, we only establish $(\Rightarrow)$.
	Define $H:=\nabla^2(\by c)(\bx)$. Then the left-hand side of \eqref{eq:coderivmonster} becomes
	\begin{align}
	\label{eq:mrnconecq1}
	0 &= Hd - \nabla 
	c(\bx)^\top Q_{k_j} v + \nabla c(\bx)^\top Aw,\\
	\label{eq:mrnconecq2}
	0 &= \nabla c(\bx) d + v, \\
	\label{eq:mrnconecq3}
	0 &= -P_jA^\top v,\\
	z_1 &= d,\  
	z_2 = v,\ 
	z_3 = w,\ 
	z_4 = 0.\nonumber
	\end{align}
	Since $z_4=0$, we need only show $z_1=z_2=z_3=0$, which we establish by showing $d=v=w=0$. First suppose $d\neq0$. From \eqref{eq:mrnconecq3} and \Cref{def:AmatrixPs}, $v\in \Null{A^\top}$. Then \eqref{eq:mrnconecq2} and gives $\nabla c(\bx)d = -v \in \Null{A^\top}$. By \Cref{lem:D(x)andA^Tc'(x)}, $d\in D(\bx)\setminus\set{0}$. Taking the inner product on both sides of \eqref{eq:mrnconecq1} with $d$ and using \eqref{eq:mrnconecq2} gives $d^\top H d = d^\top \nabla c(\bx)^\top Q_{k_j} v = -d^\top \nabla c(\bx)^\top Q_{k_j}\nabla c(\bx)d,$ so
	\[
	d^\top \nabla c(\bx)^\top Q_{k_j} \nabla c(\bx) d + d^\top H d  = 0.
	\]
	But the second-order sufficient conditions of \Cref{thm:plqsonsc} imply that for any $j\in \set{1,\dotsc,\bk}$,
	\begin{align*}
	d^\top \nabla c(\bx)^\top Q_{k_j} \nabla c(\bx) d + d^\top H d> 0.
	\end{align*}
	This contradiction implies $d=0$. But then $v=0$ by \eqref{eq:mrnconecq2}. Finally, \eqref{eq:mrnconecq1} states that $w$ must satisfy
	\(
	Aw\in \Null{\nabla c(\bx)^\top}\cap \Ran{A} = \set{0}.
	\)
	By the nondegeneracy condition of \Cref{nondegeneracy}, $w=0$. Equation \eqref{eq:gphG^j_u} gives local closedness of $\cG_{\bm\bx_j}^j$ at $(\bm\bx_j,\bm0)$, so the coderivative criterion for metric regularity \cite[Theorem 4C.2]{dontchev2014implicit} implies $\cG^j_{\bm\bx_j}$ is metrically regular at $\bm\bx_j$ for $\bm 0$, as required.
\end{proof}
%\begin{remark}
%\label{linopenfeasible}
%For any $j\in\set{1,\dotsc,k}$, the linear openness of $S_j$ implies (by setting $(d,y,\mu_j)=(0,\by,\bmu_j)$ in \eqref{eq:linopendef}) that
%\[
%	S_j((0,\by,\bar\mu_j) + \kappa_j \eps \intr{\bB}) \supset [\bx + \eps\intr{\bB}]\cap V \text{ for all }\eps>0.
%\]
%Then
%\[
%	\norm{x-\bx}\le \eps \Longrightarrow \norm{\begin{pmatrix}
%		d(x) \\ y(x) - \by \\ \mu_j(x) - \bar\mu_j
%	\end{pmatrix}} \le \kappa_j \epsilon.
%\]
%\end{remark}
%\begin{remark}
%	The Aubin property of $S_j^{-1}$ at $\bx$ for $(0,\by,\bar\mu_j)$ implies local nonemptiness of $S_j^{-1}$: for every neighborhood $U$ of $(0,\by,\bar\mu_j)$ there exists a neighborhood $V$ of $\bx$ such that $S^{-1}_j(x)\cap U\ne\emptyset$ for all $x\in V$.
%\end{remark}
The metric regularity of the mappings $\cG^j_{\bm\bx_j}$ imply a parameterized uniform version of metric regularity, where we allow $\bm\bx_j$ to move.
\begin{lemma}
	\label{strongmr}
	Let $\bx,\by,\bmu,\bc,\bk$, and $\cQ$ be as in \Cref{assum:mr}, set $\bm\bx_j:=(\bx,\by,\bmu_j)$, and let $\cG^j_{\bm\bx_j}$ be given by \eqref{eq:gphG^j_u}.
	For all $j\in\set{1,\dotsc,\bk}$, there exists a neighborhood $U_j\subset \R^{n+m+\ell}$ of $\bm\bx_j$ and a neighborhood $V_j\subset\R^{n+m+\ell+\ell}$ of $\bm 0$ such that the mapping 
	\[
	(\bm u,\bm z) \mapsto \cG^{-j}_{\bm u}(\bm z):= \left(\cG^j_{\bm u}\right)^{-1}\!\!\!\!\!(\bm z)\text{ for }(\bm u,\bm z)\in U_j\times V_j 
	\]
	is single-valued with $\cG^{-j}_{\bm u}(\bm 0)\in U_j$.
\end{lemma}
\begin{proof}
	Fix $j\in \set{1,\dotsc,\bk}$. By \Cref{newtonmr} and  \cite[Theorem 6D.1]{dontchev2014implicit}, for every $\lambda > \mathrm{reg}(\cG^j_{\bm\bx_j};\bm{\bx}_j|\bm{0})$ there exists $a>0$ and $b>0$ such that
	\begin{equation}
	\label{eq:3mr}
	\dist{\bm{x}}{\cG^{-j}_{\bm{u}}(\bm{z})} \le \lambda \dist{\bm{z}}{\cG^j_{\bm{u}}(\bm{x})},\quad \text{for every }\bm u, \bm x \in \bm\bx_j + a\bB, \bm z\in b\bB.
	\end{equation}
	By reducing $a$, if necessary, we may assume the conclusion of \Cref{posdeflocally} holds on $\bm\bx_j+a\bB$. We follow the argument given in \cite[Theorem 6D.2]{dontchev2014implicit} by recalling \eqref{eq:etadef} and choosing
	\[
	L > \mathrm{lip}(\nabla \eta_j;(\bx,\by)):=\limsup_{\substack{(x,y),(x',y')\to(\bx,\by)\\(x,y)\ne(x',y')}} \frac{\norm{\nabla\eta_j(x,y)-\nabla\eta_j(x',y')}}{\norm{(x,y)-(x',y')}}, \text{ and }\gamma > \frac{1}{2}\lambda L.
	\] Define $\bar a := \min\set{\frac{1}{\gamma},a}>0,\ U_j := \bm\bx_j + \bar a\bB$, and $V_j := b\bB$.
	We first establish nonemptiness of $\cG^{-j}_{\bm u}(\bm z)$. Fix $\bm x = \bm\bx_j$, and choose any $(\bm u, \bm z)\in U_j\times V_j$, and consider two cases in \eqref{eq:3mr}. If $\dist{\bm z}{\cG^j_{\bm u}(\bm\bx_j)}=0$, then by closedness of the set $\cG^j_{\bm u}(\bm\bx_j)$, it follows that $\bm\bx_j \in \cG^{-j}_{\bm u}(\bm z)$. On the other hand, if $0<\dist{\bm z}{\cG^j_{\bm u}(\bm\bx_j)}<\infty$, where finiteness is guaranteed because $\dom{\cG^j_{\bm u}}=\R^{m+n+\ell}$. Then the implication
	\[
	\dist{\bm\bx_j}{\cG^{-j}_{\bm u}(z)} \le \lambda \dist{\bm z}{\cG^j_{\bm u}(\bm\bx_j)}\Longrightarrow \dist{\bm\bx_j}{\cG^{-j}_{\bm u}(z)} < \infty
	\]
	holds, so in both cases $\cG^{-j}_{\bm u}(z)\neq\emptyset$. 
	\\
	We now show single-valuedness. For the same $j,\ \bm u$, and $\bm z$, write $\bm{u}=(\hat x,\hat y,\hat \mu_j)$, and suppose there are two points $\bm x_1 = (x_1,y_1,\mu_{j_1}),\ \bm x_2=(x_2,y_2,\mu_{j_2})$ satisfying $\bm x_1, \bm x_2\in \cG^{-j}_{\bm u}(\bm z)$. Then subtracting the equations in \eqref{eq:gphGj} gives
	\begin{align}
	0 &= \nabla^2(\hat yc)(\hat x)[x_2-x_1] + \nabla c(\hat x)^\top (y_2-y_1)\label{eq:thm6d1IP1}\\
	y_2-y_1 &= Q_{k_j}\nabla c(\hat x)[x_2-x_1] + AP_j(\mu_{j_2}-\mu_{j_1}) \label{eq:SinglevaluedgphGj} \\
	0 &= A^\top \nabla c(\hat x)[x_2-x_1]\label{eq:thm6d1IP3}.
	\end{align}
	Then $\nabla c(\hat x)[x_2-x_1] \in \Null{A^\top}$. Suppose $x_2 \neq x_1$. Taking the inner product on both sides of \eqref{eq:thm6d1IP1} and using the choice of $\bar a$ in accordance with \Cref{posdeflocally},
	\begin{align*}
	0&= [x_2-x_1]^\top \nabla^2(\hat yc)(\hat x)[x_2-x_1] + [x_2-x_1]^\top \nabla c(\hat x)^\top (y_2-y_1) && \text{ by }\eqref{eq:thm6d1IP1} \\
	&=	[x_2-x_1]^\top \nabla^2(\hat yc)(\hat x)[x_2-x_1] + [x_2-x_1]^\top \nabla c(\hat x)^\top [Q_{k_j}\nabla c(\hat x)[x_2-x_1] + AP_j(\mu_{j_2}-\mu_{j_1})] && \text{ by }\eqref{eq:SinglevaluedgphGj} \\
	&= [x_2-x_1]^\top \nabla^2(\hat yc)(\hat x)[x_2-x_1] + [x_2-x_1]^\top \nabla c(\hat x)^\top Q_{k_j}\nabla c(\hat x)[x_2-x_1] && \text{ by }\eqref{eq:thm6d1IP3}\\
	&>0,
	\end{align*}
	so $x_2=x_1$. But then \eqref{eq:thm6d1IP1}, \eqref{eq:SinglevaluedgphGj}, and \Cref{posdeflocally} imply
	\[
	y_2-y_1 \in \Null{\nabla c(\hat x)^\top}\cap \Ran{A} = \set{0},
	\]
	so $y_2=y_1$. The nondegeneracy condition of \Cref{nondegeneracy} and \eqref{eq:SinglevaluedgphGj} together imply
	\[
	0 = AP_j(\mu_{j_2}-\mu_{j_1}) \Longrightarrow \mu_{j_2}=\mu_{j_1},
	\]
	so single-valuedness is established. We conclude the proof by following the proof given in \cite[Theorem 6D.2]{dontchev2014implicit} and write $(x,y,\mu_j)=\bm x = \cG^{-j}_{\bm u}(0)$. Then the quadratic bound lemma and the choice of $\gamma$ gives
	\begingroup
	\allowdisplaybreaks
	\begin{align*}
	\norm{\begin{pmatrix}
		x-\bx\\ y-\by\\
		\end{pmatrix}} &\leq \norm{\bm x - \bm\bx_j} \\
	&= \dist{\bm\bx_j}{\cG^{-j}_{\bm u}(0)}\\
	&\le \lambda\dist{\bm 0}{\cG^j_{\bm u}(\bm\bx_j)}\\ 
	&\le \frac{2\gamma}{L}\dist{\bm 0}{\cG^j_{\bm u}(\bm\bx_j)} \\
	&\le \frac{2\gamma}{L}\norm{g_j(\bm u) + \nabla g_j(\bm u)(\bm\bx_j-\bm u) - g_j(\bm\bx_j)} && \text{ by \eqref{eq:defGju} and }-g_j(\bm\bx_j)\in G_0\\
	&=\frac{2\gamma}{L}\norm{F_{\bm u}(\bm\bx_j, g_j(\bm\bx_j))} && \text{ by \eqref{eq:F3linear}} \\
	%&= \frac{2\gamma}{L}\norm{
	%\begin{pmatrix}
	%\nabla c(\hat x)^\top \by + \nabla^2(\hat yc)(\hat x)[\bx-\hat x] \\
	% Q_{k_j}[\bc-c(\hat x)-\nabla c(\hat x)[\bx-\hat x]] \\
	%A^\top[c(\hat x) + \nabla c(\hat x)[\bx-\hat x] - \bc]
	%\end{pmatrix}}\\
	&= \frac{2\gamma}{L}\norm{\eta_j(\hat x,\hat y) + \nabla\eta_j(\hat x,\hat y)\begin{pmatrix}\bx-\hat x\\ \by-\hat y\end{pmatrix} - \eta_j(\bx,\by)} && \text{ by \eqref{eq:linearizationnomudep}} \\
	& \leq \gamma \norm{\begin{pmatrix}\hat x-\bx\\ \hat y-\by\end{pmatrix}}^2 \\
	& \leq \gamma \norm{\bm u -\bm\bx_j}^2 \\
	&< \bar a,
	\end{align*}
	\endgroup
	so $\bm x = \cG^j_{\bm u}(\bm0) \in U_j$.
\end{proof}
Our work so far implies that Newton's method applied to the individual mappings $\cG^j_{\bm\bx_j}$ exhibit local quadratic convergence.
\begin{theorem}
	\label{thm:Gjquadconv}
	Let $\bx,\by,\bmu,\bc,\bk$, and $\cQ$ be as in \Cref{assum:mr}, set $\bm\bx_j:=(\bx,\by,\bmu_j)$, and let $\cG^j_{\bm\bx_j}$ be given by \eqref{eq:gphG^j_u}.
	Then, the mappings $\set{\cG_{\bm\bx_j}^j}_{j=1}^{\bk}$ are strongly metrically regular (see \Cref{def:strongmr}) at $\bm\bx_j$ for $\bm0$. Moreover, for all $j\in\set{1,\dots,\bk}$, there exists a neighborhood $U_j$ of $\bm\bx_j$ such that, for every $\bm x^0\in U_j$, there is a unique sequence $\bm x_j^k = (x^k, y^k, \mu^k_j)\subset U_j$ generated by Newton's method for $g_j + G_0$ \eqref{eq:geqnewton}. Both this sequence, and the sequence $(x^k, y^k)$, converge at a quadratic rate to $\bm\bx_j$ and $(\bx,\by)$ respectively.
\end{theorem}
\begin{proof}
	The metric regularity at $\bm\bx_j$ for $\bm0$ was  established in \Cref{newtonmr}. \Cref{strongmr} with $u=\bm\bx_j$ shows $\cG_{\bm\bx_j}^{-j}$ has a single-valued localization around $\bm0$ for $\bm\bx_j$, so the strong metric regularity of $\cG^j_{\bm\bx_j}$ at $\bm\bx_j$ for $\bm0$ follows.
	
	For the second claim, we again follow the proof in \cite[Theorem 6D.2]{dontchev2014implicit} by taking $U_j$ as in \Cref{strongmr}, and choosing any $\bm x^0\in U_j$. Following the proof of the final claim of \Cref{strongmr}, we find, for every $k\ge1$, the existence and uniqueness of $\bm x^k$ given $\bm x^{k-1}$ satisfying
	\begin{align*}
	\bm 0 \in \cG^j_{\bm x^{k-1}}(\bm x^k),\ 
	\norm{\begin{pmatrix} x^k -\bx\\ y^k-\by\end{pmatrix}}  \le \norm{\bm x^k - \bm\bx_j}  \le \gamma\norm{\begin{pmatrix} x^{k-1}-\bx\\ y^{k-1}-\by\end{pmatrix}}^2  \le \gamma\norm{\bm x^{k-1} - \bm\bx_j}^2, \ 
	\text{and }\bm x^k \in U_j.
	\end{align*}
	Moreover, since $\theta:=\gamma \norm{\bm x^0-\bm\bx_j} < \gamma \bar a < 1$,
	\(
	\norm{\bm x^k - \bm\bx_j} \le \theta^{2^k-1}\norm{\bm x^0 - \bm\bx_j}^2\text{ for all }k\ge1,
	\)
	which completes the proof of quadratic convergence of both sequences.
\end{proof}
We now move from an isolated analysis of the mappings $\cG_{\bm u}^j$ to how they behave as a whole. The goal is to guarantee the $y$ obtained by solving $\bm0\in \cG^j_{\bm u}(\bm x)$ at some $\bm u = (\hat x,\hat y,\hat \mu_j)$ for $\bm x=(x,y,\mu_j)$ has $y\in \sd h(c(\hx)+\nabla c(\hx)[x-\hx])$.
\begin{theorem}
	\label{thm:mrgluing}
	Let $\bx,\by,\bmu,\bc,\bk$, and $\cQ$ be as in \Cref{assum:mr}, set $\bm\bx_j:=(\bx,\by,\bmu_j)$, and let $\cG^j_{\bm\bx_j}$ be given by \eqref{eq:gphG^j_u}.
	Suppose $i\neq j$ and $i,j \in \set{1,\dotsc,\bar k}$. There exists a neighborhood $\cN$ of $(\bx,\by,\bmu_1,\dotsc,\bmu_{\bk})=:(\bx,\by,\bmu)\in\R^{n+m+\bk\ell}$ such that, if $(\hat x,\hat y,\hat \mu_1,\dotsc,\hat \mu_{\bk})\in \cN$ and $\bm u_j := (\hat x,\hat y,\hat \mu_j),\ \bm u_i := (\hat x , \hat y, \hat \mu_i)$, with $\hat \mu_i>0$ and $\hat \mu_j>0$, then
	\begin{equation}\label{eq:gluingtransitiongoal}
	\bm x_j := \cG^{-j}_{\bm u_j}(\bm 0) = \begin{pmatrix}
	x_j \\ y_j \\ \mu_j
	\end{pmatrix},\quad  \bm x_i := \cG^{-i}_{\bm u_i}(\bm 0) = \begin{pmatrix}
	x_i \\ y_i \\ \mu_i
	\end{pmatrix}\text{ satisfy }\begin{pmatrix}
	x_j\\y_j
	\end{pmatrix} = \begin{pmatrix}
	x_i\\y_i
	\end{pmatrix}\text{ for all }i,j\in\set{1,\dotsc,\bk}.
	\end{equation}
	That is, there exists $(x,y)\in \R^n\times\R^m$ such that $(x,y)=(x_i,y_i)$ for all $i\in\set{1,\dotsc,\bk}$.
	Moreover,
	\begin{enumerate}[label=(\roman*)]
		\item $c(\hat x) + \nabla c(x)[x-\hat x]\in \cM_{\bc}$,
		\item $\mu(c(\hat x) + \nabla c(\hat x)[x-\hat x], y)_j = \mu_j>0$ for all $j\in\set{1,\dotsc,\bar k}$,
		\item $y\in \ri{\sd h(c(\hat x)+\nabla c(\hat x)[x-\hat x])}$,
	\end{enumerate} 
	where the mapping $\mu(c,y)$ is defined in \Cref{uniquemugivenv}.
\end{theorem}
\begin{proof}
	For $j\in\set{1,\dotsc,\bk}$, define $\pi_j:\R^{n+m+\bk\ell}\to\R^{n+m+\ell}$ by
	\(
	\pi_j(x,y,\mu_1,\dotsc,\mu_j,\dotsc,\mu_{\bk}) := (x,y,\mu_j).
	\)
	We first show there exists a neighborhood $\cN$ of $(\bx,\by,\bmu_1,\dotsc,\bmu_{\bk})$ such that, for all $j\in\set{1,\dotsc,\bk}$ and all $(\hx_j,\hy_j,\hat\mu_j)=\bm u_j \in \cN_j:=\pi_j(\cN)$,
	\begin{enumerate}[label=(\alph*)]
		\item the mappings $\set{\cG^{-j}_{\bm u_j}(\bm 0)}_{j=1}^{\bk}$ are single-valued with $\cG^{-j}_{\bm u_j}(\bm 0)\in \cN_j$,
		\item $\mu_j$ associated to $\cG^{-j}_{\bm u_j}(\bm 0)$ has $\mu_j>0$,
		\item the condition \eqref{eq:posdefhess2} is satisfied at all $(x,y,\mu_j)\in\cN_j$, and
		\item $c(\hx_j) + \nabla c(\hx_j)[x_j-\hx_j]\in \cM_{\bc}$, where $(x_j,y_j,\mu_j) = \cG^{-j}_{(\hx,\hy,\hat\mu_j)}(\bm 0)$.
	\end{enumerate}
	Parts (a), (b), and (c) are a consequence of \Cref{strongmr}. We now justify (d). For any $j\in\set{1,\dotsc,\bk}$, the definition of $(x_j,y_j,\mu_j) = \cG^{-j}_{(\hx,\hy,\hat\mu_j)}(\bm 0)$ implies, in particular,
	\(
	A^\top[c(\hx_j)+\nabla c(\hx_j)[x_j-\hx_j] - c(\bx)]=0.
	\)
	By the polyhedral structure of $\cM_{\bc}$, for any $w\in \Null{A^\top}=\tcone{\bc}{\cM_{\bc}}$, there exists $\tau > 0$ such that  $\bc + tw\in \cM_{\bc}$ for all $|t|<\tau$.
	\Cref{strongmr} argued that, for all sufficiently small $\epsilon>0$,
	\begin{equation}
	\label{eq:GmapsBtoB}
	%\ni \bm u \mapsto 
	\cG^{-j}_{\bm u}(\bm 0)\in (\bm\bx_j + \eps\bB) %\in \bm\bx_j + \epsilon\bB
	\text{ for all }\bm u\in \bm\bx_j+\eps\bB\text{ (see (a))}.
	\end{equation}
	The continuity of $c$ and \eqref{eq:GmapsBtoB} imply that for $\bm u_j$ sufficiently close to $\bm\bx_j,\ c(\hx_j) + \nabla c(\hx_j)[x_j-\hx_j]$ can be made as close to $c(\bx)$ as desired. Then there exists a neighborhood of $(\bx,\by,\bmu_j)$ such (d) holds. The neighborhood $\cN$ also exists because there are only finitely many indices $j$ in consideration.
	
	Now let $\bm u_j := (\hat x,\hat y,\hat \mu_j)\in \cN_j,\ \bm u_i := (\hat x , \hat y, \hat \mu_i)\in\cN_i$, with $\hat \mu_i>0$ and $\hat \mu_j>0$, and denote
	\[
	\cG^{-j}_{\bm u_j}(\bm 0) = \begin{pmatrix}
	x_j \\ y_j \\ \mu_j
	\end{pmatrix},\quad  \cG^{-i}_{\bm u_i}(\bm 0) = \begin{pmatrix}
	x_i \\ y_i \\ \mu_i
	\end{pmatrix}.
	\]	
	By \eqref{eq:gphGj},
	\begin{align}
	\label{eq:gluing1}0 &= \nabla^2(\hat yc)(\hat x)[x_j-x_i] + \nabla c(\hat x)^\top (y_j-y_i)\\
	\label{eq:gluing2}y_i &= Q_{k_i}(c(\hx) + \nabla c(\hat x)[x_i-\hx]) + AP_i\mu_i + b_{k_i}\\
	\label{eq:gluing3}y_j &= Q_{k_j}(c(\hx) + \nabla c(\hat x)[x_j-\hx]) + AP_j\mu_j + b_{k_j} \\
	\label{eq:gluing4}0 &= A^\top \nabla c(\hat x)[x_j-x_i]
	\end{align}
	Define $\hc_i := c(\hx) + \nabla c(\hx)[x_i-\hx]\in \cM_{\bc}$ by (d). By \Cref{assum:smoothmr}, $\by = Q_{k_i}\bc + b_{k_i} + AP_i\bar\mu_i = Q_{k_j}\bc + b_{k_j} + AP_j\bar\mu_j$, and in particular, 
	\begin{equation}\label{eq:twistargument}
	Q_{k_i}\bc + b_{k_i}-b_{k_j} = Q_{k_j}\bc + AP_j\bmu_j - AP_i\bmu_i
	. \end{equation}
	Then \eqref{eq:twist} with $w:=\hc_i - \bc\in\Null{A^\top},\ t=1,$ and any $y\in \sd h(\hc_i)$ gives
	\begin{align*}
	y_i &= Q_{k_i}w + Q_{k_i}\bc + b_{k_i} + AP_i\mu_i\\
	&= \left(Q_{k_j}w + A \set{
		P_j(\mu(\hc_i,y)_j-\bar \mu_j) - P_i(\mu(\hc_i,y)_i - \bmu_i)}\right) + Q_{k_i}\bc + b_{k_i} + AP_i\mu_i + b_{k_j} - b_{k_j} 
	\\
	&= Q_{k_j}w + b_{k_j} + [Q_{k_i}\bc + b_{k_i} - b_{k_j}] + AP_i\mu_i + A \set{
		P_j(\mu(\hc_i,y)_j-\bar \mu_j) - P_i(\mu(\hc_i,y)_i - \bmu_i)}\\
	&= Q_{k_j}[\hc_i-\bc] + b_{k_j} + [Q_{k_j}\bc + AP_j\bmu_j - AP_i\bmu_i] + AP_i\mu_i + A \set{
		P_j(\mu(\hc_i,y)_j-\bar \mu_j) - P_i(\mu(\hc_i,y)_i - \bmu_i)}\\
	&=Q_{k_j}\hc_i + b_{k_j} + AP_i[\mu_i - \mu(\hc_i,y)_i] + AP_j\mu(\hc_i,y)_j\\
	&\in y_j + Q_{k_j}\nabla c(\hx)[x_i-x_j] + \Ran{A}
	\end{align*}
	where the fourth equivalence follows from \eqref{eq:twistargument}.
	This implies 
	\begin{equation}
	\label{eq:gluingyiyj}
	y_j - y_i - Q_{k_j}\nabla c(\hat x)[x_j-x_i] \in \Ran{A}.
	\end{equation}
	Taking the inner product on both sides of \eqref{eq:gluing1} with $x_j-x_i$ gives 
	\begin{align*}
	0 &= [x_j-x_i]^\top\nabla^2(\hat yc)(\hat x)[x_j-x_i] + [x_j-x_i]^\top\nabla c(\hat x)^\top (y_j-y_i) \\
	&=[x_j-x_i]^\top\nabla^2(\hat yc)(\hat x)[x_j-x_i] + [x_j-x_i]^\top\nabla c(\hat x)^\top Q_{k_j}\nabla c(\hx)[x_j-x_i] \text{ by }\eqref{eq:gluingyiyj},\eqref{eq:gluing4}.
	\end{align*}
	By \Cref{posdeflocally} and \eqref{eq:gluing4}, $x_i=x_j$. Then \eqref{eq:gluingyiyj}, \eqref{eq:gluing1}, and (c) imply $y_i-y_j\in \Ran{A}\cap\Null{\nabla c(\hat x)^\top}=\set{0}$, which proves \eqref{eq:gluingtransitiongoal}.
	\\
	Since $i$ and $j$ were arbitrary, letting $x$ and $y$ denote the common values of the first two components of $\cG^{-j}_{\bm u_j}(\bm 0)$ for each $j\in\set{1,\dotsc,\bk}$. Then $Jy = \cQ(c(\hat x) + \nabla c(\hat x)[x-\hat x]) + \cB + \hat{\cA}\mu$, with $c(\hx)+\nabla c(\hx)[x-\hx]\in \cM_{\bc}$, and $\mu_1,\dotsc,\mu_{\bar k}>0$. By \eqref{eq:plqsdsystem} and \Cref{risdiffstrict}, $\mu(c(\hx)+\nabla c(\hx)[x-\hx],y)_j = \mu_j>0$, with $y\in \ri{\sd h(c(\hx)+\nabla c(\hx)[x-\hx])}$.
\end{proof}
Our final theorem integrates the ideas from \Cref{sec:partialsmoothness} and our work in this section to establish the local quadratic convergence of Newton's method for \ref{theprogram}.
\begin{theorem}
	\label{thm:finalmrthm}
	Let $\bx,\by,\bmu,\bc,\bk$, and $\cQ$ be as in \Cref{assum:mr}, set $\bm\bx_j:=(\bx,\by,\bmu_j)$, and let $\cG^j_{\bm\bx_j}$ be given by \eqref{eq:gphG^j_u}.
	There exists a neighborhood $\cN$ of $(\bx,\by,\bmu)$ on which the conclusions of \Cref{posdeflocally} are satisfied such that if $(x^0,y^0,\mu^0)\in \cN$, then there exists a unique sequence $\set{(x^k,y^k,\mu^k)}_{k\in\bN}$ satisfying the optimality conditions of \ref{introsubprob} for all $k\in\bN$, with
	\begin{enumerate}[label=(\alph*)]
		%\item the conclusions of \Cref{posdeflocally} are satisfied at all $(x,y,\mu)\in\cN$\label{item:localposdefclaim},
		\item $c(x^{k-1}) + \nabla c(x^{k-1})[x^k- x^{k-1}]\in \cM_{\bc}$,
		\item $\mu(c(x^{k-1})+\nabla c(x^{k-1})[x^k-x^{k-1}], y^k)_j>0$ for all $j\in\set{1,\dotsc,\bar k}$,
		\item $y^k\in \ri{\sd h(c(x^{k-1})+\nabla c(x^{k-1})[x^k-x^{k-1}])}$,
		\item $H_{k-1}[x^k-x^{k-1}] + \nabla c(x^{k-1})^\top y^k = 0$\label{item:cricticalsubprobclaim},
		\item $x^k-x^{k-1}$ is a strong local minimizer of the model function  $\phi_{(x^{k-1},y^{k-1})}$,\label{item:slmsubprobclaim} given by \Cref{def:appendixmodelfn}.
	\end{enumerate} 
	Moreover, the sequence $(x^k,y^k)$ converges to $(\bx,\by)$ at a quadratic rate.
\end{theorem}
\begin{proof}
	All claims except \ref{item:slmsubprobclaim} follow from \Cref{thm:Gjquadconv} and \Cref{thm:mrgluing}. By \Cref{lem:modelsosc}, \Cref{lem:modelnonascent}, and \ref{item:cricticalsubprobclaim}, claim \ref{item:slmsubprobclaim} is equivalent to showing
	\begin{equation}\label{eq:slmgoal}
	h''(c(x^{k-1})+\nabla c(x^{k-1})[x^k-x^{k-1}];\nabla c(x^{k-1})\delta) + \delta^\top H_{k-1}\delta > 0\ \forall\,\delta\in \Null{A^\top\nabla c(x^{k-1})}\setminus\set{0}.
	\end{equation}
	Using \eqref{eq:plq2ndnonneg} and partial smoothness,
	\[
	h''(c(x^{k-1})+\nabla c(x^{k-1})[x^k-x^{k-1}];\nabla c(x^{k-1})\delta) = \delta^\top \nabla c(x^{k-1})^\top Q_j \nabla c(x^{k-1}) \delta,\quad \forall\, j\in\cK(\bc),
	\]
	so \eqref{eq:posdefhess2} gives \eqref{eq:slmgoal}.
	%in the appendix by showing the second-order sufficient conditions of \Cref{thm:plqsonsc} are satisfied for the model function
	%	\[
	%		d\mapsto \phi_{(x^{k-1},y^{k-1})}(d):=h(c(x^{k-1})+\nabla c(x^{k-1})d)+\frac{1}{2}d^\top H^{k-1}d,
	%	\]
	%	at $d=x^k-x^{k-1}$.
\end{proof}
\begin{remark}
	The fact that $\set{x^k-x^{k-1}}$ is a strong local minimizer of $\phi_{(x^{k-1},y^{k-1})}$ does not mean that there are not other critical points for the model function outside the neighborhood of interest. It may be that at any iteration the problem \ref{subproblems} does not have a finite optimal value, in particular, should there exist directions of negative curvature orthogonal to the manifold.
\end{remark}
\subsection{Smooth Problems}\label{smoothproblems}
In this section, we make the following assumptions:
\begin{assumptions}
	\label{assum:smoothmr}
	Let $f$ be as in \ref{theprogram} and $(\bx,\by)\in\dom{f}\times\R^m,\ \bc:=c(\bx),\ \bk=|\cK(\bc)|$, where $\cK(\bc)$ are the active indices given in \Cref{def:activeidx}. Let $\cM_{\bc}$ be the active manifold defined in \eqref{eq:manifold}. We assume that
	\begin{enumerate}[label=(\alph*)]
		\item $\dom{h}$ is given by the Rockafellar-Wets PLQ representation of \Cref{cor:plqstructthm},\label{assumitem:smoothmrRWPLQ}
		\item $c$ is $\cC^3$-smooth,
		\item $\cK(\bc)=\set{k_0}$, \label{assumitem:smoothmrK=1}
		\item \label{assumitem:smoothmrSOSC}
		$\bx$ satisfies the second-order sufficient conditions of \Cref{thm:plqsonsc}, 
	\end{enumerate}
\end{assumptions}
\begin{remark}
	Since $\bk=1$, we omit reference to the index $k_0$ for the rest of this section.
\end{remark}
\begin{remark}
	By \ref{assumitem:smoothmrRWPLQ} and \ref{assumitem:smoothmrK=1}, $c(\bx)\in\intr{\dom{h}}$ and $\sd h(\bc)=\set{\by}$. Then, \ref{assumitem:smoothmrSOSC} becomes
	\begin{align*}
	\by = Q\bc + b,\ \nabla c(\bx)^\top \by = 0,\ 
	d^\top\nabla c(\bx)^\top Q\nabla c(\bx)d + d^\top\nabla^2(\by c)(\bx)d >0\ \forall\, d\in\R^n\setminus\set{0},\text{ where }D(\bx)=\R^n.
	\end{align*}
\end{remark}
As in \Cref{posdeflocally}, we have the following stability result.
\begin{lemma}
	\label{lem:smoothposdeflocally}
	Suppose
\(
	%\label{eq:smoothposdefhess}
	d^\top \nabla c(\bx)^\top Q\nabla c(\bx)d + d^\top \nabla^2(\by c)(\bx) d >0\) for all \(d\in \R^n\setminus\set{0}.
\)	Then, there exists a neighborhood $\cN$ of $(\bx,\by)$ such that if $(x,y)\in \cN$ then,
	\begin{equation}
	\label{eq:smoothposdefhess2}
	d^\top \nabla c(x)^\top Q \nabla c(x)d + d^\top\nabla^2(y c)(x)d > 0,\quad\forall\  d\in \R^n\setminus\set{0},
	\end{equation}
	and $c(x)\in\intr{\dom{h}}$.
\end{lemma}
Our local analogue of the KKT mapping \eqref{kktS} is the following.
\begin{definition}
	Define $g:\R^{n+m}\to\R^{n+m}$ by
	\[
	g(x,y) := \begin{pmatrix}
	\nabla c(x)^\top y\\ y - Qc(x) - b
	\end{pmatrix},\quad G := \set{0}^{n+m},
	\]
	and set $\bm\bx := (\bx, \by)$. Then,
	\[
	\nabla g(x,y) = \begin{pmatrix}
	\nabla^2(y c)(x)  & \nabla c(x)^\top \\ -Q\nabla c(x) & I_m
	\end{pmatrix},\quad g(\bx,\by) = \begin{pmatrix}
	0 \\ 0
	\end{pmatrix}
	\]
\end{definition}
\Cref{assum:smoothmr} \ref{assumitem:smoothmrSOSC} implies $\nabla g(\bx,\by)$ is nonsingular. Consequently, and the Newton method \eqref{eq:geqnewton} corresponds to the classical Newton's method for solving the \emph{equation} $g(x,y)=0$. Namely,
\begin{equation}\label{eq:smoothNewton}
\text{Find }(x^{k+1},y^{k+1})\text{ such that }g(x^k,y^k)+\nabla g(x^k,y^k)\begin{pmatrix}
x^{k+1}-x^k\\
y^{k+1}-y^k
\end{pmatrix} = 0.
\end{equation}
The local quadratic convergence of the iteration \eqref{eq:smoothNewton} near $(\bx,\by)$ with $\nabla g(\bx,\by)$ is nonsingular is well-known, with \eqref{eq:smoothNewton} corresponding to the optimality conditions for \ref{introsubprob}. We conclude with the following theorem, which parallels \Cref{thm:finalmrthm}.
\begin{theorem}
	Let $\bx,\by, \bc:=c(\bx)$, and $\cM_{\bc}$ be as in \Cref{assum:smoothmr}. Then, there exists a neighborhood $\cN$ of $(\bx,\by)$ on which the conclusions of \Cref{lem:smoothposdeflocally} are satisfied such that if $(x^0,y^0)\in \cN$, then there exists a unique sequence $\set{(x^k,y^k)}_{k\in\bN}$ satisfying the optimality conditions of \ref{introsubprob} for all $k\in\bN$, with
	\begin{enumerate}[label=(\alph*)]
		%		\item $\nabla g(x^k,y^k)$ is nonsingular,
		\item $c(x^{k-1}) + \nabla c(x^{k-1})[x^k- x^{k-1}]\in \cM_{\bc}$,
		\item $\sd h(c(x^{k-1})+\nabla c(x^{k-1})[x^k-x^{k-1}])=\set{y^k}$,
		\item $H_{k-1}[x^k-x^{k-1}] + \nabla c(x^{k-1})^\top y^k = 0$,
		\item $x^k-x^{k-1}$ is a strong local minimizer of the model function  $\phi_{(x^{k-1},y^{k-1})}$,\label{item:smoothslmsubprobclaim} given by \Cref{def:appendixmodelfn}.
	\end{enumerate} 
	Moreover, the sequence $(x^k,y^k)$ converges to $(\bx,\by)$ at a quadratic rate.
\end{theorem}
\section{Acknowledgement}
The authors thank Asen Dontchev for helpful discussions of the paper \cite{cibulka2016strong}.
\bibliographystyle{abbrv}
\bibliography{MOR_PLQ/MOR_PLQ.bib}
\section{Appendix}
\begin{lemma}
\label{lem:cqcvxsetlinmap}
	Suppose $C\subset \R^m$ is a nonempty, closed, convex set and $A\in \R^{n\times m}$. Consider the following equations:
\begin{align*}
	\tag{a}\Null{A} \cap \ri{C}&=\set{\by},\\
	\tag{b}\Null{A}\cap \parr{C} &= \{0\},\\
	\tag{c}\Null{A} \cap C &= \set{\by}.
\end{align*}
Then 
\(
	(a)\Longrightarrow (b)\Longrightarrow (c).
\)
\end{lemma}
\begin{proof}
	$[(a)\Rightarrow(b)]$
Since $\by\in C$, there exists an integer $n\geq1$ and points $\{y_1,\dotsc,y_n\}\subset \aff{C}$ that
	\(
		\mathrm{span}\set{y_i-\bar{y}}_{i=1}^n = \parr{C}.
	\)
	By convexity and the assumption $\bar{y}\in \ri{C}$, we can further assume $\{y_1,\dotsc,y_n\}\subset \ri{C}$.
	By \cite[Theorem 6.4]{rockafellar2015convex}, there exists $\set{z_1,\dotsc,z_n}\subset\ri{C}$ and $t_i>0$ such that, for all $i\in\set{1,\dotsc,n}$,
	\(
		y_i - \by  = -t_i(z_i-\by).
	\)
	Then, after relabeling, we may suppose $\{y_1,\dotsc,y_n\}\subset\ri{C}$ satisfies
	\begin{equation}
	\label{eq:parrcone}	
		\parr{C} =\bset{\sum_{i=1}^n \mu_i (y_i-\by)}{\mu_i\ge0, i\in\set{1,\dotsc,n}}
	\end{equation}
	Now suppose (b) does not hold. Then, there exists $0\ne z\in \Null{A}\cap \parr{C}$. By \eqref{eq:parrcone},
	\(
		z = \sum_i \mu_i(y_i-\by)
	\)
	with $\mu_i\ge0$ and $\sum_i\mu_i\ne0$. Define $t := \frac{1}{\sum_i \mu_i}$, and for $i\in\set{1,\dotsc,n}$, define $\lambda_i:= t\mu_i$. Then $\lambda_i\ge0$ for all $i\in\set{1,2\dotsc,n}$, with $\sum_{i=1}^n \lambda_i = 1$. Then by \cite[Theorem 6.1]{rockafellar2015convex}
	\(
		\bar{y} + tz = \bar{y} + \sum_i\lambda_i (y_i - \bar{y}) = \sum_i \lambda_i y_i \in \ri{C}.
	\)
	But then $\bar{y}$ and $\bar{y} + tz$ are two points in $\Null{A} \cap \ri{C}$, so (b) must hold.
	\\
	$[(b)\Rightarrow(c)]$
	Suppose $(b)$ and that there exists $y_1,y_2\in \Null{A} \cap C$. Then $y_1-y_2\in \Null{A}\cap \parr{C}=\set{0}$, so $y_1=y_2$. 
\end{proof}
\begin{theorem}[Normals Cones to Sets with Constraint Structure] \cite[Theorem 6.14]{rockafellar_wets_1998}
\label{thm:cq}
	Let
	\(
		C = \bset{x\in X}{F(x)\in Z}
	\)
	for closed convex sets $X\subset\R^n$ and $Z\subset\R^m$ and a $\cC^1$-mapping $F:\R^n\to\R^m$. Suppose $\bx\in C$ satisfies the constraint qualification
	\begin{equation}\label{eq:rtrCQ}
		[y\in \ncone{F(\bx)}{Z},\ -\nabla F(\bx)^\top y \in \ncone{\bx}{X}] \Longleftrightarrow y=(0,\dotsc,0).
	\end{equation}
	Then
	\(
		\ncone{\bx}{C}=\bset{\nabla F(\bx)^\top y + v}{y\in \ncone{F(\bx)}{Z}, v\in \ncone{\bx}{X}}.
	\)
\end{theorem}
\begin{definition}[The model function at $\hx$]
	\label{def:appendixmodelfn}
Let $f$ be as in \ref{theprogram} and $\hx\in\dom{f}$. Suppose $f$ satisfies \eqref{eq:bcq} at $\hx$. Define $\bm u := (\hx,\hy),\ \hat H:=\nabla^2(\hy c)(\hx)$,
\[
	\psi(v,w):=h(v) + w,\text{ and }\Phi_{\bm u}(d) := \begin{pmatrix}
	c(\hx)+\nabla c(\hx)d \\ \frac{1}{2}d^\top \hat H d
	\end{pmatrix}.
\]
Then, for any $(v,w)\in \dom{h}\times\R$ and $(d,s)\in \R^n\times\R$,
\[
\nabla \Phi_{\bm u}(d) = \begin{pmatrix}
\nabla c(\hx) \\ d^\top \hat H
\end{pmatrix},\quad
\psi'((v,w);(d,s)) = h'(v;d)+s,\quad
\psi''((v,w);(d,s))=h''(v;d).
\]
Set
\(
	\phi_{\bm u}(d) := \psi(\Phi_{\bm u}(d)) = h(c(\hx)+\nabla c(\hx)d) + \frac{1}{2}d^\top \hat Hd.
\)
	By \Cref{thm:rockafellarplqcomposite}, $\phi_{\bm u}$ is piecewise linear-quadratic, though not necessarily convex because $\hat H$ may not be positive semi-definite. However, $\phi_{\bm u}$ is convex-composite with $\psi$ piecewise linear-quadratic convex. 
\end{definition}	
The following lemma shows that if $f$ satisfies \eqref{eq:bcq} at $\hx$, then the model function at $\hx$ satisfies its \eqref{eq:bcq} throughout its domain.
\begin{lemma}
	Let $f$ be as in \ref{theprogram}, and suppose $f$ satisfies \eqref{eq:bcq} at $\hx$. Then, $\phi_{\bm u}$ given in \Cref{def:appendixmodelfn} satisfies \eqref{eq:bcq} at all points 
	\(
		\bd\in\dom{\phi_{\bm u}}=\bset{d}{c(\hx)+\nabla c(\hx)d\in\dom{h}}.
	\)
\end{lemma}
\begin{proof}
	Let $\bd\in \bset{d}{c(\hx)+\nabla c(\hx)d\in\dom{h}}$. By definition,
	\begin{align*}
		\Null{\nabla \Phi_{\bm u}(\bd)^\top}\!=\!\Null{\begin{pmatrix}\nabla c(\hx)^\top & \hat H d\end{pmatrix}}\text{ and }
		\ncone{\Phi_{\bm u}(\bd)}{\dom{\psi}}\!=\!\ncone{c(\hx)+\nabla c(\hx)\bd}{\dom{h}}\!\times\!\set{0}.
	\end{align*}
	Suppose $v=(v_1,v_2)\in 	\Null{\nabla \Phi_{\bm u}(\bd)^\top}\cap \ncone{\Phi_{\bm u}(\bd)}{\dom{\psi}} $. Then $v_2=0$, and
	\[
		v_1\in \Null{\nabla c(\hx)^\top}\cap \ncone{c(\hx)+\nabla c(\hx)\bd}{\dom{h}} \subset \Null{\nabla c(\hx)^\top}\cap  \ncone{c(\hx)}{\dom{h}} = \set{0},
	\]
	where the inclusion follows since $\ip{v_1}{\nabla c(\hx)\bd}=0$.
\end{proof}
\begin{lemma}
	\label{lem:modelsosc}	
	Let $\phi_{\bm u}$ be as in \Cref{def:appendixmodelfn}, and suppose $f$ satisfies \eqref{eq:bcq} at $\hx$. Consider the problem
	\begin{mini}
		{d}{\phi_{\bm u}(d)}{\tag{$\cP_{\phi_{\bm u}}$}}{\label{opt:modelsubproblem}}
	\end{mini}
Then, the cone of non-ascent directions $D_{\phi_{\bm u}}(\bd)$ at any $\bd\in\dom{\phi_{\bm u}}$ is given by
\begin{equation}
\label{eq:modelnonascent}
	D_{\phi_{\bm u}}(\bd)=\bset{\delta}{h'(c(\hx)+\nabla c(\hx)\bd;\nabla c(\hx)\delta) + \bd^\top \hat H\delta\le0}.
\end{equation}
Moreover, the second-order necessary and sufficient conditions of \Cref{thm:plqsonsc} applied to $\phi_{\bm u}$ are
\begin{enumerate}
	\item If $\phi_{\bm u}$ has a local minimum at $\bd$, then $0\in \hat H\bd + \nabla c(\hx)^\top \sd h(c(\hx)+\nabla c(\hx)\bd)$ and
	\[
		h''(c(\hx)+\nabla c(\hx)\bd;\nabla c(\hx)\delta) + \delta^\top \hat H \delta \ge 0,
	\]
	for all $\delta\in D_{\phi_{\bm u}}(\bd)$.
	\item If $0\in H\bd + \nabla c(\hx)^\top \sd h(c(\hx)+\nabla c(\hx)\bd)$ and
	\[
		h''(c(\hx)+\nabla c(\hx)\bd;\nabla c(\hx)\delta) + \delta^\top \hat H \delta > 0,	
	\]
	for all $\delta\in D_{\phi_{\bm u}}(\bd)\setminus\set{0}$, then $\bd$ is a strong local minimizer of $\phi_{\bm u}$.
\end{enumerate}
\end{lemma}
\begin{proof}
	Since \eqref{eq:bcq} is satisfied at all points $d\in\dom{\phi_{\bm u}}$, the chain rule of \Cref{thm:cvxcompfonc} gives
	\[
		\sd \phi_{\bm u}(d) = \hat Hd + \nabla c(\hx)^\top\sd h(c(\hx)+\nabla c(\hx)d),
	\]
	\[
		\dif \phi_{\bm u}(d)(\delta) =  h'(c(\hx)+\nabla c(\hx)d;\nabla c(\hx)\delta) + d^\top \hat H \delta,
	\]
	which is \eqref{eq:modelnonascent}. The set of Lagrange multipliers for $\phi_{\bm u}$ becomes
	\begin{equation}
	\label{eq:modelmultiplierset}
	\begin{aligned}
		M_{\phi_{\bm u}}(d) &:= \Null{\nabla \Phi_{\bm u}(d)^\top}\cap \sd \psi(\Phi_{\bm u}(d))\\&=\Null{\begin{pmatrix}\nabla c(\hx)^\top & \hat H d\end{pmatrix}}\cap(\sd h(c(\hx)+ \nabla c(\hx)d)\times\set{1}),
	\end{aligned}
	\end{equation}
	so that
	\(
		\begin{pmatrix}
		y_1&y_2
		\end{pmatrix}\in M_{\phi_{\bm u}}(d) \Longleftrightarrow \begin{cases}
		\hat Hd + \nabla c(\hx)^\top y_1,\ 
		y_1\in\sd h(c(\hx)+\nabla c(\hx)d),\ 
		y_2=1.
		\end{cases}
	\)
	The Lagrangian \cite{burke1987second} is 
	\(
		L(d,y) := \ip{y}{\Phi_{\bm u}(d)} - \psi^\star(y),\quad y=(y_1,y_2)\in \R^m\times\R,
	\) with $\nabla^2(y\Phi_{\bm u})(d)=y_2\hat H$. Then, from \Cref{thm:plqsonsc}, for any $\delta\in\R^n$,
	{\small\[
		\psi''(\Phi_{\bm u}(d);\nabla \Phi_{\bm u}(d)\delta) + \max\bset{\ip{\delta}{\nabla^2(y\Phi_{\bm u})(d)\delta}}{y\in M_{\phi_{\bm u}}(d)}=
		h''(c(\hx)+\nabla c(\hx)\bd;\nabla c(\hx)\delta) + \delta^\top \hat H \delta.
	\]}
\end{proof}
The following lemma relates an active manifold at a solution to \ref{theprogram} to the directions of non-ascent for the model function \Cref{def:appendixmodelfn}. It is an immediate consequence of \Cref{thm:ps}, \Cref{lem:D(x)andA^Tc'(x)}, and \eqref{eq:modelnonascent}, and the proof is identical to \Cref{lem:D(x)andA^Tc'(x)}.
\begin{lemma}[Model non-ascent directions]
	\label{lem:modelnonascent}
	Let $f$ be as in \ref{theprogram}, $\bx\in\dom{f},\ \bc:=c(\bx),\ \bk=|\cK(\bc)|$, where $\cK(\bc)$ are the active indices given in \Cref{def:activeidx}. Let $(\hx,\hy)$ and $\phi_{\bm u}$ be as in \Cref{def:appendixmodelfn}, and let the active manifold $\cM_{\bc}$ be as in \eqref{eq:manifold}, with $\dom{h}$ given by the Rockafellar-Wets PLQ representation theorem. Suppose $0 = \hat H \bd + \nabla c(\hx)^\top \by,\ c(\hx)+ \nabla c(\hx)\bd\in\cM_{\bc}$, and $\by\in\ri{\sd h(c(\hx)+\nabla c(\hx)\bd)}$. Then, $\phi_{\bm u}$ satisfies \eqref{eq:strictcriticality} at $\bd$ for $(\by,1)$, and
		
	if $\bk\ge2$, then, in the notation of \Cref{def:AmatrixPs},
	%If $c(\hx)+ \nabla c(\hx)\bd\in\cM_{\bc}$, and $\by\in\ri{\sd h(c(\hx)+\nabla c(\hx)\bd)}$. 
	%Then, $\phi_{\bm u}$ satisfies \eqref{eq:strictcriticality} at $\bd$ for $(\by,1)$ and 
	$D_{\phi_{\bm u}}(\bd) = \Null{A^\top \nabla c(\hx)}$.
	
	if $\bk=1$,
	%If $c(\hx)+ \nabla c(\hx)\bd\in\cM_{\bc}$, and $\set{\by}=\sd h(c(\hx)+\nabla c(\hx)\bd)$. 
	then, 
	%$\phi_{\bm u}$ satisfies \eqref{eq:strictcriticality} at $\bd$ for $(\by,1)$ and 
	$D_{\phi_{\bm u}}(\bd) = \R^n$.
\end{lemma}
\end{document}